\documentclass[10pt]{article}

\usepackage{hyperref,amsmath,amsthm,amssymb}

\def\gcal{\mathcal{G}}

\def\pcal{\mathcal{P}}
\def\lcal{\mathcal{L}}
\def\cal{\mathcal{H}}
\def\ucal{\mathcal{U}}

\def\scal{\mathcal{S}} 

\def\bbz{\mathbb{Z}}
\def\bbq{\mathbb{Q}}
\def\bbf{\mathbb{F}}
\def\bbr{\mathbb{R}}
\def\bba{\mathbb{A}}

\def\bbh{\mathbb{H}}

\def\bbg{\mathbb{G}}

\def\bbu{\mathbb{U}}
\def\bbl{\mathbb{L}}

\def\gfr{\mathfrak{g}}
\def\hfr{\mathfrak{h}}

\def\afr{\mathfrak{a}}

\def\pfr{\mathfrak{p}}
\def\f{\mathfrak{f}}
\def\F{\mathfrak{F}}
\def\L{\mathfrak{L}}
\def\X{\mathbb{X}}

\def\lfr{\mathfrak{l}}
\def\ufr{\mathfrak{u}}

\def\vare{\varepsilon}

\def\Aff{{\rm Aff}}
\def\Aut{{\rm Aut}}

\def\Ad{{\rm Ad}}

\def\Lie{{\rm Lie}}

\def\GL{{\rm GL}}
\def\SL{{\rm SL}}

\def\spec{{\rm Spec}}
\def\adj{{\rm adj}}

\def\h{\hspace{1mm}}

\def\lhde{\trianglelefteq}

\def\a{\alpha}
\def\b{\beta}
\def\d{\delta}

\def\e{\varepsilon}
\def\f{\varphi}
\def\g{\gamma}
\def\Ga{\Gamma}

\def\l{\lambda}

\def\o{\omega}

\def\L{{\mathbb L}}
\def\G{{\mathbb G}}
\def\H{{\mathbb H}}
\def\U{{\mathbb U}}
\def\Q{{\mathbb Q}}

\def\Z{{\mathbb Z}}

\def\F{{\mathbb F}}
\def\P{{\mathbb P}}
\def\V{{\mathbb V}}
\def\W{{\mathbb W}}

\def\wh{\widehat}
\def\wt{\widetilde}

\def\GG{\mathcal{G}}
\def\HH{\mathcal{H}}
\def\LL{\mathcal{L}}

\def\KK{\mathcal{K}}

\def\sm{\backslash}

\def\pr{{\rm pr}}
\def\Ker{{\rm Ker}}
\def\Tr{{\rm Tr}}
\def\Mat{{\rm Mat}}
\def\Im{{\rm Im}}

\def\be{\begin{equation}}
\def\ee{\end{equation}}
\def\bea{\begin{eqnarray}}
\def\eea{\end{eqnarray}}
\def\bean{\begin{eqnarray*}}
\def\eean{\end{eqnarray*}}

\newtheorem{lem}{Lemma}
\newtheorem{thm}[lem]{Theorem}
\newtheorem*{thma}{Theorem A}

\newtheorem*{thmc}{Theorem C}

\newtheorem*{lemd}{Lemma D}

\newtheorem*{prpb}{Proposition B}

\newtheorem{prp}[lem]{Proposition}
\newtheorem{prop}[lem]{Proposition}
\newtheorem{cor}[lem]{Corollary}

\theoremstyle{definition}
\newtheorem{examp}[lem]{Example}
\newtheorem{question}[lem]{Question}
\newtheorem{rmk}[lem]{Remark}
\newtheorem{dfn}[lem]{Definition}

\def\la{\lesssim}

\title{Expansion in perfect groups}

\begin{document}
\author{Alireza Salehi Golsefidy\footnote{A. S-G. was partially supported by the
NSF grant
DMS-1001598.} and P\'eter P. Varj\'u\footnote{P.P. V. was partially
supported by the NSF grant DMS-0835373.}}
\maketitle

\begin{abstract}
Let $\Ga$ be a subgroup of $\GL_d(\bbz[1/q_0])$ generated by a finite symmetric set S.
For an integer $q$, denote by $\pi_q$ the projection
map $\Z[1/q_0]\to\Z[1/q_0]/q\Z[1/q_0]$.
We prove that the Cayley graphs of $\pi_q(\Ga)$ with respect to the generating sets
$\pi_q(S)$ form a family of expanders when $q$ ranges over square-free integers
with large prime divisors if and only if the connected component of
the Zariski-closure of $\Ga$ is perfect, i.e. it has no nontrivial Abelian quotients.
\end{abstract}

\section{Introduction}
\label{sc_intro}
Let $\GG$ be a graph, and for a set of vertices
$X\subset V(\GG)$, denote by $\partial X$ the set of
edges that connect a vertex in $X$ to one in $V(\GG)\sm X$.
Define
\[
c(\GG)=\min_{X\subset V(\GG),\;\;|X|\le|V(\GG)|/2}
\frac{|\partial X|}{|X|},
\]
where $|X|$ denotes the cardinality of the set $X$.
A family of graphs is called a family of expanders, if
$c(\GG)$ is bounded away from zero for
graphs $\GG$ that belong to the family.
Expanders have a wide range of applications in computer
science (see e.g. Hoory, Linial and Widgerson \cite{HLW} for a
survey on expanders) and
recently they found remarkable applications in pure mathematics
as well.

Let $S$ be a symmetric (i.e. closed for taking inverses) subset
of $\GL_d(\bbq)$ and let $\Gamma$  be the group generated by $S$.
For any positive integer $q$, let $\pi_q:\bbz\rightarrow \bbz/q\bbz$
be the residue map.
If the prime factors of $q$ are large, $\pi_q$ induces a
homomorphism from $\Gamma$ to $\GL_d(\bbz/q\bbz)$.
We denote this and all the similar maps by $\pi_q$ also.
In this article, we give a necessary and sufficient condition
under which the family of Cayley graphs $\gcal(\pi_q(\Gamma),\pi_q(S))$
form expanders as $q$ runs through square-free integers
with large prime factors.
Let us recall that if $S \subset G$ is a symmetric set of generators, then
the Cayley graph $\GG(G,S)$ of $G$ with respect to the
generating set $S$ is defined to be the graph whose vertex set
is $G$, and two vertices $x,y\in G$ are connected
exactly if $y\in Sx$.

\begin{thm}
\label{th_gengrp}
Let $\Gamma\subseteq \GL_d\left(\bbz[1/{q_0}]\right)$ be the group generated by a symmetric set $S$. Then $\gcal(\pi_q(\Gamma),\pi_q(S))$ form a family of expanders when $q$ ranges over square-free integers coprime to $q_0$ if and only if the connected component of the Zariski-closure of $\Gamma$ is perfect.
\end{thm} 

\subsection{Motivation and related results}

A result of the type of Theorem \ref{th_gengrp} was proved by Bourgain, Gamburd and Sarnak~\cite{BGS}.
They proved that such Cayley graphs form expanders if
$\Gamma\subseteq \SL_2(\Z)$ is Zariski-dense.
Their main motivation was to formulate and prove
an affine sieve theorem.
Moreover, they proved the affine sieve theorem for groups more
general than $\SL_2$ provided a conjecture of Lubotzky
holds (see \cite[Conjecture {1.5}]{BGS} and \cite{Lub0})
which is a special case
of our Theorem \ref{th_gengrp}.
Following the ideas in \cite{BGS} and using Theorem~\ref{th_gengrp},
in a forthcoming paper, Salehi Golsefidy and Sarnak~\cite{SS}
get a similar affine sieve theorem whenever the character group of the connected component of the Zariski-closure of $\Gamma$ is trivial.

Following \cite{BGS}, now there is a rich literature of sieving
applications of expander graphs not limited to number theory.
We refer to the recent surveys of Kowalski \cite{Kow} and Lubotzky
\cite{Lub} for more details on these developments.
Besides sieving, results similar to the above theorem are useful
in studying covers of hyperbolic 3-manifolds, see the paper
of Long, Lubotzky and Reid \cite{LLR} following the work of
Lackenby \cite{Lac}.

The question about the expanding properties of mod $q$
quotients was first studied only for ``thick" groups,
namely lattices in semisimple Lie groups.
The first results used the representation theory of the underlying
Lie group; property (T) in Kazhdan \cite{Kaz} and Margulis \cite{Mar}
and automorphic forms in e.g. \cite{Sel}, \cite{BS}, \cite{KS2} and
\cite{Clo}.
Later Sarnak and Xue \cite{SX} developed a more elementary method.
Kelmer and Silberman \cite{KS} combined this method with
recent advances on automorphic forms to obtain a very general result
on arbitrary arithmetic lattices.
The advantage of these results over our method is that they give
explicit and very good bounds.

Lubotzky was the first person who asked this question
for a ``thin" group in his famous 1-2-3 conjecture
(see \cite{Lub0}).\footnote{He asked this question for
$S=\left\{\left[\begin{array}{cc} 1&\pm 3\\0&1\end{array}\right],
\left[\begin{array}{cc} 1&0\\\pm 3&1\end{array}\right]\right\}$.}
Shalom \cite{Sh1}, \cite{Sh2} obtained the first results which establish the
expander property for quotients of certain non-lattices.
A few years later Gamburd \cite{Gam} showed that quotients of $\Ga<\SL_2(\Z)$
are expanders if the Hausdorff dimension of the
limit set of $\Ga$ is larger than 5/6.
The first paper which achieved a result which depend merely on the
Zariski-closure of $\Ga$ was obtained by Bourgain and Gamburd \cite{BG1}.
Their assumptions were that the Zariski-closure of $\Ga$ is
$\mathbb{SL}_2$, and
the modulus $q$ is prime.
In the past four years, several articles appeared which extended that
result, see \cite{BG1}--\cite{BGS}, and \cite{Var}.
However, there are still interesting questions to explore. For instance,

\begin{question}\label{q:AnyModulus}
Does the family of Cayley graphs $\gcal(\pi_q(\Gamma),\pi_q(S))$
form expanders as $q$ runs through {\it any}
positive integer with large prime factors if the connected
component of the Zariski-closure of $\Gamma$ is perfect?
\end{question}

Bourgain and the second author \cite{BV} give an affirmative answer to this question
when the Zariski-closure of $\Gamma$ is $\mathbb{SL}_d$.

\begin{question}
If $\Gamma\subseteq \GL_d(\bbf_p(t))$ is generated by a symmetric set $S$,
then what is the necessary and sufficient condition such that
$\gcal(\pi_{\pfr}(\Gamma),\pi_{\pfr}(S))$
form expanders as $\pfr$ runs through square-free polynomials
with large degree prime factors?
\end{question}

Moreover, one might hope that the answer is positive even to
the following very general question that was communicated to us
by Alex Lubotzky:

\begin{question}[Lubotzky]\label{q:Lub}
Let $\Gamma\subseteq \GL_n(A)$ is a finitely generated subgroup,
where $A$ is an integral domain which is generated by the traces of
the elements of $\Gamma$.
Is it true that if the Zariski-closure of $\Gamma$ is semisimple, then
the Cayley graphs
$\gcal(\pi_{\afr}(\Gamma),\pi_{\afr}(S))$ form a family of expanders as $\afr$
ranges through finite index ideals of $A$?
\end{question}

We also mention that there are 
studies devoted to the problem of expansion with respect to random
generators, see \cite{BGT2} and \cite{BGT3}, and
in the work of Breuillard and Gamburd \cite{BrG} it is proved that
except maybe for a set of primes $p$ of zero density, $\SL_2(\F_p)$
are expanders with respect to any generators.

\subsection{Groups defined over number fields}

Let $k$ be a number field and $\Gamma$ be a finitely generated
subgroup of $\GL_n(k)$.
Let us also remark that Theorem~\ref{th_gengrp}
tells us under what condition the Cayley graphs 
of the square free quotients of $\Gamma$ form expanders.
To be precise, by the means of restriction of scalars,  
we can view $\Gamma$ as a subgroup $\Gamma\rq{}$
of $R_{k/\bbq}(\GL_d)(\bbq)\subseteq \GL_{rd}(\bbq)$,
where $r=\dim_{\bbq}k$.
Now it is easy to see that
$\gcal(\pi_{\mathfrak{q}}(\Gamma),\pi_{\mathfrak{q}}(S))$
form expanders as $\mathfrak{q}$ runs through square free 
ideals of $\mathcal{O}_k$ (the ring of integers in $k$)
with large prime factors if and only if 
$\gcal(\pi_q(\Gamma\rq{}),\pi_q(S\rq{}))$
form expanders as $q$ runs through square free 
integers with large prime factors.
By Theorem~\ref{th_gengrp}, we know the necessary and 
sufficient condition under which the latter holds.
In the following example, we present a finitely 
generated subgroup $\Gamma$ of $\GL_d(\bbq[i])$
whose Zariski-closure is Zariski connected and 
perfect but $\gcal(\pi_{\pfr}(\Gamma),
\pi_{\pfr}(S))$ do not form expanders as $\pfr$ runs 
through prime ideals of $\bbz[i]$.
It shows that in general it is necessary to view $\Gamma$ as a subgroup 
of $\GL_{rd}(\bbq)$ and then look at its Zariski-closure.
 
\begin{examp}\label{ex:RestrictionScalar}
Let $h$ be a non-degenerate symplectic form on $V=\bbz^m$.
Let $\bbh$ be the Heisenberg 
group associated with $h$.
To be precise, $\bbh(\bbz)$ is the set $V\times \bbz$ endowed with 
the group law
\[
(v,t)\cdot(v\rq{},t\rq{}):=(v+v\rq{},t+t\rq{}+h(v,v\rq{})).
\]
From the definition it is clear that $\bbh$ is a central
extension of the group scheme $\mathbb{V}$ associated with $V$.
The action of the symplectic group ${\rm Sp}_{h,V}$ on
$\mathbb{V}$ can be naturally extended to an action on
$\bbh$ by acting trivially on the center.
Now let  $\bbg={\rm Sp}_{h,V}\ltimes \bbh$ and 
\[
\Gamma:=\{(\gamma,(v,t))\in \bbg(\bbz[i])|\h
\gamma\in {\rm Sp}_{h,V}(\bbz),\h v\in V,\h t\in  \bbz[i]\}. 
\]
It is easy to see $\bbg_{\bbq[i]}:=\bbg\times_{\spec(\bbz)}\spec(\bbq[i])$
is a perfect Zariski-connected $\bbq[i]$-group and
$\Gamma$ is a finitely generated Zariski-dense subgroup
of $\bbg_{\bbq[i]}$. On the other hand, the Zariski-closure of $\Gamma'$ in $R_{\bbq[i]/\bbq}(\bbg)$ is isomorphic to $\bbg\times\bbg_a$, which is not perfect. Thus the Cayley graphs $\gcal(\pi_{\pfr}(\Gamma),\pi_{\pfr}(S))$ do not form expanders as $\pfr$ runs through primes in $\bbz[i]$.
\end{examp}

As we have seen in Example~\ref{ex:RestrictionScalar},
in general, the connected component of the Zariski-closure
$\bbg$ of $\Gamma$ in $\GL_d(k)$ might be perfect but the
connected component of its Zariski-closure $\bbg\rq{}$
in $R_{k/\bbq}(\bbg)(\bbq)\subseteq \GL_{rd}(\bbq)$,
where $r=\dim_{\bbq}k$, might be not perfect.
However it is easy to see that if $\bbg$ is 
semisimple, so is $\bbg\rq{}$.
Hence we get the following corollary:

\begin{cor}\label{c:Semisimple}
Let $\Gamma\subseteq \GL_d(k)$ be the group generated
by a symmetric set $S$, where $k$ is a number field.
If the Zariski-closure of $\Gamma$ is semisimple,
then $\gcal(\pi_{\mathfrak{q}}(\Gamma),\pi_{\mathfrak{q}}(S))$
form a family of expanders when $\mathfrak{q}$
ranges over square-free ideals of the ring of integers
$\mathcal{O}_k$ in $k$ with large prime factors. 
\end{cor}

\subsection{Outline of the argument}

Similarly to most of the previous works done on this
problem~\cite{BG1}--\cite{BGS} and \cite{Var},
first we prove {\it escape from proper subgroups} and then 
we show the occurrence of the {\it flattening phenomenon}.

\begin{prp}[Escape from subgroups]
\label{pr_escp}
Let $\G$ be a Zariski-connected, perfect algebraic group defined over $\Q$.
Let $S\subset \G(\Q)$
be finite and symmetric.
Assume that $S$ generates a subgroup $\Ga<\G(\Q)$
which is Zariski dense in $\G$.

Then there is a constant $\d$ depending only on $S$,
and there is a symmetric set $S'\subset \Ga$
such that the following holds.
For any square-free integer $q$ which is relatively
prime to the denominators of the entries of $S$, and for any proper
subgroup $H<\pi_q(\Ga)$ and for any even integer $l\ge\log q$,
we have
\[
\pi_q[\chi_{S'}^{(l)}](H)\ll[\pi_q(\Ga):H]^{-\d}.
\]
\end{prp}

The notation used in this and the next proposition is explained in detail
in Section \ref{sc_not}.
Now we only mention that $\chi_{S'}$ is the normalized counting (probability) measure on $S'$ and $\chi_{S'}^{(l)}$ is the $l$-fold convolution of $\chi_{S'}$ with itself.

\begin{prp}[$l^2$-flattening] 
\label{pr_flatening}
Let $\G$ be a Zariski-connected, perfect algebraic group defined over $\Q$.
Let $\Ga<\G(\Q)$ be a finitely generated Zariski-dense subgroup.
Then for any $\e>0$, there is some $\d>0$ depending only on
$\e$ and $\G$
such that the following holds.
Let $q$ be a square-free integer which is relatively prime
to the entries of the elements of $\Ga$ and let
$\mu$ and $\nu$ be probability measures on $\pi_q(\Ga)$ such that
$\mu$ satisfies
\[
\|\mu\|_2>|\pi_q(\Ga)|^{-1/2+\e}\quad{\rm and}\quad
\mu(gH)<[\pi_q(\Ga):H]^{-\e}
\]
for any $g\in \pi_q(\Ga)$ and for any proper subgroup $H<\pi_q(\Ga)$.
Then
\[
\|\mu*\nu\|_2<\|\mu\|_2^{1/2+\d}\|\nu\|_2^{1/2}.
\]
\end{prp}

We deduce Theorem~\ref{th_gengrp} from
the above propositions.
The method to prove spectral gap using analogues of these
propositions was discovered by Bourgain and Gamburd \cite{BG1}
building on ideas that go back to
Sarnak and Xue \cite{SX}.
Very briefly it goes as follows:

By Proposition \ref{pr_escp}, we can bound
$\pi_q[\chi_{S'}^{(l)}](gH)$ which is
the probability that the random walk after $l\approx \log q$
steps is in a coset of a proper subgroup $H$.
In particular, taking $H=\{1\}$, we get
$\|\chi_{S'}^{(l)}\|_{2}\le|\pi_q(\Gamma)|^{\delta}$.
Now we can apply Proposition \ref{pr_flatening} and iterate it
to get improved bounds.
Finally the representation theory of $\G(\Z/q\Z)$ gives
a lower bound for the multiplicity of the eigenvalues of the
adjacency matrix of the Cayley graph $\GG(\pi_q(\Gamma),\pi_q(S))$.
Then we can use a trace formula to deduce an upper bound
for the eigenvalues.

The papers \cite{BG1}--\cite{BGS}, \cite{Var} and the current
work are all based on the above strategy.
The difference between the proofs
is in the way the analogues of these two propositions are proved.

We divided the proof of Proposition~\ref{pr_escp} into two parts.
First, mainly using Nori's result, we lift up the problem to $\Gamma$,
and show that ``small" lifts of a certain large subgroup of $H$ is
inside a proper algebraic subgroup $\H$.
(The idea of using Nori's theorem in this context is not new,
it goes back
to the paper of Bourgain and Gamburd \cite{BG3}.)
Then we give a geometric description for being in a proper
algebraic subgroup in the spirit of Chevalley's theorem.
 To this end, we construct finitely
many irreducible representations $\rho_i$ of the semisimple
quotient of $\bbg$. Then for any $i$ we also give an algebraic
family $\{\phi_{i,v}\}_v$ of affine transformation lifts of
$\rho_i$ to $\bbg$. And we show that a proper algebraic
subgroup $\bbh$ either fixes a line via $\rho_i$ or fixes
a point via $\phi_{i,v}$ for some $i$ and $v$. 
In the second step, using some ideas of Tits, we construct
certain ``ping-pong players", and show that,
in the process of the random walk with respect
to this set of generators, the probability of fixing either a particular line or a point in these finitely many algebraic families of affine representations is exponentially small.

In order to prove Proposition~\ref{pr_flatening}, first we
prove a triple product theorem similar to
Helfgott's result~\cite{Hel}, \cite{Hel2}.
I.e. we show that if $A\subset\G(\Z/q\Z)$ is suitably distributed
among proper subgroup cosets (to be made precise, see
Proposition \ref{pr_product})
then $|A.A.A|\ge|A|^{1+\d}$.
Then the proposition can be deduced from the Balog-Szemer\'edi-Gowers
Theorem just as in \cite{BG1}.

We comment on the new ideas of the current work compared
to the previous results, especially to \cite{Var},
where Theorem \ref{th_gengrp} was proved for $\G=R_{k/\Q}(\SL_d)$.
We also indicate which of these ideas are relevant also
when the Zariski closure $\G$ of $\Ga$
is semisimple, since this case is of special interest.

Compared to \cite{Var}, the "ping-pong argument" used in
the current work is more flexible.
In \cite{Var}, the argument applies only
for representations  that are both proximal and irreducible.
Whereas in the current paper we give a more self-contained
argument that needs only irreducibility.
This is significant because even in the semisimple case,
it could be difficult to construct suitable representations
with both properties.
Moreover, we do not rely on the result of Goldsheid
and Margulis on proximality of Zariski dense subgroups.
This allows us to work both 
in the Archimedean and the non-Archimedean setting which is
needed to prove Theorem \ref{th_gengrp} for $S$-integers.
(The theorem of Goldsheid and Margulis does not hold over
p-adic fields.)

When $\G$ is not semi-simple further new ideas are needed.
The unipotent radical is in the kernel
of any irreducible representation.
In order to detect proper subgroups which
surject onto the semisimple factor of $\G$,
we introduce algebraic families of affine representations.

We also need to use more complicated constructions when
we use Nori's result to lift subgroups of finite groups
to algebraic subgroups.
On the other hand we eliminate the use of  the quantitative
Nullstellensatz which was a tool in \cite{Var}.

It was proved in \cite{Var}
(see Proposition B in Section
\ref{sc_product})
that if the triple product theorem holds for a family of
simple groups
(which also satisfy some additional, more technical properties), then
the triple product theorem is also true for their direct product.
The proof of this result in \cite{Var} is closely related to the
proof of the square-free sum-product theorem proved in \cite{BGS}.
The triple product theorem for finite simple groups of Lie type of
bounded rank was achieved in a recent breakthrough of
Breuillard, Green, Tao \cite{BGT} and
of Pyber, Szab\'o \cite{PS} independently.
(See Theorem C in Section \ref{sc_product}.)
These results are used as black boxes in our paper.
When $\G$ is semisimple, then Proposition~\ref{pr_flatening} almost
immediately follows from these results.
The new contribution of the current work in the proof
of Proposition~\ref{pr_flatening} is when $\G$ is not semisimple. 
To this end, we have to deal with certain
semidirect products and one can see some similarities
with the work of Alon, Lubotzky and Wigderson~\cite{ALW}.

We note that the all the constants appearing in the paper are effective.
However, an explicit computation would be tedious especially since some of our references
are non-explicit, too.
In particular the paper of Nori \cite{Nor} uses non-effective techniques, but it can be made effective
using some results of computational algebraic geometry.
We discuss this briefly in the Appendix.

All of our arguments are constructive, and the constants could be computed in a straightforward
way, except for some of the proofs in Section \ref{sc_pingpong}.
At those places, we prove the existence of certain objects by nonconstructive
means.
However, these objects can be found by an algorithm simply by checking
countable many possibilities.
The existence of the object implies that the algorithm terminates in finite
time.

For example Proposition
\ref{pr_generators} claims the existence of a finite subset of $\Ga$
and certain subsets of vector spaces with certain properties.
It is easy to see that we can choose those sets to be bounded by rational hyperplanes, so the data
whose existence is claimed in the proposition can be found within a countable set.
Since it is a finite computation to check the required properties, one can always find a suitable subset
of $\Ga$ and the accompanied data by finite computation.

The other place is the proof of Proposition \ref{pr_pingpong}, where we show that the intersection of a collection
of sets parametrized by an integer $k$ is empty for some $k$.
Although the proof does not give a clue how large $k$ needs to be, but we can always compute it
by computing the intersection of the sets for every $k$ until it becomes empty.

The organization of the paper is as follows.
In Section \ref{sc_not} we introduce some notation.
Section \ref{sc_escp} is devoted to the proof of Proposition
\ref{pr_escp}.
In Section \ref{sc_product} we prove Proposition
\ref{pr_flatening}.
In Section \ref{sc_proof}, we finish the proof of Theorem~\ref{th_gengrp}.
Finally in the appendix  the effectiveness of Nori's results~\cite{Nor}
is showed.

\bigskip

\noindent{\bf Acknowledgment.}
We would like to thank Peter Sarnak and Jean Bourgain for their interest
and many insightful conversations.
We are very grateful to Brian Conrad for his help in the
proof of Theorem~\ref{t:EGA}. We are also in debt to Alex Lubotzky for his interest and 
permission to include Question~\ref{q:Lub}.
We also wish to thank the referee for her or his suggestions 
and careful reading of the paper.
 
\section{Notations}
\label{sc_not}

We introduce some notation that will be used throughout the
paper.
We use Vinogradov's notation $x\ll y$ as a shorthand for $|x|<C y$
with some constant $C$.
Let $G$ be a group.
The unit element of any multiplicatively 
written group is denoted by 1.
For given subsets $A$ and $B$, we denote their product-set by
\[
A.B=\{gh\,|\, g\in A,h\in B\},
\]
while the $k$-fold iterated product-set of $A$ is denoted by
$\prod_k A$.
We write $\wt A$ for the set of inverses of all elements of $A$.
We say that $A$ is symmetric if $A=\wt A$.
The number of elements of a set $A$ is denoted by $|A|$.
The index of a subgroup $H$ of $G$ is denoted by
$[G:H]$ and we write $H_1\la_L H_2$ if $[H_1:H_1\cap H_2]\le L$
for some subgroups $H_1,H_2<G$.
We denote the center of $G$ by $Z(G)$.
If $\rho$ is a representation of $G$, then we denote the underlying vector space by $W_{\rho}$
and we denote by $(W_\rho)^G$ the set of points fixed by all elements of $G$.
Occasionally (especially when a ring structure is present) we
write groups additively, then we write
\[
A+B=\{g+h\,|\, g\in A,h\in B\}
\]
for the sum-set of $A$ and $B$, $\sum_k A$ for the $k$-fold
iterated sum-set of $A$ and 0 for the unit element.

If $\mu$ and $\nu$ are complex valued functions on $G$, we define
their convolution by
\[
(\mu*\nu)(g)=\sum_{h\in G}\mu(gh^{-1})\nu(h),
\]
and we define $\wt \mu$ by the formula
\[
\wt\mu(g)=\mu(g^{-1}).
\]
We write $\mu^{(k)}$ for the $k$-fold convolution of $\mu$ with itself.
As measures and functions are essentially the same on discrete
sets, we use these notions interchangeably, we will also
use the notation
\[
\mu(A)=\sum_{g\in A} \mu(g).
\]
A probability measure is a nonnegative measure with total mass 1.
Finally, the normalized counting measure on a finite
set $A$ is the probability
measure
\[
\chi_A(B)=\frac{|A\cap B|}{|A|}.
\]

\section{Escape from subgroups}
\label{sc_escp}

In this section we prove Proposition \ref{pr_escp}.
Some ideas are taken from \cite{Var} but there are substantial
new difficulties especially due to the lack of proximality 
of the adjoint representation and because
we also consider groups that are not semisimple.

 We begin this section by recalling some results from the literature
on the subgroup structure of $\pi_q(\Ga)$.
Then in Section \ref{sc_descr}, in order to solve the problem of escaping from proper subgroups of $\pi_q(\Gamma)$, we lift it up to a problem on escaping from certain proper subgroups of $\Gamma$. For that purpose, we consider ``small" lifts of elements of $H$ in $\Gamma$; namely, for a square-free integer $q$ and a subgroup $H<G_q$, we write
\[
\lcal_\d(H):=\{h\in \Ga|\pi_q(h)\in H,\;
\|h\|_{\scal}<[G_q:H]^\d\},
\]
where $\|h\|_{\scal}=\max_{p\in\scal\cup\{\infty\}} \|h\|_p$ and $\|h\|_p$ is the operator norm on $\bbq_p^d$.
In the next section, we give a geometric description of
the set $\LL_\d(H)$ in terms of its action in an irreducible
representation.
Then in Section \ref{sc_pingpong}, we present an argument
to show that only a small fraction of the elements of $\Ga$
satisfy this geometric property.
Finally, we combine these two results to get Proposition \ref{pr_escp}. 

Let $\bbg\subseteq \mathbb{GL}_d$ be a Zariski-connected $\bbq$-group. Then it is well-known that its unipotent radical $\bbu$ is also defined over $\bbq$ (e.g. see \cite[Proposition 14.4.5]{Spr}) and it has a Levi subgroup defined over $\bbq$, i.e. a reductive subgroup $\bbl$ defined over $\bbq$ such that $\bbg$ is $\bbq$-isomorphic to $\bbl\ltimes \bbu$. If $\bbg$ is a perfect group, i.e. $\bbg=[\bbg,\bbg]$, then clearly $\bbl$ is a semisimple group. We say that $\bbg$ is simply-connected if $\bbl$ is simply-connected. If $\bbl$ is a simply-connected $\bbq$-group, then we can write $\bbl$ as product of absolutely almost simple groups. The absolute Galois group permutes these factors. So there are number fields $\kappa_i$ and absolutely almost simple $\kappa_i$-groups $\bbl_i$ such that
\[
\bbl\simeq \prod_{i=1}^m R_{\kappa_i/\bbq}(\bbl_i)
\]
as $\bbq$-groups. By a result of Bruhat-Tits~\cite[Section 3.9]{Tit}, for large enough $p$, $\bbl(\bbz_p)$ is a hyper-special parahoric subgroup and so $\bbl(\bbf_p)$ is a product of quasi-simple groups. We also have that $\bbu(\bbf_p)$ is a finite $p$-group, and, for large enough $p$, $\bbg(\bbf_p)\simeq \bbl(\bbf_p)\ltimes \bbu(\bbf_p)$ and is a perfect group. Again as part of Bruhat-Tits theory, we know that $\bbl(\bbf_p)$ is generated by order $p$ elements. (To be precise, we know that, for large enough $p$, the special fiber of $\bbl$ over $\bbf_p$ is a connected, simply connected, semisimple  $\bbf_p$-group.) Thus, for large enough $p$, $\bbg(\bbf_p)=\bbg(\bbf_p)^+$, where $G^+$ is the subgroup generated by $p$-elements, for any subgroup $G$ of $\GL_d(\bbf_p)$. As part of Nori's Strong Approximation~\cite[Theorem 5.4]{Nor}, we have that,
\begin{thma}
Let $\bbg\subseteq \mathbb{GL}_d$ be a Zariski-connected, perfect, simply-connected $\bbq$-group. Let $\Gamma$ be a finitely generated Zariski-dense subgroup of $\bbg(\bbq)$; then there is a finite set $\scal$ of primes such that,
\begin{itemize}
\item[1.] The closure of $\Gamma$ in $\prod_{p\in \pcal\setminus \scal}\bbg(\bbz_p)$ is an open subgroup.
\item[2.] There is a constant $p_0$ depending on $\Gamma$ such that for any square free integer $q$ with prime factors larger than $p_0$, $\pi_q(\Gamma)=\bbg(\bbz/q \bbz)$.
\item[3.] There is a constant $p_0$ depending on $\bbg$ and its embedding into $\mathbb{GL}_d$, such that for any square free integer $q$ with prime factors larger than $p_0$,
\[
\prod_{p|q}\pi_p :\bbg(\bbz/q \bbz) \xrightarrow{\sim} \prod_{p|q} \bbg(\bbz/p\bbz).
\]
Moreover, for $p>p_0$, $\bbg(\bbf_p)=\bbl(\bbf_p)\ltimes \bbu(\bbf_p)$, $\bbl(\bbf_p)$ is a product of quasi-simple finite groups whose number of factors is bounded in terms of $\dim \bbl$. 
\end{itemize}
\end{thma}
If $\Gamma$ is a finitely generated subgroup of $\bbg(\bbq)$, $\iota:\widetilde{\bbg}\rightarrow \bbg$ is a $\bbq$-isogeny, and $\widetilde{\bbg}=\widetilde{\bbl}\ltimes \bbu$ is simply connected, then it is easy to see that $\Gamma\cap \iota(\widetilde{\bbg})$ is a finite index subgroup of $\Gamma$. Furthermore, for large enough $p$, $\iota(\widetilde{\bbg}(\bbz_p))\subseteq \bbg(\bbz_p)$ and the pre-image of  the first congruence subgroup of $\bbg(\bbz_p)$ is the first congruence subgroup of $\widetilde{\bbg}(\bbz_p)$. In particular, there is a square free number $q_0$ such that, for any $q$, $\pi_q(\Gamma)\simeq \pi_{(q,q_0)}(\Gamma)\times \prod_{p|q/(q,q_0)} \pi_p(\Gamma)$; moreover $\pi_p(\Gamma)=\iota(\widetilde{\bbl}(\bbf_p)) \ltimes \bbu(\bbf_p)$ and $\iota(\widetilde{\bbl}(\bbf_p))$ is a product of quasi-simple finite groups if $p$ is large enough. Let us also add that $\bbu/[\bbu,\bbu]$ is a commutative unipotent $\bbq$-group, and so it is a $\bbq$-vector group, i.e. it is $\bbq$-isomorphic to $\bbg_a^M$ for some $M$ (the logarithm and exponential maps give $\bbq$-isomorphisms between a commutative unipotent $\bbq$-group and its Lie algebra). Thus, for large enough $p$, $(\bbu/[\bbu,\bbu])(\bbf_p)$ is an $\bbf_p$-vector space. As $[\bbu,\bbu]$ is an $\bbf_p$-split unipotent algebraic group, $(\bbu/[\bbu,\bbu])(\bbf_p)=\bbu(\bbf_p)/[\bbu,\bbu](\bbf_p)$. The next lemma, shows that, for large enough $p$, we have $[\bbu,\bbu](\bbf_p)=[\bbu(\bbf_p),\bbu(\bbf_p)]$, and so overall we have that, for large enough $p$, $\bbu(\bbf_p)/[\bbu(\bbf_p),\bbu(\bbf_p)]$ is an $\bbf_p$-vector space. Let $\gamma_k(\bbu)$ be the $k$-th lower central series, i.e. $\gamma_1(\bbu)=\bbu$ and $\gamma_{i+1}(\bbu)=[\bbu,\gamma_i(\bbu)]$. It is well-known that, if $\bbu$ is defined over $\bbq$, then all of the lower central series are also defined over $\bbq$.
\begin{lem}\label{l:LowerCentralSeries}
Let $\bbu$ be a unipotent $\bbq$-algebraic group. Then,  for any $k$ and large enough $p$, $\gamma_k(\bbu)(\bbf_p)=\gamma_k(\bbu(\bbf_p))$.
\end{lem}
\begin{proof}
As $\bbu$ has a $\bbq$-structure, there is a lattice  $\Gamma_U$ in $\bbu(\bbr)$. In particular, it is a finitely generated, Zariski-dense subgroup of $\bbu(\bbq)$. Thus $\gamma_k(\Gamma_U)$ is Zariski-dense in $\gamma_k(\bbu)$, for any $k$. By Nori's result, $\gamma_k(\Gamma_U)$ modulo $p$ is the the full group $\gamma_k(\bbu)(\bbf_p)$, for large enough $p$, which finishes the proof.
\end{proof}

For the rest of this section, $\scal$ is a finite set of primes such that $\Gamma\subseteq \GL_d(\bbz_\scal)$, $q$ will be a square-free integer,
and we assume that it has no prime divisor less than a constant
which depends on $\Ga$.
We write $G_q=\pi_q(\Ga)$.
For future reference we record the properties of $G_q$
that we deduced above:
For any square-free integer $q$ with sufficiently large prime divisors,
and for any sufficiently large prime $p$, we have
\begin{enumerate}
\item
$G_q=\prod_{p|q}G_p$,
\item
$G_p=\G(\F_p)^+=L_p\ltimes U_p$, where
\begin{enumerate}
\item
$L_p$ is a product of quasi-simple finite groups of Lie type over
finite fields which are extensions of $\F_p$;
moreover the number of the quasi-simple factors has an upper-bound
independent of $p$,
\item
$U_p$ is a $k$-step nilpotent $p$-group, where $k$ is independent
of $p$, and $U_p/[U_p,U_p]$ is isomorphic to an $\F_p$-vector space; moreover $\log_p |U_p|$ is bounded independently of $p$.
\end{enumerate}
\end{enumerate}

 Let us also recall from Nori's paper \cite[Theorem B and C]{Nor} that for large enough $p$, any subgroup $H$ of ${\rm GL}_d(\bbf_p)$ satisfies the following properties:
\begin{enumerate}
\item[3.] There is a Zariski-connected algebraic subgroup $\bbh$ of $\mathbb{GL}_d$ defined over $\bbf_p$ such that $\bbh(\bbf_p)^+=H^+$.
\item[4.] \label{pg_4}There is a commutative subgroup $F$ of $H$ such that $H^+\cdot F$ is a normal subgroup of $H$ and $[H:H^+\cdot F]<C$, where $C$ just depends on $d$ the size of the matrices.
\item[5.] \label{page}There is a correspondence between $p$-elements of $H$ and nilpotent elements of $\hfr(\bbf_p)$, where $\hfr$ is the Lie algebra of $\bbh$; moreover $\hfr(\bbf_p)$ is generated by its nilpotent elements. 
\end{enumerate}

\subsection{Description of subgroups}
\label{sc_descr}

In this section we describe in geometric terms the set $\LL_\d(H)$ defined above.
In fact what  we show is that there is a subgroup $H^\sharp$ of small
index, such that $\LL_\d(H^\sharp)$ is contained in a certain proper algebraic subgroup of $\G$.
It is well-known by Chevalley's theorem, that then there is a representation of $\G$
in which $\LL_\d(H^\sharp)$ fixes a line that is not fixed by the whole group.
For technical reasons we need that the representations come from a fixed finite
family; therefore we construct them explicitly. 
In addition, the methods of the next section would require that the representations are
irreducible, which is not possible to fulfill since $\G$ is not necessarily semisimple.
For this reason, besides the irreducible representations we also consider homomorphisms into the group
of affine transformations.
Unfortunately, a finite family of such homomorphism is not rich enough to capture all possible
subgroups. We need to consider uncountable families, where the linear part of the action is the same
and the translation part can be parametrized by elements of an affine space.
The precise formulation is contained in the next proposition that we will use later as a black box in the paper.

\begin{prop}\label{pr_descr}
Let $\G$ be a Zariski-connected perfect $\bbq$-group, and let
 $\Gamma$ be a Zariski-dense, finitely generated subgroup of $\bbg(\bbq)$.
Then there are non-trivial irreducible representations
$\rho_i$ for $1\le i\le m$ of $\G$ and morphisms
$\f_i:\bbg\times \mathbb{V}_i\rightarrow \Aff(\mathbb{W}_{\rho_i})$
with the following properties:
\begin{enumerate}
\item For any $i$, $\mathbb{V}_i$ is a (possibly $0$-dimensional) affine space.
$\mathbb{V}_i$,  $\rho_i$ and $\f_i$  are defined over a local field $\mathcal{K}_i$;
and $\rho_i(\Gamma)$ is an unbounded subset of $\GL(W_{\rho_i})$,
where $W_{\rho_i}=\mathbb{W}_{\rho_i}(\mathcal{K}_i)$.
\item For any $i$ and $0\neq v\in\mathbb{V}_i(\mathcal{K}_i)$, $\f_{i,v}=\f_i(\cdot,v)$
is a group homomorphism to $\Aff(W_{\rho_i})$ whose linear parts are
$\rho_i$, and $\G(\mathcal{K}_i)$
does not fix any point of $W_{\rho_i}$ via this action.
\item There are positive constants $C$ and $\delta$ such that the following holds.
Let $q=p_1\cdots p_n$ be a square-free number, such that each $p_i$ is a sufficiently large prime,
and let $H$ be a proper subgroup of $\pi_q(\Gamma)$.
Then there is a subgroup $H^{\sharp}$ of index at most $C^n$
in $H$ that satisfies one of the following two conditions:
 \begin{enumerate}
 \item For some $1\le i\le m$, there is $w\neq 0$ in $W_{\rho_i}$
such that $\rho_i(h)([w])=[w]$, for any $h\in \lcal_{\delta}(H^{\sharp})$.
 \item For some $1<i\le m$, there is $v\in \mathbb{V}_i(\mathcal{K}_i)$
and $w\in W_{\rho_i}$  such that $\|v\|=1$ and $\phi_{i,v}(h)(w)=w$,
for any $h\in\lcal_{\delta}(H^{\sharp})$. 
 \end{enumerate}
\end{enumerate}
\end{prop}

We will easily deduce Proposition~\ref{pr_descr} from the following more technical version.

\begin{prop}\label{p:descr}
Let $\bbg$ and $\Gamma$ be as in the setting of Proposition~\ref{pr_descr} and $\bbg=\bbl\ltimes\bbu$, where $\bbl$ is a semisimple group and $\bbu$ is a unipotent group. Then there are finitely many representations $\rho_1,\ldots,\rho_{m'},\psi_1,\ldots,\psi_k$ of $\bbg$ with the following properties:
 \begin{enumerate}
 \item[1.] For any $i$, $\bbu\subseteq \ker(\rho_i)$ and the restriction of $\rho_i$ to $\bbl$ is a non-trivial irreducible representation.
 \item[2.] For any $i$, there is a sub-representation $\mathbb{W}^{(1)}_i$ of $\mathbb{W}_{\psi_i}$ such that
 \begin{enumerate}
 \item $\bbu$ acts trivially on $\mathbb{W}^{(1)}_i$ and $\mathbb{W}_{\psi_i}/\mathbb{W}^{(1)}_i$.
 \item $\mathbb{W}^{(1)}_i$ is a non-trivial irreducible representation of $\bbl$ that we denote by $\rho_{m'+i}$.
 \item $\mathbb{W}_{\psi_i}=\mathbb{W}^{(1)}_i\oplus \mathbb{W}^{(2)}_i$, where $\mathbb{W}^{(2)}_i=\mathbb{W}_{\psi_i}^{\bbl}$ is the set of $\bbl$-invariant vectors.
 \end{enumerate}
 \item[3.] For any $i$, there are local fields $\mathcal{K}_i$ such that $\rho_i$  is defined over $\mathcal{K}_i$ and $\psi_i$ are defined over $\mathcal{K}_{i+m'}$; moreover $\rho_i({\rm pr}(\Gamma))$, where ${\rm pr}$ is the projection to $\bbl$, is an unbounded subset of $\rho_i(\bbl(\mathcal{K}_i))$.
 \item[4.] There are positive constants $C$ and $\delta$ such that the following holds.
Let $q=p_1\cdots p_n$ be a square-free number, such that each $p_i$ is a sufficiently large prime,
and let $H$ be a proper subgroup of $\pi_q(\Gamma)$.
Then there is a subgroup $H^{\sharp}$ of index at most $C^n$
in $H$ that satisfies one of the following two conditions:
 \begin{enumerate}
 \item For some $i$, there is $w\neq 0$ in $W_{\rho_i}=\mathbb{W}_{\rho_i}(\mathcal{K}_i)$ such that $\rho_i(h)([w])=[w]$, for any $h\in \lcal_{\delta}(H^{\sharp})$.
 \item For some $i$, there is a vector $0\neq w\in W_{\psi_i}$ such that $\psi_i(h)(w)=w$, for any $h\in\lcal_{\delta}(H^{\sharp})$. Moreover there is no nonzero vector $w'$ in $W_{\psi_i}$ such that $\psi_i(\bbg)(w')=w'$.
 \end{enumerate}
 \end{enumerate}
 \end{prop}
 \begin{proof}[Proof of Proposition~\ref{pr_descr} assuming Proposition~\ref{p:descr}]
 For any $1\le i\le m'$, we let $\rho_i$ be the same as in Proposition~\ref{p:descr}. For these representations, we take $\mathbb{V}_i=0$ and let $\f_i(g,0)=\rho_i(g)$.  For $1\le i\le k$, we let $\rho_{m'+i}$ be the representation of $\bbl$ on $\mathbb{W}^{(1)}_{i}$ and $\mathbb{V}_{m'+i}=\mathbb{W}_i^{(2)}$. For any $w_1\in \mathbb{W}^{(1)}_{i}$, let 
 \[
 \f_{m'+i}(g,v)(w_1):=\psi_i(g)(w_1+v)-v.
 \]
Notice that, since $g$ acts trivially on $\mathbb{W}_{\psi_i}/\mathbb{W}^{(1)}_i$, $\f_{m'+i}(g,v)\in \Aff(\mathbb{W}_{i}^{(1)})$,  for any $i$. 

With these choices, in order to complete the proof, it is enough to make the following observations. If for some $m'<i\le m$ and $0\neq v\in \mathbb{V}_i$, $\bbg$ fixes a point $w_1\in\mathbb{W}_{\rho_i}$, then by definition, $w_1+v$ is fixed by $\bbg$. Therefore $w_1+v=0$, and so $w_1=v=0$, which is a contradiction. On the other hand, if $\psi_i(h)(w)=w$, where $w=w_1+v$, $w_1\in \mathbb{W}_i^{(1)}$ and $v\in \mathbb{W}^{(2)}_i=\mathbb{V}_i$, then $\f_{i,v}(h)(w_1)=w_1.$ 
 \end{proof}
 We prove Proposition~\ref{p:descr} in two steps. First we show that for an appropriate choice of $\delta$ and $H^{\sharp}$, the Zariski-closure of the group generated by $\lcal_{\delta}(H^{\sharp})$ is a proper subgroup of $\bbg$. Then, for any proper closed subgroup of $\bbg$, we construct the desired representations.  We start with some auxiliary lemmata describing the normal subgroups of $G_p$. 
 \begin{lem}\label{l:NormalSubgroupsSemisimple}
Let $L=\prod_{i=1}^m L^{(i)}$, where $L^{(i)}$ are quasi-simple groups. Then any normal subgroup $H$ of $L$ is of the form $\prod_{i\in I}L^{(i)}\times Z$, where $I$ is a subset of $\{1,\ldots,m\}$ and $Z\le\prod_{i\not\in I}Z(L^{(i)})$.
\end{lem}
\begin{proof}
For any $i$, either ${\rm pr}_i(H)$, the projection onto $L^{(i)}$, is central or ${\rm pr}_i(H)=L^{(i)}$. If $(s_1,\ldots,s_m)$ is in $H$, then, for any $g\in L^{(i)}$, $(1,\ldots,1,[g,s_i],1,\ldots,1)$ is also in $N$. On the other hand, the group generated by $[g,s_i]$ is a normal subgroup of $L^{(i)}$. If $s_i$ is not central, then the above group cannot be central as $Z\left(L^{(i)}/Z(L^{(i)})\right)=\{1\}$, and so it is the full group $L^{(i)}$. The rest of the argument is straightforward.
\end{proof}
\begin{lem}~\label{l:NormalSubgroups}
Let $L$ be a direct product of quasi-simple finite groups which acts on a finite nilpotent group  $U$. Then any normal subgroup $H$ of $L\ltimes U$ is of the form $(H\cap L)\ltimes (H\cap U)$ if the prime factors of $|U|$ are larger than an absolute constant depending only on the size of the center of $L$. Moreover $H\cap L$ acts trivially on $U/H\cap U$.
\end{lem}
\begin{proof}
Passing to $H/H\cap U \unlhd L \ltimes (U/H\cap U)$, we can and will assume that $H\cap U=\{1\}$. Thus projection to $L$ induces an embedding and we get a map $\f:{\rm pr}(H)\rightarrow U$, where $\pr:L\ltimes U\to L$ is the projection map, such that 
\[
H=\{(s,\f(s))|\h s\in {\rm pr}(H)\}.
\]
One can easily check that $\f$ is a $1$-cocycle, i.e. $\f(s_1s_2)=s_2^{-1} \f(s_1) s_2\cdot \f(s_2)$. Furthermore, for any $u\in U$ and $s\in L$, we have 
\[
(1,u^{-1})(s,\f(s))(1,u)=(s,s^{-1}u^{-1}s\cdot \f(s)\cdot u)\in H;
\]
thus $\f(s)=s^{-1}u^{-1}s\cdot \f(s)\cdot u$, for any $u\in U$ and $s\in L$. In particular, setting $s=s_2$ and  $u=\f(s_1)^{-1}$, and then using the cocycle relation, we have 
\[
\f(s_2)\cdot \f(s_1)=s_2^{-1}\f(s_1)s_2\cdot \f(s_2)=\f(s_1s_2).
\]
Since $\f(1)=1$, by the above equation, we have that $\f(s^{-1})=\f(s)^{-1}$. Therefore, by the above discussion, $\theta(s)=\f(s^{-1})$ defines a homomorphism from ${\rm pr}(H)$ to $U$. On the other hand, ${\rm pr}(H)$ is a normal subgroup of $L$, so by Lemma~\ref{l:NormalSubgroupsSemisimple}, if the prime factors of $|U|$ are larger than the size of the center of $L$, then $\theta$ is trivial, which finishes the proof of the first part. For the second part, it is enough to notice that $sus^{-1}u^{-1}$ is in $H\cap U$, for any $s\in H\cap L$ and $u\in U$.
\end{proof}
\begin{cor}\label{c:NormalSubgroupPerfect}
Let $L$ be a product of quasi-simple finite groups, which acts on a finite nilpotent group $U$. Assume that $G=L\ltimes U$ is a perfect group. If $H$ is a normal subgroup of $G$ and the projection of $H$ onto $L$ is surjective, then $H=G$.
\end{cor}
\begin{proof}
By Lemma~\ref{l:NormalSubgroups}, $H=L\ltimes U'$, for a normal subgroup $U'$ of $U$, and $L$ acts trivially on $U/U'$. Thus $L\ltimes U/U'\simeq L\times U/U'$ is not a perfect group unless $U=U'$. On the other hand, any quotient of $G$ is perfect, which finishes the proof.
\end{proof}

In the next step, for any proper subgroup $H$ of $G_q$, we will find another subgroup containing $H$ which is of product form and is of comparable size. 
\begin{lem}\label{l:ProductForm}
Let  $H$ be a proper subgroup of $G=\prod_{p\in \Sigma} G_p$, where 
\begin{itemize}
\item[0.] $\Sigma$ is a finite set of primes larger than $7$,
\item[1.] $G_p=L_p\ltimes U_p$, where $L_p=\prod_i L^{(i)}_p$, $L^{(i)}_p$ are quasi-simple groups of Lie type over a finite field of characteristic $p$, and $U_p$ is a $p$-group,
\item[2.] $G_p$ is perfect,
\item[3.] $|G_p|\le p^k$ for any $p\in \Sigma$, with a fixed $k$ independent of $p$.
\end{itemize}
Then there is a positive number $\delta$ depending on $k$, such that 
\[
\prod_{p\in \Sigma}[G_p:\pi_p(H)]\ge [G:H]^{\delta},
\]
where $\pi_p$ is the projection onto $G_p$.
\end{lem}
\begin{proof}
We prove this by induction on $|G|$. Let $\Sigma'=\{p\in \Sigma|\h \pi_p(H)=G_p\}$. If $\Sigma'$ is empty, then $[G_p:\pi_p(H)]\ge p$ for any $p$, as $\pi_p(H)$ is a proper subgroup of $G_p$ and $G_p$ is generated by $p$ elements. Therefore
\[
\prod_{p\in \Sigma} [G_p:\pi_p(H)]\ge \prod_{p\in \Sigma}p\ge \prod_{p\in \Sigma}|G_p|^{1/k}\ge [G:H]^{1/k}.
\]
So we shall assume that $\Sigma'$ is non-empty. For any $p$, $H_p:=G_p\cap H$ is a normal subgroup of $\pi_p(H)$; in particular, $H_p$ is a normal subgroup of $G_p$ when $p\in \Sigma'$. Let 
\[
G_p'=\begin{cases}G_p&\text{if $p\not\in \Sigma'$,}\\
&\\
G_p/H_p &\text{if $p\in \Sigma'$,}
\end{cases}
\h\h\h\quad{\rm and}
\]
\[
H'=H/\prod_{p\in \Sigma'} H_p\subseteq \prod_{p\in \Sigma} G_p'=:G'.
\]
If $H_p$ is not trivial for some $p\in \Sigma'$, then $|G'|<|G|$. Moreover, by Lemma~\ref{l:NormalSubgroups} and Corollary~\ref{c:NormalSubgroupPerfect}, we can apply the induction hypothesis, and we get that
\[
[G:H]^{\delta}=[G':H']^{\delta}\le \prod_{p\in\Sigma} [G_p':\pi_p(H')]= \prod_{p\in\Sigma} [G_p:\pi_p(H)],
\]
and we are done. Thus, without loss of generality, we can and will assume that $H_p$ is trivial for any $p\in \Sigma'$. 

Let $p_0\in \Sigma'$. Since $H_{p_0}$ is trivial and $\pi_{p_0}(H)=G_{p_0}$, there is an epimorphism from $H':=\pi_{\Sigma\setminus\{p_0\}} (H)$ to $G_{p_0}$ and $H'$ is isomorphic to $H$. Let $N$ be the kernel of this epimorphism. By the induction hypothesis, we have that
\[
\prod_{p\in \Sigma\setminus\{p_0\}}[G_p:\pi_p(N)]\ge [\prod_{p\in \Sigma\setminus\{p_0\}} G_p:N]^{\delta}.
\]
On the other hand, $|H|=|H'|=|G_{p_0}|\cdot|N|$; so 
\begin{equation}\label{e:HN}
\prod_{p\in \Sigma\setminus\{p_0\}}[G_p:\pi_p(N)]\ge [G:H]^{\delta}.
\end{equation}
By (\ref{e:HN}), it is clear that, if we have $\pi_p(N)=\pi_p(H)$ for all $p$, then we are done. Thus we can assume that $\pi_p(N)\neq \pi_p(H)$  for some $p$.

Since $H'/N\simeq G_{p_0}$, it is clear that, for any $p\neq p_0$, $\pi_p(H)/\pi_p(N)$ is a homomorphic image of $G_{p_0}$. Thus, by the Jordan-H\"{o}lder Theorem, either $G_{p_0}$ and $\pi_p(H)$ have some composition factor in common or $\pi_p(H)=\pi_p(N)$. On the other hand, the composition factors of $G_{p}$ are the cyclic group of order $p$ and $L_{p}^{(i)}$. In particular, if $p$ and $p'$ are distinct primes and either $p$ or $p'$ is larger than $7$, then $G_p$ and $G_{p'}$ do not have any composition factor in common.  So $\pi_p(H)=\pi_p(N)$ if $p\in \Sigma'\setminus\{p_0\}$.  
So, for any $p_0\in\Sigma'$, there is $p\in \Sigma\setminus\Sigma'$ such that $|L^{(i)}_{p_0}|$ divides $|G_p|$, for some $i$. In particular, $p_0$ divides $|G_p|$. Thus we have that
\[
\prod_{p\in \Sigma'}p\h |\prod_{p\in\Sigma\setminus\Sigma'} |G_p|.
\]
Therefore $\prod_{p\in \Sigma'}p\le \prod_{p\in\Sigma\setminus\Sigma'} p^{k}.$ Now it is straightforward to show that 
\[
\prod_{p\in\Sigma}[G_p:\pi_p(H)]\ge [G:H]^{\frac{1}{k(k+1)}},
\]
and we are done.
\end{proof}

After these preparations we are ready to make the first step towards the
proof of Proposition \ref{p:descr}.
Recall that we try to describe the set $\lcal_\d(H)$ in a geometric way,
which set is the set of "small lifts" of $H$ to $\G(\Z)$.
The next Proposition shows that for every $H<G_q$, we can find a subgroup
$H^\sharp$ for which $\lcal_\d(H^\sharp)$ is contained in a proper algebraic
subgroup of $\G$.

\begin{prop}\label{p:LiftingUp}
There are positive numbers $\delta$ and $C$ for which the following holds:

Take any proper subgroup $H$ of $G_q$, where $q$ is a square free
integer with large prime factors.
Then one can find a proper algebraic subgroup $\H<\G$ defined over $\Q$,
and a subgroup $H^{\sharp}$ of $H$, such that
\begin{enumerate}
\item $\lcal_{\delta}(H^{\sharp})\subset \H(\Q)$,
\item $[H:H^{\sharp}]\le C^n,$
\end{enumerate}
where $n$ is the number of prime factors of $q$.
\end{prop}
\begin{proof}
We begin the proof by constructing $H^\sharp$.
Let $q''$ be the product of prime factors of $q$ such that $G_p\neq \pi_p(H)$.
 By Nori's result (see 4. on page \pageref{pg_4}), for large enough $p$, there is a commutative subgroup  $F_p$ of $\pi_p(H)$ such that
$H_p^{\sharp}:=\pi_p(H)^+\cdot F_p$ is a normal subgroup of $\pi_p(H)$ and $[\pi_p(H):H_p^{\sharp}]<C$, where $C$ just depends on the size of the matrices.
Let  $H_{q''}^{\sharp}=\prod_{p|q''} H_p^{\sharp}$ and 
\[
H^{\sharp}=\{h\in H|\h \pi_{q''}(h)\in H_{q''}^{\sharp}\}. 
\]
It is clear that
$[H:H^{\sharp}]\le [\prod_{p|q''}\pi_p(H):H_{q''}^{\sharp}]\le C^n$,
where $n$ is the number of prime factors of $q$.
We will see that without loss of generality we can replace $q$  by $q''$ and $H$ by $H^\sharp_{q''}$.
By Lemma \ref{l:ProductForm} and the discussions in the beginning of Section~\ref{sc_escp}, we have that
\[
[G_q:H]^{\delta_1}\le [G_{q''}:\prod_{p|q''}\pi_p(H)].
\]
Hence 
\[
[G_q:H^{\sharp}]^{\delta_1}\le [G_{q''}:H_{q''}^{\sharp}].
\]
 Therefore $\lcal_{\delta}(H^{\sharp})\subseteq \lcal_{\delta/\delta_1}(H_{q''}^{\sharp})$.
So, without loss of generality, we assume that
 \begin{enumerate}
\item $H=\prod_{p|q} H_p$, where $H_p$ is a proper subgroup of 
$G_p$. 
\item $H^{\sharp}=\prod_{p|q} H^{\sharp}_p$.
 \end{enumerate}

Consider an embedding $\G\subseteq \G\L_d$ defined over $\Q$.
Below we will show that for each prime $p|q$
there is a polynomial $P_p\in(\Z/p\Z)[x_{1,1},\ldots,x_{d,d}]$ of degree
at most 4 such that all elements $g\in H_p^\sharp$
satisfy this equation, i.e. $P_p(g)=0$, but not all elements of $G_p$ do.
Now we show that this easily implies
the first part of the proposition with some
$\d>0$.
To do this, we first show that there is a polynomial
$P\in\Q[x_{1,1},\ldots,x_{d,d}]$ which vanishes on $\lcal_\d(H^\sharp)$
but not on all of $\G$.
Consider the usual degree 4 monomial map:
$\Psi:\G\L_d\to\bba_{\binom{d^2+4}{4}}$.
Denote by $D$ the dimension of the linear subspace of $\bba_{\binom{d^2+4}{4}}$
spanned by $\Psi(\G)$.
We need to show that $\Psi(\lcal_\d(H^\sharp))$ spans a subspace of
dimension lower than $D$.
Suppose the contrary, and let $g_1,\ldots,g_D\in\Psi(\lcal_\d(H^\sharp))$
which are linearly independent.
We can consider the $g_i$ as column vectors, and form a
${\binom{d^2+4}{4}}\times D$ matrix.
By independence this has a nonzero $D\times D$ subdeterminant,
whose entries are all less than
\[
\frac{1}{D!}q^{\frac{1}{D(|\scal|+1)}}
\]
in the $\|\cdot\|_\scal$-norm,
if we choose $\d$ sufficiently
small.
Recall that $\scal$ is the set of primes which occur in the denominators
of elements of $\Ga$.
Now, the value of this subdeterminant is a nonzero rational number
less than $q^{1/(|\scal|+1)}$, whose denominator is less than
$q^{|\scal|/(|\scal|+1)}$.
This implies that the projection mod $p$ of this determinant
is still nonzero for some $p|q$, which contradicts the existence of $P_p$
to be demonstrated below.

So far we showed that for sufficiently small $\d$, $\lcal_\d(H^\sharp)$
is contained in a proper subvariety $\X\subseteq\G$.
By \cite[Proposition 3.2]{EMO}, there is an integer $N$ such that
if $A\subseteq\G(\Q)$ is a finite symmetric set generating a Zariski dense
subgroup of $\G$, then $\prod_N A\nsubseteq\X(\Q)$.
This implies that $\lcal_{\d/N}(H^\sharp)$ is contained in a proper
algebraic subgroup.
We note that the proof of \cite[Proposition 3.2]{EMO}
gives that $N$ depends only on the dimension, the degree and the number of
irreducible components of $\X$.
The proof in \cite[Proposition 3.2]{EMO} is based on the idea, that
by Zariski density, one can find an element $g\in A$ such that
$\X\cap g\X$ is either of lower dimension or contains less components
of maximal dimension than $\X$.
$N$ is the number of iterations we need to make to get a trivial
intersection.
It is clear that we can keep track of the dimension, the number of components
and the degree of the varieties that arise this way.
Hence the
procedure terminates in $N$ steps controlled by those parameters only.

It remains to show our claim about the existence of the polynomials $P_p$.
In what follows, let $\gfr$ be the Lie algebra of $\bbg$
and $\gfr^*$ its dual.

We consider the adjoint representation of $\bbg$
and its dual  on these spaces, respectively.
For large enough $p$, these actions reduce to the action
of $\bbg(\bbf_p)$ on $\gfr(\bbf_p)$ and $\gfr^*(\bbf_p)$;
moreover $\gfr^*(\bbf_p)=\gfr(\bbf_p)^*$.
There is a natural non-degenerate bilinear form on
$\gfr\otimes \gfr^*$, which is the linear extension of the following map
\[
\langle v_1\otimes f_1, v_2\otimes f_2\rangle:= f_1(v_2) f_2(v_1).
\]
(It is worth mentioning that this bilinear map is $\bbg$-invariant.)
It is also well-known that $\gfr\otimes \gfr^*$ is isomorphic to
${\rm End}(\gfr)$ as a $\bbg$-module, where $\bbg$ acts on
${\rm End}(\gfr)$ via the conjugation by the adjoint representation.
We denote both of these representations by $\rho$.  

Clearly, we can assume that any prime divisor $p|q$ is sufficiently large,
hence $\langle\cdot,\cdot\rangle$ induces a non-degenerate bilinear
form on $\gfr(\bbf_p)\otimes \gfr^*(\bbf_p)$. 

For any $x$ and $y$ in $\gfr\otimes \gfr^*$, let $\eta_{x,y}$
be a polynomial in $d^2$ variables with coefficients in
$\bbz[1/q_0]$, which is defined as follows,
\[
\eta_{x,y}(g):=\langle \rho(g)(x), y\rangle,
\] 
We show that for a prime divisor $p$ of $q$,
we can find some $g_0\in S$ and
$x$ and $y$ in $\gfr\otimes \gfr^*$ such that
$\eta_{x,y}(g)=0$ modulo $p$,
for any $h\in H_p^{\sharp}$, and $\eta_{x,y}(g_0)=1$.
Since $P_p:=\eta_{x,y}$ is of degree at most 4, this proves our claim,
and hence the proposition.

Since $\langle\cdot,\cdot \rangle$ is a non-degenerate
bilinear form on $\gfr(\bbf_p)\otimes \gfr(\bbf_p)^*$,
it is enough to show that there is a proper subspace
of $\gfr(\bbf_p)\otimes \gfr(\bbf_p)^*$ which is
$H_p^{\sharp}$-invariant, but not $\bbg(\bbf_p)^+$-invariant. 

If $H_p^+$ is not a normal subgroup of $\bbg(\bbf_p)^+$,
then clearly, by Nori's result (see 3. and 5. on page \pageref{page}
in Section~\ref{sc_escp}), $\hfr(\bbf_p)$ is $H_p^{\sharp}$-invariant,
but not $\bbg(\bbf_p)^+$-invariant.
So $\hfr(\bbf_p)\otimes \gfr^*(\bbf_p)$ has the desired property. 

Now, let us assume that $H_p^+$ is a normal subgroup.
Since it is a proper normal subgroup of $\bbg(\bbf_p)^+$,
by Corollary~\ref{c:NormalSubgroupPerfect},
the projection ${\rm pr}(H_p^+)$ of $H_p^+$ to $\bbl(\bbf_p)^+$
is a proper normal subgroup. On the other hand, we know that
\[
\bbl(\bbf_p)^+\simeq \prod_{i=1}^m \prod_{\pfr\in \pcal(k_i), \pfr|p}
\bbl_i(\bbf_{\pfr})^+,
\]
and $\bbl_i(\bbf_{\pfr})^+$ is quasi-simple, for any $i$ and $\pfr$.
Hence, by Lemma \ref{l:NormalSubgroupsSemisimple},
there is a non-empty subset $I$ of possible indices $(i,\pfr)$ such that
\[
{\rm pr}(H_p^+)\subseteq \prod_I \bbl_i(\bbf_{\pfr})^+
\times \prod_{I^c} Z(\bbl_i(\bbf_{\pfr})^+),
\]
where $I^c$ is the complement of $I$ in the set of all the possible indices.
Thus we have that ${\rm pr}(H_p^+)= \prod_I \bbl_i(\bbf_{\pfr})^+$ 
as $H_p^+$ is generated by $p$-elements.
We also notice that $\gfr=\lfr \oplus \ufr$,
where $\lfr$ is the Lie algebra of $\bbl$ and $\ufr$ is the Lie algebra
of $\bbu$.
Moreover 
\[
\gfr(\bbf_p)=\left[\bigoplus_{i=1}^m\bigoplus_{\pfr\in \pcal(k_i),
\pfr|p} \lfr_i(\bbf_{\pfr})\right] \oplus \ufr(\bbf_p),
\]
where $\lfr_i$ is the Lie algebra of $\bbl_i$.
For any possible indices $(i,\pfr)$, let $F_{i,\pfr}$
be the projection of $F_p$ (defined at the beginning of this proof)
to $\bbl_i(\bbf_p)^+$.
We notice that $\lfr_i(\bbf_{\pfr})$ is an irreducible
$\bbl_i(\bbf_{\pfr})^+$-module.
Therefore, if $\lfr_i(\bbf_{\pfr})$ is not an irreducible
$F_{i,\pfr}$-module, for some $(i,\pfr)$, then one can easily get a
proper subspace of $\gfr(\bbf_p)$ which is invariant under $H_p^{\sharp}$,
but not under $\bbg(\bbf_p)^+$ and finish the argument as before.
So, without loss of generality, we assume that $\lfr_i(\bbf_{\pfr})$
is an irreducible $F_{i,\pfr}$-module.
Since $F_{i,\pfr}$ is a commutative group,
the $\bbf_{\pfr}$-span $E_{i,\pfr}$ of its image in
${\rm End}(\lfr_i(\bbf_\pfr))$ is a field extension of
$\bbf_{\pfr}$ of degree $\dim_{\bbf_\pfr} \lfr_i(\bbf_\pfr)$.
(In particular, by the Double Centralizer Theorem, the centralizer
of $E_{i,\pfr}$ in ${\rm End}(\lfr_i(\bbf_\pfr))$ is itself.)
Now we consider the subspace $W$ of ${\rm End}(\gfr(\bbf_p))$
consisting of elements $x$ with the following properties,
\[
\begin{array}{rll}
1.&x(\ufr(\bbf_p))\subseteq \ufr(\bbf_p)&\\
&&\\
2.&x(\lfr_i(\bbf_{\pfr}))\subseteq \lfr_i(\bbf_{\pfr})\oplus
\ufr(\bbf_p)&\text{if $(i,\pfr)\in I$,}\\
&&\\
3.&\exists\h y\in E_{i,\pfr}, \forall \h l_i\in
\lfr_i(\bbf_{\pfr}):\h x(l_i)-y(l_i)\in\ufr(\bbf_p)&
\text{if $(i,\pfr)\not\in I$}. 
\end{array}
\]
We claim that $W$ is $H_p^{\sharp}$-invariant,
but not $\bbg(\bbf_p)^+$-invariant.
Let $g\in \bbg(\bbf_p)^+$ and $x\in W$;
then ${\rm Ad}(g)^{-1}x{\rm Ad}(g)$ clearly satisfies the first
and second conditions.
It is straightforward to check that ${\rm Ad}(g)^{-1}x\h{\rm Ad}(g)$
satisfies the third condition for all $x$ if and only if 
\begin{equation}\label{e:FieldInvariant}
{\rm Ad}(g_{i,\pfr})^{-1} E_{i,\pfr} \h{\rm Ad}(g_{i,\pfr})= E_{i,\pfr},
\end{equation}
for any $(i,\pfr)\not\in I$, where $g_{i,\pfr}$ is the
projection of $g$ onto $\bbl_i(\bbf_{\pfr})^+$.
So clearly $W$ is $H_p^{\sharp}$-invariant.
On the other hand, any element in
\[
N(E_{i,\pfr})=\{x\in {\rm GL}(\lfr_i(\bbf_{\pfr}))|\h x^{-1}
E_{i,\pfr} x= E_{i,\pfr}\}
\]
induces a Galois element and $E_{i,\pfr}$ is a maximal
subfield of ${\rm End}(\lfr_i(\bbf_{\pfr}))$.
Therefore $[N(E_{i,\pfr}):E_{i,\pfr}^*]\le \dim_{\bbf_{\pfr}}
\lfr_i(\bbf_{\pfr})$. Thus $\bbg(\bbf_p)^+$ cannot leave $W$
invariant if $p$ is large enough, as we wished. 

\end{proof} 

The next proposition is the source of the desired representations claimed in Proposition~\ref{p:descr}.
\begin{prop}\label{p:ProperSubgroup}
Let $\bbg=\bbl\ltimes\bbu$ be a perfect group, where $\bbl$ is a semisimple group and $\bbu$ is a unipotent group. There are finitely many representations $\rho_1,\ldots,\rho_{m'},\psi_1,\ldots,\psi_k$ of $\bbg$ with the following properties:
\begin{enumerate}
\item For any $i$, $\bbu\subseteq \ker(\rho_i)$ and the restriction of $\rho_i$ to $\bbl$ is a non-trivial irreducible representation.
\item For any $i$, there is a sub-representation $\mathbb{W}^{(1)}_i$ of $\mathbb{W}_{\psi_i}$ such that
\begin{enumerate}
\item $\bbu$ acts trivially on $\mathbb{W}^{(1)}_i$ and $\mathbb{W}_{\psi_i}/\mathbb{W}^{(1)}_i$.
\item $\mathbb{W}^{(1)}_i$ is a non-trivial irreducible representation of $\bbl$ that we denote by $\rho_{m'+i}$.
\item $\mathbb{W}_{\psi_i}=\mathbb{W}^{(1)}_i\oplus \mathbb{W}^{(2)}_i$, where $\mathbb{W}^{(2)}_i=\mathbb{W}_{\psi_i}^{\bbl}$ is the set of $\bbl$-invariant vectors.
\end{enumerate}
\item For any proper subgroup $\bbh$ of $\bbg$, one of the following holds:
\begin{enumerate}
\item For some $i$, there is $w\neq 0$ in $\mathbb{W}_{\rho_i}$ such that $\rho_i(\bbh)([w])=[w]$.
\item For some $i$, there is $0\neq w\in \mathbb{W}_{\psi_i}$ such that $\psi_i(\bbh)(w)=w$. Moreover there is no non-zero vector $w'$ in $\mathbb{W}_{\psi_i}$ such that $\psi_i(\bbg)(w')=w'$.
\end{enumerate}
\end{enumerate}
\end{prop}
\begin{proof}
We divide the argument into several cases. First we consider the case, where the projection of $\bbh$ onto $\bbl$ is not surjective. In the second case we will assume that the group generated by its projection to $\bbu$ is a proper subgroup of $\bbu$, and the third case finishes the argument. In each step, we introduce only finitely many representations which satisfy the desired properties.
 
In the first case, without loss of generality, we can assume that $\bbg=\bbl$ is a semisimple group and $\bbh$ is a proper subgroup. If $\bbh$ is not a normal subgroup, then, in the representation $\wedge^{\dim \hfr} \Ad$, $[w]=\wedge^{\dim \hfr} \hfr\in W$ is $\bbh$-invariant, but it is not $\bbl$-invariant. Since $\bbl$ is semisimple, we can take a decomposition of $W$ into irreducible components. Since $[w]$ is not $\bbl$-invariant, its projection to one of the non-trivial irreducible components is not zero, and so this representation satisfies the condition 3(a).

Notice that this process introduced only finitely many representations which satisfy the properties of $\rho_i$.

If $\bbh$ is a proper normal subgroup and $\bbl=\prod_{i=1}^{m_0}\bbl_i$, where $\bbl_i$ is an absolutely almost simple group, then there is a proper subset $I$ of indices such that
\[
\bbh\subseteq \prod_{i\in I} \bbl_i\times Z(\prod_{i\not\in I} \bbl_i).
\]
Consider the action of $\bbl$ on the Lie algebra $\lfr_i$ of $\bbl_i$ via the adjoint representation of $\bbl_i$, for any $i\not\in I$. Clearly any line in this representation is fixed by $\bbh$ and not by $\bbg$, which finishes the proof of the first case.

We notice that if $\bbh$ is a proper subgroup of $\bbg$, then $\bbh[\bbu,\bbu]$ is also a proper subgroup of $\bbg$. So, without loss of generality, we assume that $\bbu$ is a vector group, i.e. isomorphic to $\bbg_a^{k_0}$, for some $k_0$.

Now we assume that the projection of $\bbh$ onto $\bbl$ is surjective, but the group generated by its projection to $\bbu$ is a proper subgroup. So, without loss of generality, we assume that $\bbh=\bbl\ltimes \bbu'$, where $\bbu'$ is a proper subgroup of $\bbu$. We consider $\widetilde{\mathbb{W}}=\bbu$ as an $\bbl$-space (and $\widetilde{\mathbb{W}}'=\bbu'$ is a proper $\bbl$-subspace), and take its decomposition into homogeneous subspaces, i.e. 
\[
\widetilde{\mathbb{W}}=\widetilde{\mathbb{W}}_1\oplus \widetilde{\mathbb{W}}_2\oplus \cdots \oplus\widetilde{\mathbb{W}}_n,
\]
where ${\rm Hom}_{\bbl}(\widetilde{\mathbb{W}}_i,\widetilde{\mathbb{W}}_j)=0$ if $i\neq j$, and $\widetilde{\mathbb{W}}_i\simeq \mathbb{W}_i^{m_i}$, where $\mathbb{W}_i$ is an irreducible $\bbl$-space.
Here ${\rm Hom}_{\bbl}(\widetilde{\mathbb{W}}_i,\widetilde{\mathbb{W}}_j)$ denotes the space
of  $\bbl$-equivariant linear maps from $\widetilde{\mathbb{W}}_i$  to $\widetilde{\mathbb{W}}_j$.
Since $\mathbb{W}'$ is a proper $\bbl$-subspace, its projection to at least one of the homogeneous spaces is proper. Therefore, without loss of generality, we assume that $\widetilde{\mathbb{W}}\simeq \mathbb{W}^m$, where $\mathbb{W}$ is an irreducible $\bbl$-space. We call a map $f$ from $\widetilde{\mathbb{W}}$ to $\mathbb{W}$ an affine map  if 
\[
f(t_1 \tilde{w}_1+t_2 \tilde{w}_2)=  t_1 f(\tilde{w}_1)+t_2f(\tilde{w}_2),
\]
for any $\tilde{w}_1$ and $\tilde{w}_2$ in $\widetilde{\mathbb{W}}$ and $t_1+t_2=1$. Let $\Aff(\widetilde{\mathbb{W}},\mathbb{W})$ be the set of all affine maps. If $f$ is an affine map, then there are $f_{{\rm lin}}\in {\rm Hom}(\widetilde{\mathbb{W}},\mathbb{W})$ and $w\in \mathbb{W}$ such that 
\[
f(x)=f_{{\rm lin}}(x)+w,
\]
for any $x\in\widetilde{\mathbb{W}}$; $f_{\rm lin}$ is called the linear part of $f$. Let $\Aff_{\bbl}(\widetilde{\mathbb{W}},\mathbb{W})$ be the set of affine maps whose linear part is in ${\rm Hom}_{\bbl}(\widetilde{\mathbb{W}},\mathbb{W})$. Therefore 
\[
\Aff_{\bbl}(\widetilde{\mathbb{W}},\mathbb{W})= {\rm Con}(\widetilde{\mathbb{W}},\mathbb{W})\oplus{\rm Hom}_{\bbl}(\widetilde{\mathbb{W}},\mathbb{W}),
\]
where 
${\rm Con}(\widetilde{\mathbb{W}},\mathbb{W})$ is the space of the constant functions. We claim that the representation $\psi$ of  $\bbg=\bbl\ltimes \widetilde{\mathbb{W}}$ on $\Aff_{\bbl}(\widetilde{\mathbb{W}},\mathbb{W})$, defined by
\[
\psi(l,\tilde{w})(f)(x):=l\cdot f(l^{-1}\cdot x-\tilde{w}),
\]
satisfies our desired conditions. Alternatively, we can say that if $f(x)=f_{{\rm lin}}(x)+w,$ then
\[
\psi(l,\tilde{w})(f)(x)=f_{{\rm lin}}(x)+l\cdot w -l\cdot f_{{\rm lin}}(\tilde{w}).
\]
Both of the subspaces ${\rm Con}(\widetilde{\mathbb{W}},\mathbb{W})$ and ${\rm Hom}_{\bbl}(\widetilde{\mathbb{W}},\mathbb{W})$ are $\bbl$-invariant. Moreover, $\bbl$ acts trivially on ${\rm Hom}_{\bbl}(\widetilde{\mathbb{W}},\mathbb{W})$ and ${\rm Con}(\widetilde{\mathbb{W}},\mathbb{W})$ is isomorphic to $\mathbb{W}$ as an $\bbl$-space, and, in particular, it is irreducible. It is also clear that $\widetilde{\mathbb{W}}$ acts trivially on ${\rm Con}(\widetilde{\mathbb{W}},\mathbb{W})$ and $\Aff_{\bbl}(\widetilde{\mathbb{W}},\mathbb{W})/{\rm Con}(\widetilde{\mathbb{W}},\mathbb{W})$. 

Now since $\widetilde{\mathbb{W}}'$ is a proper $\bbl$-subspace of $\widetilde{\mathbb{W}}$, there is $0\neq f\in{\rm Hom}_{\bbl}(\widetilde{\mathbb{W}},\mathbb{W})$ such that $\widetilde{\mathbb{W}}'\subseteq \ker(f)$. Thus, by the definition of $\psi$, $\psi(l,\tilde{w}')(f)=f,$
 for any $(l,\tilde{w}')\in \bbl\ltimes\widetilde{\mathbb{W}}'$. Now assume that, for some $w\in \mathbb{W}$ and $f_{{\rm lin}}\in {\rm Hom}_{\bbl}(\widetilde{\mathbb{W}},\mathbb{W})$, $f_{w}(x)=f_{{\rm lin}} (x)+w$ is $\bbg$-invariant. Then $f_w$ is a constant function as it is invariant under any translation, i.e. $f_{{\rm lin}}=0$. Hence $w$ is $\bbl$-invariant and so $w$ is also $0$. Therefore this representation satisfies all the desired properties.

Now, we assume that the projection of $\bbh$ to $\bbl$ is surjective and the group generated by the projection of $\bbh$ to $\bbu$ generates $\bbu$.
By Corollary~\ref{c:NormalSubgroupPerfect}, $\bbh$ is not a normal subgroup of $\bbg$.
(In fact, Corollary \ref{c:NormalSubgroupPerfect} was stated for finite groups,
however, the proof works verbatim for algebraic groups, as well.)
Hence, again,  in the representation $\tilde{\rho}=\wedge^{\dim \hfr} \Ad$, $[w_0]=\wedge^{\dim \hfr} \hfr\in \widetilde{\mathbb{W}}$ is $\bbh$-invariant, but it is not $\bbg$-invariant.  We claim that $\bbh$ does not have any character, and therefore $w_0$ is fixed by $\bbh$. Let $\chi$ be a character of $\bbh$; then $\bbh\cap\bbu\subseteq \ker(\chi)$ as $\bbu\cap\bbh$ is a unipotent group. So $\chi$ factors through a character of $\bbh/(\bbh\cap\bbu)\simeq \bbl$. Since $\bbl$ is semisimple, $\chi$ is trivial. 

Since $\bbu$ is a unipotent and normal subgroup of $\bbg$, 
\[
\widetilde{\mathbb{W}} \supseteq (\tilde{\rho}(\bbu)-1)(\widetilde{\mathbb{W}}) \supseteq \cdots \supseteq (\tilde{\rho}(\bbu)-1)^{n_0+1}(\widetilde{\mathbb{W}})=0,
\]
and for any $i$, $(\tilde{\rho}(\bbu)-1)^i(\widetilde{\mathbb{W}})$ is a $\bbg$-space. Let $k+1$ be the smallest possible integer such that the projection of $w_0$ to $\widetilde{\mathbb{W}}/\langle(\tilde{\rho}(\bbu)-1)^{k+1}\cdot \widetilde{\mathbb{W}}\rangle$ is not $\bbg$-invariant. Notice that $k$ is definitely positive as $\bbg=\bbh\cdot\bbu$, $w_0$ is $\bbh$-invariant and $\bbu$ acts trivially on $\widetilde{\mathbb{W}}/\langle(\tilde{\rho}(\bbu)-1)\cdot \widetilde{\mathbb{W}}\rangle$. After going to the quotient space, we can and will assume that $\langle(\tilde{\rho}(\bbu)-1)^{k+1}\cdot \widetilde{\mathbb{W}}\rangle=0$, i.e. $\bbu$ acts trivially on $\langle(\tilde{\rho}(\bbu)-1)^{k}\cdot \widetilde{\mathbb{W}}\rangle$.

Let $\widehat{\mathbb{W}}$ be the subspace of $\widetilde{\mathbb{W}}$ such that 
\[
(\widetilde{\mathbb{W}}/\langle(\tilde{\rho}(\bbu)-1)^{k}\cdot \widetilde{\mathbb{W}}\rangle)^{\bbg}=\widehat{\mathbb{W}}/\langle(\tilde{\rho}(\bbu)-1)^{k}\cdot \widetilde{\mathbb{W}}\rangle,
\]
i.e. $\widehat{\mathbb{W}}=\{w\in \widetilde{\mathbb{W}}|\h \forall\h g\in\bbg,\h \tilde{\rho}(g)(w)-w\in \langle(\tilde{\rho}(\bbu)-1)^{k}\cdot \widetilde{\mathbb{W}}\rangle\}$. By the the above argument, $w_0\in \widehat{\mathbb{W}}$. 

Now take a decomposition $\mathbb{W}_1\oplus\cdots\oplus \mathbb{W}_{n}$ of $\langle(\tilde{\rho}(\bbu)-1)^{k}\cdot \widetilde{\mathbb{W}}\rangle$ into irreducible $\bbl$-spaces.
Since $\bbl$ is a semisimple group and it acts trivially on $\widehat{\mathbb{W}}/\langle(\tilde{\rho}(\bbu)-1)^{k}\cdot \widetilde{\mathbb{W}}\rangle$, there is a subspace $\mathbb{W}_0$ such that $\widehat{\mathbb{W}}=\mathbb{W}_0\oplus\mathbb{W}_1\oplus\cdots\oplus\mathbb{W}_n$ and $\bbl$ acts trivially on $\mathbb{W}_0$.
We claim that the projection of $w_0$ to one of the non-trivial irreducible components is non-zero. Otherwise, $w_0$ is $\bbl$-invariant and so it is also invariant by the image of the projection of $\bbh$ onto $\bbu$. By our assumptions on $\bbh$, we conclude that $w_0$ is $\bbg$-invariant, which is a contradiction.
 So, for some $i$,
we have that the projection of $w_0$ to 
\[
\mathbb{W}=\widehat{\mathbb{W}}/\oplus_{j\neq i} \mathbb{W}_j
\]
is not $\bbg$-invariant.
Now let $\mathbb{W}^{(1)}=\oplus_j \mathbb{W}_j/\oplus_{j\neq i} \mathbb{W}_j$ and $\mathbb{W}^{(2)}= \mathbb{W}^{\bbl}/\mathbb{W}^{\bbg}$. Clearly 
\[
\mathbb{W}/\mathbb{W}^{\bbg}=\mathbb{W}^{(1)}\oplus \mathbb{W}^{(2)},
\]
$\bbu$ acts trivially on $\mathbb{W}^{(1)}$ and $\mathbb{W}/\mathbb{W}^{(1)}$, and $\mathbb{W}^{(1)}$ is an irreducible $\bbl$-space; moreover $\bbh$ fixes $\overline{w}_0=w_1+w_2$, where $\overline{w}_0$ is the image of $w_0$ in $\mathbb{W}/\mathbb{W}^{\bbg}$, $w_1\in \mathbb{W}^{(1)}$ and $w_2\in \mathbb{W}^{(2)}$.  Moreover, if $\overline{w}'\in (\mathbb{W}/\mathbb{W}^{\bbg})^{\bbg}$, then $g\cdot w'-w'\in \mathbb{W}^{\bbg}\subseteq \mathbb{W}^{\bbl}=\mathbb{W}^{(2)}$. On the other hand, $\bbg$ acts trivially on $\mathbb{W}/\mathbb{W}^{(1)}$, and so $g\cdot w'-w'\in \mathbb{W}^{(1)}$. Therefore overall, we have that $w'\in \mathbb{W}^{\bbg}$, i.e. $\overline{w}'=0$ as we wished.
\end{proof}
\begin{rmk}\label{r:ProperSubgroups}
From the proof of Proposition~\ref{p:ProperSubgroup}, it is clear that, if $\bbg$ is defined over $\bbq$, then there is a number field $\kappa$ such that all the desired representations $\rho_i$ and $\psi_i$ are defined over $\kappa$. 
\end{rmk}
\begin{lem}\label{l:Unbounded}
Let $\Gamma$ be a Zariski-dense subgroup of  $\bbg\subseteq \mathbb{GL}_{d}$, a Zariski-connected $\bbq$-group, such that $\Gamma\subseteq \bbg(\bbz_{\scal})$. Let $\rho$ be a non-trivial representation of $\bbg$ which is defined over a number field $\kappa$. Then there are $p\in \scal\cup\{\infty\}$ and a place $\pfr$ of $\kappa$  such that, 
\begin{itemize}
\item[1.] $\pfr$ divides $p$, i.e. $\bbq_p$ is a subfield of $\kappa_{\pfr}$.
\item[2.] $\rho(\Gamma)$ is an unbounded subset of $\rho(\bbg(\kappa_{\pfr}))$.
\end{itemize}
\end{lem}
\begin{proof}
If this is not the case, then $\rho(\Gamma)$ is bounded in $\rho(\bbg(\kappa_{\pfr}))$, for any $\pfr\in\pcal(\kappa)$. Hence $\rho(\Gamma)$ is a finite group. On the other hand, $\rho(\Gamma)$ is Zariski-dense in $\rho(\bbg)$. Moreover we know that $\rho(\bbg)$ is Zariski-connected. Thus $\rho$ is trivial which is a contradiction.
\end{proof}
 \begin{proof}[Proof of Proposition~\ref{p:descr}] This is a direct consequence of Proposition~\ref{p:LiftingUp}, Proposition~\ref{p:ProperSubgroup}, Remark~\ref{r:ProperSubgroups} and Lemma~\ref{l:Unbounded}.
 \end{proof}

\subsection{A ping-pong argument}
\label{sc_pingpong}

We recall the notation from Proposition \ref{pr_descr}.
$\G$ is a Zariski-connected perfect $\Q$-group, and $\Ga$ is
a Zariski dense subgroup of $\G(\Q)$.
We are given finitely many irreducible representations $\rho_i$ for $1\le i\le m$
each defined over a local field $\KK_i$, and $\rho_i(\Ga)$ is unbounded.
Furthermore, for each $i$ we are given an affine space $\V_i$ and a morphism
$\f_i:\G\times\V_i\to\Aff(\W_{\rho_i})$.
Often we think about $\f_i$ as homomorphisms from $\G$ to $\Aff(\W_{\rho_i})$
parametrized by the elements of $\V_i$.
Then we also write $\f_i(\cdot,v)=\f_{i,v}(\cdot)$.
We also recall that for any $i$ and $0\neq v\in\V_i(\KK_i)$,
$\f_{i,v}:\G(\KK_i)\to\Aff(W_{\rho_i})$ is a homomorphism
whose linear part is $\rho_i$, and
no point of $W_{\rho_i}$ is fixed under the action of $\G(\KK_i)$.
Our aim in this section is to prove that if we modify our generating set
in an appropriate way, then only a small
fraction of our group satisfy a condition like 3.(a) or 3.(b)
in Proposition
\ref{pr_descr}.

Let $A$ be a subset of a group that generates freely
a subgroup.
A reduced word over $A$ is a product of
the form $g_1\cdots g_l$, where $g_i\in A$ or $g_i^{-1}\in A$
and $g_ig_{i+1}\neq 1$ for any i.
We write $B_l(A)$ (or simply $B_l$) for the set of reduced
words over $A$ of length $l$.
\begin{prp}
\label{pr_pingpong}
Let notation be as above.
Then there is a set
$A\subset\Ga$ generating freely a subgroup $\Ga'$, which satisfies
the following properties.
Write $S'=A\cup\widetilde A$.
Then for any $i$ and for any vector $w\in W_{\rho_i}$, we have
\[
|\{g\in B_l|\rho_i(g)[w]=[w]\}|<|B_l|^{1-c},
\]
furthermore, for any $i$, $v\in\V_{i}(\KK_i)$ with $\|v\|=1$ and for $w\in W_{\rho_i}$ we have
\[
|\{g\in B_l|\f_{i,v}(g)w=w\}|<|B_l|^{1-c},
\]
where $c$ is a constant depending on $S$ and on the representations.
\end{prp}

The rest of this section is devoted to the proof of this proposition and
in Section \ref{sc_proof4} we combine it with Proposition \ref{pr_descr} to
get Proposition \ref{pr_escp}.
First we construct a desirable set of generators.

\begin{prp}
\label{pr_generators}
Let notation be as above.
Then there is a symmetric set $S'\subset\Ga$, and a number
$R>0$ and for each $g\in S'$
and $1\le i\le m$ there are two sets $K_g^{(i)}\subset
U_g^{(i)}\subset W_{\rho_i}$ such that the following
hold.
\begin{enumerate}
\item
For each $i$ and $g$ we have $\rho_i(g)(U_g^{(i)})\subset
K_g^{(i)}$.
\item
For each $i$, any vector $w\neq0\in W_{\rho_i}$ is
contained in $U_g^{(i)}$
for at least two elements $g\in S'$.
\item
For each $i$ and for any elements $g_1,g_2\in S'$ we have
$K_{g_1}^{(i)}\subset U_{g_2}^{(i)}$ unless $g_1g_2=1$.
\item 
For each $i$ and for any elements $g_1,g_2\in S'$ we have
$K_{g_1}^{(i)}\cap K_{g_2}^{(i)}=\emptyset$ unless $g_1=g_2$.
\item
For each $i$, $v\in\V_i(\KK_i)$ with $\|v\|=1$ and $g\in S'$ we have that if
$w\in U_g^{(i)}$ and $\|w\|>R$,
then $\f_{i,v}(g)w\in K_g^{(i)}$ and $\|\f_{i,v}(g)w\|>\|w\|$.
\item
For each $i$, $v\in\V_i(\KK_i)$ with $\|v\|=1$ and $w\in W_{\rho_i}$,
there are at least two elements
$g_1,g_2\in S'$
such that $\f_{i,v}(g_1)w\neq w$ and $\f_{i,v}(g_2)w\neq w$.
\end{enumerate}
\end{prp}

The construction of the set $S'$ relies on the notion of
quasi-projective transformation introduced by Furstenberg \cite{Fur} and
further studied by Goldsheid and Margulis \cite{GM} and Abels, Margulis,
and Soifer \cite{AMS}.
We use a slightly different notion also used by Cano and Seade \cite{CS},
which suits better our purposes.
Let $b\in \Mat_d(\KK)$ be a not necessarily invertible linear
transformation, where $\KK$ is a local field.
Write $V^+(b)=\Im(b)$ and $V^-(b)=\Ker(b)$.
Denote by $\P(\KK^d)$ the projective space, and in general we denote
by $\P(\cdot)$ the projectivization of a concept.
Then $\P(b):\P(\KK^d)\sm \P(V^-(b))\to\P(\KK^d)$ is a partially defined map
on the projective space, and we call it a quasi-projective transformation.

Consider a sequence $\{g_i\}_{i=1}^\infty\subset \GL_d(\KK)$.
It is easy to see (see e.g. \cite[Proposition 2.1]{CS})
that it contains a subsequence,
still denoted by $\{g_i\}_{i=1}^\infty$
such that
\be
\label{eq_qpt}
\lim_{i\to \infty}g_i/\|g_i\|=b
\ee
uniformly for some linear transformation $b:\KK^d\to \KK^d$.
Here and everywhere below
$\|\cdot\|$ denotes a fixed submultiplicative matrix-norm.
Moreover, this implies that
\[
\lim_{i\to \infty}\P(g_i)=\P(b)
\] 
uniformly on compact subsets of $\P(\KK^d)\sm\P(V^-(b))$.
Let $\Ga\le \GL_d(\KK)$ be a group and
denote by $\overline \Ga$ the set of maps $b$ for which (\ref{eq_qpt})
holds for some sequence $\{g_i\}_{i=1}^\infty\subseteq \Ga$.
The following lemma, crucial for us, is statement b, in
\cite[Lemma 4.3]{AMS}.
For completeness we give the proof.
\begin{lem}
\label{lm_AMS}
Let $\{g_i\}_{i=1}^\infty,\{h_i\}_{i=1}^\infty\subseteq \GL_d(\KK)$
be two sequences such that
\be
\label{eq_unif}
\lim_{i\to \infty}g_i/\|g_i\|=b_1\quad {\rm and}\quad
\lim_{i\to \infty}h_i/\|h_i\|=b_2
\ee
for some linear transformations
$b_1,b_2$.
If $b_1b_2\neq0$, then
\[
\lim_{i\to \infty}g_ih_i/\|g_ih_i\|=\l b_1b_2
\]
for some nonzero $\l\in \KK$.
\end{lem}
\begin{proof}
Since the convergences in (\ref{eq_unif}) are uniform, we
have
\[
\lim_{i\to\infty}\frac{g_i}{\|g_i\|}\circ\frac{h_i}{\|h_i\|}=b_1b_2.
\]
Observe the following:
If $\{\g_i\}_{i=1}^\infty\subseteq \GL_d(\KK)$
and $\g_i/\l_i$ converge to a nonzero linear transformation
for a sequence of scalars
$\{\l_i\}_{i=1}^\infty\subset \KK$, then $\g_i/\|\g_i\|$
is convergent, too.
This proves the lemma.
\end{proof}

This lemma implies that
if $b_1,b_2\in\overline \Ga$, and $b_1b_2\neq0$, then
$\l b_1b_2\in\overline \Ga$.
This property is crucial for us.
Denote by $r$ the minimum of the ranks of the elements in $\overline \Ga$.
If $b_1\in\overline \Ga$ is of rank $r$ and if
$V^+(b_1)\cap V^-(b_2)\neq \{0\}$
for some $b_2\in\overline \Ga$,
then $b_1b_2=0$, whence $V^+(b_1)\subset V^-(b_2)$.

We will use the following lemma to construct the first element
of $S'$.
This lemma is a variant of \cite[Lemma 5.15]{AMS}.
\begin{lem}
\label{lm_contractive}
Let $\G$ be a Zariski-connected algebraic group, and let
$\rho_1,\ldots,\rho_m$ be irreducible representations defined
over local fields $\KK_i$.
Let $\Ga\le \G(\Q)$ be Zariski-dense.
For each $1\le i\le m$ denote by $r_i$ the minimal rank of
an element in $\overline{\rho_i (\Ga)}$.

Then for each $1\le i\le m$ there is a $b_i\in\overline{\rho_i (\Ga)}$
and there is a sequence of elements $\{h_j\}_{j=1}^\infty\subseteq \Ga$
such that the following hold.
\begin{enumerate}
\item
For each $1\le i\le m$, $V^+(b_i)\cap V^-(b_i)=\{0\}$ and
$\dim V^+(b_i)=r_i$.
\item
For each $1\le i\le m$, we have
\[
\lim_{j\to\infty}\rho_i(h_j)/\|\rho_i(h_j)\|=b_i.
\]
\end{enumerate}
\end{lem} 
\begin{proof}
Let $1\le k\le m$ and assume that
$\{h_j\}_{j=1}^\infty\subseteq \Ga$ is a sequence
such that for each $i$ we have
$\rho_i(h_j)/\|\rho_i(h_j)\|\to b_i$ for some
linear transformation $b_i$, and 1. holds for $i<k$.
We show below that we can replace $\{h_j\}_{j=1}^\infty$
with another sequence such that 1. holds for $i=k$ as well.
Then the Lemma follows by induction.

Let $\{h_j'\}_{j=1}^\infty\subseteq \Ga$ be a sequence such that
$\rho_k(h_j')/\|\rho_k(h_j')\|\to b_k'$, where $b_k'$ is a linear
transformation of rank $r_k$.
By taking a subsequence we can assume that
$\rho_i(h_j')/\|\rho_i(h_j')\|\to b_i'$ for some linear transformation
$b_i'$ for all $1\le i\le m$.
Take two elements $g_1,g_2\in \Ga$.
We consider the sequence
\[
\{\wt h_j=g_1h_j'g_2h_j\}_{j=1}^\infty
\]
By Lemma \ref{lm_AMS} we get that for all $1\le i\le m$ we have
\[
\rho_i(\wt h_j)/\|\rho_i(\wt h_j)\|\to \l_i \rho_i(g_1)b_i'\rho_i(g_2)b_i
\]
provided $\rho_i(g_1)b_i'\rho_i(g_2)b_i\neq 0$.
Then for each $i$, there is a nonempty
Zariski-open subset $X_i$ of $\G(\KK_i)$ such that for
$g_2\in X_i$ we have
\[
\rho_i(g_2)(V^+(b_i))\nsubseteq V^-(b_i').
\]
The Zariski-openness is clear and non-emptiness follows from the
irreducibility of $\rho_i$.
(A more detailed argument for a similar statement will be given
in the proof of Lemma \ref{lm_Zopen1}).
Now take $g_2\in \Ga\cap\bigcap X_i$.
Then
$\rho_i(g_1)b_i'\rho_i(g_2)b_i\neq 0$
no matter how we choose $g_1$.
Similarly, there is a nonempty
Zariski-open subset $X_i'$ of $\G(\KK_i)$ such that for
$g_1\in X_i'$, we have
\be
\label{eq_g1}
\rho_i(g_1)(V^+(b_i'\rho_i(g_2)b_i))\nsubseteq V^{-}(b_i'\rho_i(g_2)b_i).
\ee
Take $g_1\in \Ga\cap\bigcap X_i'$.
For $i\le k$, the rank of $\rho_i(g_1)b_i'\rho_i(g_2)b_i$
is $r_i$, and then (\ref{eq_g1})
implies that
\[
V^{+}(\rho_i(g_1)b_i'\rho_i(g_2)b_i)\cap V^-(\rho_i(g_1)b_i'\rho_i(g_2)b_i)
=\{0\}
\]
by the remarks after Lemma \ref{lm_AMS}, which we wanted to show.
\end{proof}

Let $\{g_i\}_{i=1}^n\subseteq \GL_d(\KK)$ be a sequence
such that
\[
\lim_{i\to\infty}g_i/\|g_i\|=b\quad{\rm and}\quad
\lim_{i\to\infty}g_i^{-1}/\|g_i^{-1}\|=\widetilde b
\]
for some non-invertible $b,\widetilde b\in \Mat_d(\KK)$.
Let $w\in V^{+}(b)$ and assume to the contrary that
$w\notin V^-(\widetilde{b})$.
Then there is some vector $u_1\in \KK^d$ such that
\[
\lim_{i\to \infty}g_i(u_1)/\|g_i\|=w.
\]
By uniform convergence, we then have
\[
\lim_{i\to \infty}\frac{g_i^{-1}(g_i(u_1)/\|g_i\|)}
{\|g_i^{-1}\|}=u_2.
\]
for some nonzero $u_2\in \KK^d$.
This implies that $\|g_i\|\cdot\|g_i^{-1}\|$ is bounded
which contradicts to the non-invertibility of $b$.
Therefore we can conclude that $V^{+}(b)\subset V^-(\widetilde b)$
and $V^+(\widetilde b)\subset V^-(b)$.

In the proof of Proposition \ref{pr_generators}
we will use Lemma \ref{lm_contractive}
to produce an element $g_0\in\Ga$ with certain nice properties,
and then we will define $A$ to be a set of appropriate conjugates of it,
whom we will find using the following two lemmata.
\begin{lem}
\label{lm_Zopen1}
Let $\G$ be a Zariski-connected algebraic group defined over
a local field $\KK$, and
let $\rho$ be an irreducible representation of it.
Let $V_1^+,V_1^-,V_2^+,V_2^-\subseteq W_\rho$ be subspaces
such that $V_1^+\cap V_1^-=V_2^+\cap V_2^-=\{0\}$ and
$V_1^+\subseteq V_2^-$ and $V_2^+\subseteq V_1^-$.
Let $M$ be an integer and denote by $X\subseteq \G(\KK)^M$
the set of $M$-tuples $(g_1,\ldots,g_M)$
such that the following hold.
If we have $\rho(g_\a)(V_i^+)\subseteq \rho(g_\b)(V_j^-)$
for some $1\le i,j\le2$
and $1\le\a,\b\le M$, then $\a=\b$ and $i+j=3$.
Then $X$ is a nonempty Zariski-open set.
\end{lem}
\begin{proof}
Let $v_1,\ldots,v_r$ be a basis for $V_1^+$ and let
$\psi_1,\ldots,\psi_r$ be a basis for the space of functionals
vanishing on $V_{2}^-$.
Then the condition
$\rho(g_1)(V_1^+)\subseteq \rho(g_2)(V_2^-)$ is equivalent
to the equations
$\langle\rho(g_1)v_j,\rho(g_2)^*\psi_i\rangle=0$ for $1\le i,j,\le r$.
The other conditions can be described in terms of algebraic
equations similarly, whence the Zariski-openness follows.

It is clear that there is an $M$-tuple $(g_1,\ldots,g_M)$
for which
the single condition $\rho(g_1)(V_1^+)\nsubseteq \rho(g_2)(V_2^-)$
is satisfied.
For example we can take $g_2=1$
pick a vector $w_1\in V_1^+$ and choose $g_1$ in such a way that
$\rho(g_1)w_1\notin V_2^-$,
the existence of $g_1$ follows from irreducibility.
It is a similar argument to show that the other constraints can
be satisfied, so $X$, being the intersection of finitely many
nonempty Zariski-open sets, is nonempty.
\end{proof}

\begin{lem}
\label{lm_Zopen2}
Let $\G$ be a Zariski-connected algebraic group, and
let $\rho$ be an irreducible representation of it.
Let $\V$ be an affine space and
let $\f:\G\times\V\to \Aff(\W_\rho)$ be a morphism such that the linear part of $\f_v$ is $\rho$.
Assume that for some $0\neq v\in\V$ and $w\in W_\rho$ there is an element
$g\in\G(\KK)$ such that $\f_v(g)w\neq w$.

Then for $M\ge 2\dim(W_\rho)+\dim(\V)+1$,
there is a nonempty Zariski-open set $X\subset \G(\KK)^{M}$
such that if $(g_1,\ldots,g_{M})\in X$ then the following hold.
\begin{enumerate}
\item
Let $w\in W_\rho$ and $W\subsetneq W_\rho$ be a proper linear subspace. Then
for any set of indices $I\subset\{1,\ldots,M\}$ with
$|I|=2\dim(W_\rho)-1$,
there is some $i\in I$ such that $\rho(g_i)w\notin W$.
\item
Let $v\in\V(\KK)$, $w\in W_\rho$ and $W\subset W_\rho$ be an affine subspace, then
for any set of indices $I\subset\{1,\ldots,M\}$ with $|I|=2\dim(W_\rho)+\dim(\V)+1$,
there is some $i\in I$ such that $\f_v(g_i)w\notin W$.
\end{enumerate}
\end{lem}
\begin{proof}
We only show that property 1. can be satisfied, 2. is similar,
and then we can take the intersection of the two sets.
Moreover, it is enough to show that 1. can be satisfied for
the index set $I=\{1,\ldots,2\dim(W_\rho)-1\}$.
Consider the algebraic variety
\[
\P(W_\rho)\times\P(W_\rho^*)\times\G(\KK)^M.
\]
Consider also the subvariety
\[
Y=\{([w],[\psi],g_1,\ldots,g_M):\langle\rho(g_i)(w),\psi\rangle=0
\;{\rm for}\;1\le i\le2\dim(W_\rho)-1\}.
\]
By the irreducibility of $\rho$ it follows that for
$([w],[\psi])\in\P(W_\rho)\times\P(W_\rho^*)$ fixed,
the variety
\[
Y_{w,\psi}=\{g\in\G(\KK):\langle\rho(g)(w),\psi\rangle=0\}
\]
is a proper subvariety of $\G(\KK)$, hence
$\dim(Y_{w,\f})\le\dim(\G)-1$.
This implies that the fiber of $Y$ over $([w],[\psi])$ is of codimension
at least $2\dim(W_\rho)-1$ in $\G(\KK)^M$.
Now let $Z$ be the Zariski-closure of the image of $Y$
under the projection map
\[
\P(W_\rho)\times\P(W_\rho^*)\times\G(\KK)^M\to\G(\KK)^M.
\]
Then $\dim(Z)\le\dim(Y)$, hence $Z$ is a proper subvariety,
and by construction its complement satisfies 1.
for $I=\{1,\ldots,2\dim(W_\rho)-1\}$.
\end{proof}

\begin{proof}[Proof of Proposition \ref{pr_generators}]
If $\rho$ is a representation of $\G$, then  write $\widetilde\rho$
for the representation that associates the transpose inverse of
$\rho(g)$ for every $g\in\G$.
Apply Lemma \ref{lm_contractive}
to the representations $\rho_1,\ldots,\rho_m,
\widetilde \rho_1,\ldots,\widetilde\rho_m$.
We get a sequence $\{h_i\}_{i=1}^\infty\subset\Ga$
and linear transformations
$b_1^{(1)},\ldots,b_m^{(1)},b_1^{(2)},\ldots,b_m^{(2)}$
with the following properties.
For each $1\le i\le m$, we have
\[
\lim_{j\to\infty}\rho_i(h_j)/\|\rho_i(h_j)\|=b_i^{(1)}\quad{\rm and}\quad
\lim_{j\to\infty}\rho_i(h_j^{-1})/\|\rho_i(h_j^{-1})\|=b_i^{(2)}.
\]
Furthermore we have that
$\dim(V^+(b_i^{(1)}))=\dim(V^{+}(b_i^{(2)}))=r_i$
and $V^+(b_i^{(j)})\cap V^-(b_i^{(j)})=\{0\}$.
Here and everywhere below $r_i$ denotes the minimal rank of the
elements of $\overline{\rho_i(\Ga)}$.
By the remarks preceding Lemma \ref{lm_Zopen1}
we see that $V^+(b_i^{(j)})\subset V^-(b_i^{(3-j)})$
for $j=1,2$.
Let $d$ be the maximum of the dimensions of the representation
spaces $W_{\rho_i}$ and parameter spaces $\V_i$.
Apply Lemma \ref{lm_Zopen1}
with $M=3d+2$ for each $\rho_i$ and for the subspaces
$V_j^+=V^+({b_i^{(j)}})$ and $V_j^-=V^-(b_i^{(j)})$, $j=1,2$.
This way we get Zariski-open subsets $X_i\subset\G(\KK_i)^M$.
Also apply Lemma \ref{lm_Zopen2} for the representations $\rho_i$
and for the morphisms $\f_i$, this gives
Zariski-open subsets $X_i'\subset\G(\KK_i)^M$.
Since $\Ga$ is Zariski dense,
we get elements $g_1,\ldots,g_{M}\in\Ga$
such that $(g_1,\ldots, g_M)\in X_i\cap X_i'$ for all $i$, hence
they have the following properties.
Recall that if $c_1,c_2\in\overline{\rho_i(\Ga)}$, and
$\dim(V^+(c_1))=r_i$,
then either $V^+({c_1})\cap V^-(c_2)=\{0\}$ or
$V^+(c_1)\subset V^-(c_2)$.
For each $i$ we have
\[
\rho_i(g_\a)(V^+(b_i^{(j)}))\cap \rho_i(g_\b)(V^-({b_i^{(k)}}))=\{0\}
\]
for every $1\le j,k\le2$
and $1\le\a,\b\le M$, except for $\a=\b$ and $i+j=3$.
Using 1. in Lemma \ref{lm_Zopen2} with
$W=V^-(b_i^{(1)})$, we also have that
\[
\rho_i(g_{\a_1})(V^-(b_i^{(1)}))\cap\ldots\cap
\rho_i(g_{\a_{2d-1}})(V^-(b_i^{(1)}))=\{0\}
\]
for any $1\le\a_1<\ldots<\a_{2d-1}\le M$.

We show that if we set $A=\{g_1h_jg_1^{-1},\ldots,g_Mh_jg_M^{-1}\}$
and $j$ is large enough then we can choose the sets
$K_g^{(i)}$ and $U_g^{(i)}$ in such a way that the proposition
holds.
At this point we fix $i$, and omit the corresponding indices
everywhere.

For a set $X\subset\P(\KK^d)$ denote by $B_\e(X)$ the set
of points which are of distance at most $\e$ from $X$
with respect to any fixed metric which induces the standard topology
on $\P(\KK^d)$.
Let $\e>0$ be sufficiently small, so that
when $V_1,\ldots V_l$ are among the subspaces
$\rho(g_k)(V^\pm({b^{(j)}}))$, then
\[
B_\e(\P(V_1))\cap\ldots\cap B_\e(\P(V_l))\neq\emptyset
\]
only if
$\P(V_1)\cap\ldots\cap\P(V_l)\neq\emptyset$.
For $g=g_kh_jg_k^{-1}\in A$
define
\bean
K_g&=&\{w\in \KK^d\sm\{0\}|[w]\in B_\e(\P(\rho(g_k)(V^+({b^{(1)}}))))\}
\quad{\rm and}\quad\\
U_g&=&\{w\in \KK^d\sm\{0\}|w]\notin B_\e(\P(\rho(g_k)(V^-({b^{(1)}}))))\}.
\eean
We define $K_g$ and $U_g$ in a similar manner for $g\in\widetilde A$,
but we use $b^{(2)}$ instead of $b^{(1)}$.
Properties 2., 3. and 4. can be deduced immediately from the properties
of the spaces $\rho(g_k)(V^+(b^{(1)}))$ and
$\rho(g_k)(V^-({b^{(1)}}))$ provided $\e$ is small enough.
We remark that property 4. follows from property 3., since by construction
$K_g\cap U_{g^{-1}}=\emptyset$.

Property 1. holds if $j$ is large enough, since $\P(\rho(g_kh_jg_k^{-1}))$
converges to $\P(\rho(g_k)b^{(1)}\rho(g_k^{-1}))$ uniformly
on compact subsets of $\P(W_\rho)\sm\P(V^-({b^{(1)}}))$.
For property 5, we note that there is a constant $c>0$ depending
on $\e$ such that 
\[
\|\rho(g_kh_jg_k^{-1})(w)\|>c\|\rho(g_kh_jg_k^{-1})\|\cdot\|w\|
\]
for $w\in U_g$.
Since $\lim_{j\to\infty}\|\rho(g_kh_jg_k^{-1})\|=\infty$,
we have $c\|\rho(g_kh_jg_k^{-1})\|>1$ for large $j$.
Then for $\|w\|>R$ large and $v\in\V(\KK)$, $\|v\|=1$, the translation component of
$\f_v(g_kh_jg_k^{-1})$ is negligible compared to the
linear part, and property 5 follows.
Here we used that the unit ball in $\V(\KK)$ is compact, hence for $j$ fixed, the
translation part of $\f_v(g_kh_jg_k^{-1})$ is bounded in $v$.
Finally, denote by $W\subset W_\rho$ the possibly empty affine subspace
that consist of the fixed points of $\f_v(h_j)$.
Then the set of fixed points of $\f_v(g_kh_jg_k^{-1})$ is
$g_k(W)$.
From this we see that property 6. follows from part 2. of Lemma
\ref{lm_Zopen2}.
\end{proof}

\begin{proof}[Proof of Proposition \ref{pr_pingpong}]
Let $A\subset\Ga$ be the set that we constructed in Proposition
\ref{pr_generators}, and let $K_g^{(i)}$ and $U_g^{(i)}$
be the corresponding sets.
In what follows we fix $i$ and omit the corresponding indices.

We deal with the two parts separately, first we estimate
the size of the set $\{g\in B_l|\rho(g)[w]=[w]\}$.
For $0\le k<l$ denote by $X_k$ the set of those
reduced words $g_l\cdots g_1\in B_l$ for which
$\rho(g_k\cdots g_1)w\in U_{g_{k+1}}$, and $k$ is the
smallest index with this property.
We write $X_l$ for those words which are not contained in any of the $X_k$
with $k<l$.
Let $g_l\cdots g_1\in X_k$.
We remark that by the properties of the sets $U_g$ and $K_g$,
we have
\[
\rho(g_{k+1}\cdots g_{1})w\in K_{g_{k+1}}\subset U_{g_{k+2}}.
\]
In fact, by induction we can conclude that
$\rho(g_j\cdots g_{1})w\in K_{g_j}$ 
for $j>k$.
Assume further that $\rho(g_l\cdots g_{1})[w]=[w]$.
Then we also have
\[
\rho(g_{j+1}^{-1}\cdots g_{l}^{-1})[w]
=\rho(g_j\cdots g_{1})[w]\in \P(K_{g_j})
\]
for $j>k$.
Since the sets $K_g$ are disjoint, we see that $[w]$ determines
$g_j$ uniquely for $j>k$.
Indeed, once $g_l,\ldots,g_{j+1}$ are known, they determine which of the
sets $\P(K_{g_j})$ does $\rho(g_{j+1}^{-1}\cdots g_{l}^{-1})[w]$ belong to.
On the other hand we know that for $j\le k$, we have
$\rho(g_{j-1}\cdots g_1)w\notin U_{g_{j}}$.
Since $\rho(g_{j-1}\cdots g_1)w$ is covered by at least two of
the sets $U_g$, we have at most $|S'|-2$ possibilities for
$g_j$.
Therefore we have
\[
|\{g\in B_l|\rho(g)[w]=[w]\}\cap X_k|\le (|S'|-2)^k,
\]
from where the first part of the proposition follows easily.

Now we give an estimate for $|\{g\in B_l|\f_v(g)w=w\}|$.
We show that there is an integer $k$ such that
for any $v\in\V(\KK)$ with $\|v\|=1$, $w\in W_\rho$ and $g\in S'$
there is a reduced word $\o\in B_k$ of
length $k$ with the following properties.
The first letter of $\o$ is not $g$, and
we have $\|\f_v(\o)w\|>R$, $\|\f_v(\o)w\|>\|w\|$ and
$\f_v(\o)w\in K_{g'}$, where $g'$ is the last letter of $\o$.
If $|w|>R$, this is easy, since there are at least two letters
$g'\in S'$ such that $w\in U_{g'}$.
We can also make an existing word longer, since we can preserve the
required properties no matter how we continue it as long as it stays
reduced.
Consider the case when $\|w\|\le R$.
Denote by $D\subset W_\rho$ the solid ball of radius $R$.
To each reduced word $\o$ we associate a set
$E_{\o}\subset\{v\in\V(\KK):\|v\|=1\}\times W_\rho$
defined by
\[
E_{\o}=\{(v,x):x\in\f_v(\o^{-1})(D)\}.
\]
We need to show that there is a number $k$ such that
\[
\bigcap_{\o\in B_k :\h\o {\rm\: does\: not\: contain\:  g}}E_{\o}=\emptyset.
\]
Since the $E_{\o}$ are compact, it is enough to show that
\[
\bigcap_{\o:\h \o {\rm\: does\: not\: contain\: g}}E_{\o}=\emptyset.
\]
We show that for each $v_0\in\V(\KK)$ with $\|v_0\|=1$,
\[
(\{v_0\}\times W_\rho)\cap\bigcap_{\o:\h\o {\rm\: does\: not\: contain\: g}}E_{\o}=\emptyset
\]
Assume to the contrary that there are at least two points
\[
(v_0,w_1),(v_0,w_2)\in\{v_0\}\times W_\rho\cap
\bigcap_{\o:\h\o {\rm\: does\: not\: contain\: g}}E_{\o}.
\]
Then $w_1-w_2\in U_h$ for some $g\neq h\in S'$.
Property 5 in Proposition \ref{pr_generators} implies that there is some
$c>1$ such that $\|\rho(h)w\|>c\|w\|$
for all $w\in U_h$.
Then $2R>\|\f_{v_0}(h^l)w_1-\f_{v_0}(h^l)w_2\|>c^l\|w\|$, a contradiction.
Assume to the contrary that
\[
\{v_0\}\times W_\rho\cap\bigcap_{\o:\h\o {\rm\: does\: not\: contain\: g}}E_{\o}=\{(v_0,w)\}
\]
for a point $w\in W_\rho$.
Then $w$ is fixed by all elements of $S'$ except maybe for $g$,
which contradicts to property 6 in Proposition \ref{pr_generators}.
So far we showed all required properties of $\o$ except that $\f_v(\o)w$ belong
to the right set $K_{g'}$.
However, by properties 2 and 5 in Proposition \ref{pr_generators} there are at
least two letters that we can append to $\o$ to fulfill these last requirement as well.
For at least one of these two, the word stays reduced.

Now consider a reduced word $g_l\cdots g_1$ for which
$\f_v(g_l\cdots g_1)w=w$.
Then we also have for all $1\le j<l/k$
\be
\label{eq_fixed}
\f_v(g_{jk}\cdots g_{1}g_l\cdots g_{jk+1})
(\f_v(g_{jk}\cdots g_1)w)=\f_v(g_{jk}\cdots g_1)w.
\ee
The above argument shows that out of the $(|S'|-1)^k$
possibilities for $g_{(j+1)k}\cdots g_{jk+1}$, there
is at least one for which (\ref{eq_fixed}) does not hold
since the vector on the left hand side is longer than the one on the right.
Although $g_{jk}\cdots g_{1}g_l\cdots g_{jk+1}$ may not
be reduced, if $j<l/k-1$, we still get a reduced word ending
with $g_{(j+1)k}\cdots g_{jk+1}$ after all possible reductions.
This shows that
\[
|\{g\in B_l|\f_v(g)w=w\}|
<\left(\frac{(|S'|-1)^k-1}{(|S'|-1)^k}\right )^{l/k-2}|B_l|
\]
giving the second half of the proposition.
\end{proof}

\subsection{Proof of Proposition \ref{pr_escp}}
\label{sc_proof4}

We show that the proposition holds if
$S'$ is the set of generators constructed
in Proposition \ref{pr_pingpong}.
Let $q$ be a square-free integer and $H<\pi_q(\Ga)$ a subgroup.
Denote by $q_1$ the product of those prime factors of $q$
which are large enough so that Propositions \ref{pr_descr}
and \ref{pr_pingpong} hold.
Let $H_1=\pi_{q_1}(H)<\pi_{q_1}(\Ga)$.
Then clearly
\[
[\pi_{q_1}(\Ga):H_1]\ge q_1/q[\pi_q(\Ga):H]\gg [\pi_q(\Ga):H],
\]
hence if we show the claim for $q_1$ and $H_1$, it will
follow for $q$ and $H$ as well with a worse implied constant.
Form now on we assume that $q=q_1$.

Let $H^\sharp$ be the subgroup corresponding to $H$ in Proposition
\ref{pr_descr}.
Let $l\le c_1\log [\pi_q(\Ga):H]$ be an integer,
where $c_1$ is a sufficiently small constant.
Let $h\in\Ga$ be such that $\pi_q(h)\in H^\sharp$ and $h\in B_{l}$,
where $B_{l}$ is the set of reduced words of length
$l$ over the alphabet $S'$.
If $c_1$ is small enough, then $\|h\|<[\pi_q(\Ga):H^\sharp]^\d$
with the same $\d$ for which Proposition \ref{pr_descr}
holds.
Then by definition, $h\in\LL_\d(H^\sharp)$.
If we combine Propositions \ref{pr_descr} and \ref{pr_pingpong},
then we get
\[
|B_{l}\cap\LL_\d(H^\sharp)|<|B_{l}|^{1-c_2}
\]
for some $c_2>0$.

Write $|S'|=2M$
Set $P_k(l)=\chi_{S'}^{(2k)}(\o)$, where $\o\in B_l$.
Since $|B_l|=2M(2M-1)^{l-1}$ for $l\ge1$,
\be
\label{eq_psi}
1=P_k(0)+\sum_{l\ge1}|B_l|P_k(l).
\ee
By a result of Kesten \cite[Theorem 3.]{Kes}, we have
\[
\limsup_{k\to\infty} (P_k(0))^{1/k}=(2M-1)/M^2.
\]
From general properties of Markov chains (see \cite[Lemma 1.9]{Woe})
it follows that
\[
P_k(0)\le\left(\frac{2M-1}{M^2}\right)^k.
\]
Since $\chi_{S'}^{(2k)}$ is
symmetric, we have $P_k(0)=\sum_g[\chi_{S'}^{(k)}(g)]^2$, hence
$P_k(l)\le P_k(0)$ for all $l$
by the Cauchy-Schwartz inequality.
Now we can write for $k\le c_1\log q/2$:
\bean
\chi_{S'}^{(2k)}(\LL_\d(H^\sharp))&=&\sum_l|B_l\cap \LL_\d(H^\sharp)|P_k(l)\\
&\le&\sum_l|B_l|^{1-c_2}P_k(l)\\
&\le&\sum_{l\le k/10}(2M)^{l}\left(\frac{2M-1}{M^2}\right)^k\\
&&+(2M-1)^{-c_2k/10}
\sum_{l\ge k/10}|B_l|P_k(l)\\
&<&\frac{(2M)^{11k/10+1}}{M^{2k}}+(2M-1)^{-c_2k/10}.
\eean
The inequality between the third and fourth lines follows form
(\ref{eq_psi}).

Note that $\pi_q[\chi_{S'}^{(2k)}](H^\sharp)$
is non-increasing with $k$.
Let $g\in\pi_q(H)$, since $S'$ is symmetric, we have
\[
\pi_q[\chi_{S'}^{(2k)}](gH^\sharp)^2\le \pi_q[\chi_{S'}^{(4k)}](H^\sharp).
\]
Thus
\[
\pi_q[\chi_{S'}^{(2k)}](H)\le C^n\left(\pi_q[\chi_{S'}^{(4k)}](H^\sharp)\right)^{1/2},
\]
and this finishes the proof, since any positive power of $q$ dominates $C^n$ if $q$
is large enough.

\section{Growth of product-sets}
\label{sc_product}

In this section we aim to prove Proposition \ref{pr_flatening},
but first we recall some results we will need later on.
We fix a square-free integer $q=p_1\cdots p_n$, and assume that
each prime divisor $p_i$ of $q$ is bigger than some large but fixed
constant.
Then as we saw at the beginning of Section \ref{sc_escp},  we have
\[
G:=\pi_q(\Ga)=G_{p_1}\times\ldots\times G_{p_n},
\]
where $G_{p_i}=\pi_{p_i}(\Ga)=L_{p_i}\ltimes U_{p_i}$ and $L_{p_i}$
is a product of quasi-simple groups generated by their elements
of order $p_i$ and $U_{p_i}$ is a $p_i$-group.
Furthermore, we have that $U_{p_i}/[U_{p_i}:U_{p_i}]$
is isomorphic to $(\F_{p_i}^{d_i},+)$ for some integers $d_i$
which are bounded independently of $q$.
We write $L=L_{p_1}\times\ldots\times L_{p_n}$
and $U=U_{p_1}\times\ldots\times U_{p_n}$.
The following result is a  statement about abstract groups
satisfying certain assumptions that we will recall later
in Section \ref{sc_assump}.
We will also show there that the group $L$
satisfy these assumptions with parameters that are
independent of $q$.

\begin{prpb}[{\cite[Proposition 14]{Var}}]
Let $G$ be a group satisfying (A0)--(A5).
For any $\e>0$ there is a $\d>0$ depending only on $\e$
and the constants in (A0)--(A5) such that the following holds.
If $A\subseteq G$ is symmetric such that
\[
|A|<|G|^{1-\e}\quad {\rm and}\quad \chi_A(gH)<[G:H]^{-\e}|G|^{\d}
\]
for any $g\in G$ and any proper $H<G$,
then $|\prod_3 A|\gg|A|^{1+\d}$.
\end{prpb}

Assumptions (A0)--(A5) are more or less straightforward to check except for
(A4) which basically boils down to showing Proposition B
for quasi-simple groups that are the direct factors of $L_{p_i}$.
This statement was first proved by Helfgott for $\SL_2(\Z/p\Z)$
\cite{Hel} and for $\SL_3(\Z/p\Z)$ \cite{Hel2}, and later
it was extended by Dinai \cite{Din2}
for $\SL_2(\F_q)$ for an arbitrary finite field
$\F_q$.
Now the statement is known for all finite simple groups of Lie type
due to a recent breakthrough by Breuillard, Green, Tao
\cite{BGT}, and Pyber, Szab\'o \cite{PS}.
We can either use \cite[Theorem 4]{PS} or \cite[Corollary 2.4]{BGT}.
The statement in the formulation of \cite[Theorem 4]{PS} is the following.
\begin{thmc}
Let $L$ be a simple group of Lie type of rank $r$ and $A$ a generating set of $L$.
Then either $\Pi_3 A=L$ or $|\Pi_3 A|\gg |A|^{1+\vare_0}$, where
$\vare_0$ and the implied constant depend only on $r$.
\end{thmc}

The following useful Lemma is based on the Balog-Szemer\'edi-Gowers Theorem
and it is implicitly contained in
\cite{BG1}.
\begin{lemd}[{\cite[Lemma 15]{Var}}]
Let $\mu$ and $\nu$ be two probability measures on an arbitrary
group $G$ and let $K>2$ be a number.
If
\[
\|\mu*\nu\|_2>\frac{\|\mu\|_2^{1/2}\|\nu\|_2^{1/2}}{K}
\]
then there is a symmetric set $A\subset G$ with
\[
\frac{1}{K^R\|\mu\|_2^2}\ll |A|\ll \frac{K^R}{\|\mu\|_2^2},
\]
\[
|{\textstyle\prod_3 A}|\ll K^R|A|\quad {\rm and}
\]
\[
\min_{g\in A}\left(\wt \mu *\mu\right)(g)\gg \frac{1}{K^R|A|},
\]
where $R$ and the implied constants are absolute.
\end{lemd}

Using the Lemma it is very easy to deduce Proposition
\ref{pr_flatening} from the following

\begin{prp}
\label{pr_product}
Let $\G$ be a Zariski-connected
perfect algebraic group defined over $\Q$.
Let $\Ga<\G(\Q)$ be a finitely generated Zariski-dense subgroup.
Then for any $\e>0$, there is some $\d>0$ depending only on
$\e$ and $\G$
such that the following holds.
Let $q$ be a square-free integer without small prime factors.

If $A\subseteq G_q$ is symmetric such that
\[
|A|<|G_q|^{1-\e}\quad {\rm and}\quad \chi_A(gH)<[G_q:H]^{-\e}|G_q|^{\d}
\]
for any $g\in G_q$ and any proper $H<G_q$,
then $|\prod_3 A|\gg|A|^{1+\d}$.
\end{prp}

We defer the proof to the following sections, and now we show how it implies the

\begin{proof}[Proof of Proposition \ref{pr_flatening}]
Assume that the conclusion of the proposition fails, i.e.
that there is an $\e$ such that for any $\d$, there
is a $q$ and there are probability measures $\mu$ and $\nu$ with
\[
\|\mu\|_2>|G_q|^{-1/2+\e}\quad{\rm and}\quad
\mu(gH)<[G_q:H]^{-\e}
\]
for any $g\in G_q$ and for any proper $H<G_q$, and yet
\[
\|\mu*\nu\|_2\ge\|\mu\|_2^{1/2+\d}\|\nu\|_2^{1/2}.
\]
Take $K=\|\mu\|_2^{-\d}$ in Lemma D.
Note that by the third property of the set $A$,
we have
\[
\chi_A(gH)\ll K^R\wt \mu*\mu(gH)\le K^R\max_{h\in G_q} \mu(hH)
\ll|G_q|^{R\d}[G_q:H]^{-\e}.
\]
Now $|\prod_3 A|\ll K^R|A|$ contradicts Proposition \ref{pr_product},
if $\d$ is small enough.
In fact, when $q$ contains small prime factors,
Proposition \ref{pr_product} does not apply, but we still get a
contradiction for $\pi_{q'}(A)$ and the group $G_{q'}$, where
$q'$ is the product of the prime factors of $q$ which are not
too small for Proposition \ref{pr_product}.
Also note that when $q$ is smaller than a fixed constant
we can get the contradiction by the trivial inequality
$|\prod_3 A|\ge|A|+1$.
\end{proof}

\subsection{Growth in the unipotent radical}

As mentioned before, Proposition B and Theorem C
together imply Proposition \ref{pr_product}, when
$\G$ is semisimple.
When the unipotent radical $\U$ is nontrivial, we need to do some
work which is carried out in this section.
Recall the definition of $G$ and $L$ from the beginning of Section
\ref{sc_product}.
Denote by $\rm pr$ the projection homomorphism $G\to L$.

The purpose of this section is to prove the following

\begin{prp}
\label{pr_unipotent}
For every $\e>0$ there is an integer $C$ such that the following holds.
Let $A\subseteq G$ be a symmetric set such that ${\rm pr}(A)=L$,
and
\[
\chi_A(gH)<[G:H]^{-\e}|G|^{1/C}
\]
for any $g\in G$ and any proper $H<G$.
Then $\pi_{q_1}[\prod_C A]= G_{q_1}$ for some $q_1>q^{1-\e}$.
\end{prp} 

We need a couple of  Lemmata.
Let $\wh G=\wh G_1\times\ldots\times\wh G_n$ be a direct product of groups.
For each $i$, let $\b_i:\wh G_i\to \wh L$
be a given homomorphism into a group $\wh L$.
Denote by $\b:\wh G\to\wh L$
the homomorphism induced by the $\b_i$ in the obvious way.
For each $i$ write $\pr_i: \wh G\to \wh G_i$
for the projection homomorphisms.
We introduce the following distance for two elements $g_1,g_2\in \wh G$:
\be
\label{eq_dist}
d(g_1,g_2)=\sum_{i:\pr_i(g_1)\neq\pr_i(g_2)}\log |\Ker(\b_i)|.
\ee

\begin{lem}
\label{lm_Farah}
Let $A\subseteq \wh G$ be a symmetric set with $\b(A)=\wh L$ and $1\in A$.
Assume that for every $g\in A.A.A$ with $\b(g)=1$ we have
$d(1,g)\le\e\log |\Ker (\b)|$ for some $\e>0$.
Then $A$ can be covered with at most $2^n|\Ker (\b)|^{25\e}$
cosets of a subgroup of $\wh G$ of order at most $|\wh L|$.
\end{lem}

We will apply this Lemma in the following setting:
We will have normal subgroups $N_{p_i}\lhde U_{p_i}$ which are normal
in $G_{p_i}$ as well, and we will set $\wh G_i={G_{p_i}}/N_{p_i}$ and $\wh L=L$.
The homomorphism $\b_i$ will be the projection  $L_{p_i}\ltimes(U_{p_i}/N_{p_i})\to L_{p_i}$.
The purpose of the lemma is to find an element $g$ in a product-set of $A$
which is in the kernel of $\pr$, but have a large conjugacy class.
In a subsequent lemma we will recover the normal subgroup generated by $g$ in the
product-set of a bounded number of copies of $A$.
This will allow us to increase $N_{p_i}$, and proceed to the next step of the iteration.

\begin{proof}[Proof of Lemma \ref{lm_Farah}]
Let $\psi:\wh L\to \wh G$ be a map such that $\b\circ\psi=Id$, and
$\psi(\wh L)\subseteq A$, this is possible due to the assumption $\b(A)=\wh L$.
By assumption, we have $d(g^{-1}\psi(\b(g)),1)<\e\log |\Ker (\b)|$ for any $g\in A$,
hence
\be
\label{eq_graph}
d(\psi(\b(g)),g)<\e\log |\Ker (\b)|.
\ee
Moreover, for any $g,h\in \wh L$, we have
\bean
d(\psi(g)\psi(h),\psi(gh))<\e\log |\Ker (\b)|\quad{\rm and}\\
d(\psi(g)^{-1},\psi(g^{-1}))<\e\log |\Ker (\b)|.
\eean
These two inequalities mean that $\psi$ is an
$\e |\Ker(\b)|$-homomorphism of type II with respect to $d$
in the sense of Farah, see \cite[Section 1]{Far}.
Then by \cite[Theorem 2.1]{Far}, there is a homomorphism $\f:\wh L\to \wh G$
such that
\[
d(\psi(g),\f(g))<24\e\log |\Ker (\b)|
\]
for every $g\in \wh L$.
Combining this with (\ref{eq_graph}), we get that for every element
$g\in A$, there is $h\in\f(\wh L)$ such that $d(g,h)<25\e\log |\Ker (\b)|$.
By definition, this means that there is an index set
$I\subset\{1,\ldots,n\}$ such that $\pr_i(g)=\pr_i(h)$ if $i\notin I$
and $\prod_{i\in I} |\Ker(\b_i)|<|\Ker(\b)|^{25\e}$.
For a fixed $I$, the elements $g$ which satisfy this
condition can be covered by at most $|\Ker(\b)|^{25\e}$
cosets of the group $\f(\wh L)$.
This proves the claim, because we have $2^n$ possibilities for $I$.
\end{proof}

As already promised, we show that we can recover the normal subgroup
generated by the element constructed in the previous lemma.
We need to introduce more notation.
Let $p_1,\ldots, p_n$ be primes and with the notation as above, assume that
$\wh G_i$
is generated by its $p_i$-elements.
Assume that
$\wh G_i=\wh L_i\ltimes\wh U_i$ is a semidirect product and $\wh L=\wh L_1\times\ldots\times\wh L_n$
and that the kernel of $\b_i$ is $\wh U_i$.
We assume further that $\wh U_i$ is isomorphic to $(\F_{p_i}^{d_i},+)$.
Then writing $\wh U_i$ additively,
we can associate to it an $\F_{p_i}$-vector space $M_i$ which is
an $\wh L_i$-module such that the action $v\mapsto g\cdot v$ of $g\in \wh L_i$
on $M_i$ descends from the conjugation action $u\mapsto gug^{-1}$ of $g\in\wh G_i$
on $\wh U_i$.
We note that since $\wh U_i$ is commutative, its action by conjugation is trivial on itself,
so the above is well-defined.
Write
$\wh U =\wh U_1\times\ldots\times\wh U_n $ .

\begin{lem}
\label{lm_normal}
Let $p_1,\ldots, p_n, \wh G_i,\wh L_i, \wh U_i$ and $M_i$ satisfy the above assumptions.
Furthermore, assume that no $M_i$ contain a one dimensional composition factor.
Let  $A\subseteq\wh G$ be
a symmetric set with $1\in A$ and ${\b}(A)=\wh L$ and let
$g\in A$ be any element with ${\b}(g)=1$.
Denote by $N$ the smallest normal subgroup of $\wh{G}$
that contains $g$.

Then there is a constant $c$ depending only on $\max d_i$ such that
\[
\textstyle\prod_{c}A\supseteq N.
\]
\end{lem}

The proof of Lemma \ref{lm_normal} requires
\begin{lem}\label{l:OrbitSums}
Let $p$ be a prime and
let $H\subseteq \GL(M)$ be a group generated by its $p$-elements,
where $M$ is a vector space of dimension $d$ over $\bbf_p$ and $p\gg d$.
Assume that no non-zero vector is fixed by $H$, i.e. the trivial representation
is not a sub-representation of $M$.
Then there is a constant $c=c(d)$, only depending on $d$, such that
$\sum_c  H\cdot v$ contains a non-zero $H$-subspace, for any $0\neq v\in M$.
\end{lem}
Here and everywhere below $H\cdot v$ denotes the orbit of $v$ under the action of $H$ on $M$.
\begin{proof}
Let $g\in H$ be an element of order $p$.
Then $x=g-1\in {\rm End}(M)$ is a nilpotent element.
Let $k$ be the largest integer such that $x^k$ is not zero.
It is clear that $k$ is at most $d$.
So
\begin{align}
\notag \sum\!_{_{2^d}} H \ni &(g^j-1)^k=((1+x)^j-1)^k&&\forall \h 0\le j\le p-1\\
\notag =& \left( x\sum_{i=0}^{j-1}\binom{j}{i+1}x^i \right)^k=j^k x^k &&\text{ since $x^{k+1}=0$.}
\end{align} 
Since any element of $\bbf_p$ can be written as the sum of at most $k$ $k$-th powers,
we have that $\sum_{d2^d} H\supseteq \bbf_p x^k$.
Thus $\sum_{d2^d} H\supseteq \bbf_p Hx^k H$, and therefore 
\[
\sum\!_{_{d^32^d}}H\supseteq \langle x^k \rangle,
\]
where $ \langle x^k \rangle$ is the ideal generated by
$x^k$ in $\mathcal{A}=\bbf_p[H]$, the $\bbf_p$-span of $H$ in ${\rm End}(M)$. 

Now, we prove the lemma by induction on $d$.
If $\mathcal{A}\cdot v$ is a proper $H$-subspace, we get the claim by the induction hypothesis.
If not, then $\langle x^k \rangle\cdot v$ is a non-zero $H$-subspace of
$\sum_{d^32^d}H\cdot v$, as we wished.
Note that, if $\langle x^k \rangle\cdot v=0$ and $\mathcal{A}\cdot v=M$,
then $\langle x^k\rangle \cdot M=0$,
which is a contradiction as $M$ is a faithful $\mathcal{A}$-module.
\end{proof}
\begin{cor}\label{c:OrbitSumsNoTrivialFactor}
Let $p$ be a prime and let $H\subseteq \GL(M)$ be a group generated by its $p$-elements,
where $M$ is a vector space of dimension $d$ over $\bbf_p$ and $p\gg d$.
Assume that none of the composition factors of $M$ is one-dimensional.
Then there is a constant $c=c(d)$, depending only on $d$,
such that $\sum_c H\cdot v$
is equal to the $H$-subspace generated by $v$.
\end{cor}
\begin{proof}
Using Lemma~\ref{l:OrbitSums}, one can easily prove this, by induction on $d$. 
\end{proof}

\begin{proof}[Proof of Lemma \ref{lm_normal}] 
Since $\wh U=\wh U_1\times \cdots \times \wh U_n$ is a normal subgroup of $\wh G$,
we get a homomorphism $\theta$ from $\wh G$ to $\Aut(\wh U)$, $\wh G$ acting on $\wh U$
by conjugation.
As $\wh U$ is commutative, $\theta$ factors through $\wh L$ and we get back the action of
$\wh L_i$ on $\wh U_i$, for any $i$.
Moreover, $\theta$ commutes with the projection homomorphisms $\pr_i$.

Let $g=(u_1,\ldots,u_n)$, where $u_i\in \wh U_i$, for any $i$.
Denote by $N_i$ the normal subgroup generated by $u_i$ in $\wh G_i$.
Translating Corollary \ref{c:OrbitSumsNoTrivialFactor} to the language
of multiplicative groups, we get a constant $c$, depending only on $\max d_i$, such that
\[
\textstyle\prod_c\{h(u_1,\ldots,u_n)h^{-1}:h\in A\}\supseteq N_1\times\cdots\times N_n.
\]
It is clear that $N_1\times\cdots\times N_n$ is a normal subgroup of $\wh G$ which contains $g$.
Thus 
\[
\textstyle\prod_{3c} A= \textstyle \prod_c A.A.\wt A \supseteq N,
\]
as we wished.
\end{proof}

The above lemmata allows us to deal with the case when $U$ is commutative.
In the general case we will work with $U/[U,U]$, and recover a subset of $U$
which projects onto $U/[U,U]$.
Then we will use the following lemma to recover $U$.
This lemma is very similar to the main idea behind the papers \cite{GS} and \cite{Din}.

\begin{lem}
\label{lm_nilpotent}
Let $\wh U$ be a finite $k$-step nilpotent group generated by $m$
elements,
and let $A\subseteq\wh U$ be a subset such that
$A.[\wh U,\wh U]=\wh U$.
Then $\prod_{C(k,m)} A=\wh U$.
\end{lem}
\begin{proof}
Consider the lower central series
$\wh U=\Ga_1\rhd \Ga_2\rhd\ldots\rhd \Ga_{k+1}=\{1\}$
defined by $\Ga_{i+1}=[\wh U,\Ga_i]$.
Then for $1\le i\le k$, $K_{i}=\Ga_{i}/\Ga_{i+1}$ is
a commutative group.
It is well known (see \cite[Corollary 1.12]{War})
that for any $i,j$ we have
$[\Ga_{i},\Ga_j]\subseteq \Ga_{i+j}$
and for any $x,y,z\in U$ we have the identities
(see \cite[equations 1.4 and 1.5]{War}):
\bean
[x,yz]=[x,z][x,y]^z\\
{}[xy,z]=[x,z]^y[y,z],
\eean
where $x^y=y^{-1}xy$.
Therefore the maps
\[
\f_{i}:K_1\times K_i\to K_{i+1}
\]
defined by
\[
\f_i(g\Ga_2,h\Ga_{i+1})=[g,h]\Ga_{i+2}
\]
are well-defined, and they are homomorphisms in both variables.

We show that for any $i$
\be
\label{eq_nilpotent}
\textstyle\prod_m \f_i(K_1,K_i)=K_{i+1}.
\ee
Let $x_1,\ldots,x_m$ be generators for $K_1$.
Then any element of $K_{i+1}$ is of the form
\[
\f_i(x_1^{a_{1,1}}\cdots x_m^{a_{1,m}},y_1)\cdots
\f_i(x_1^{a_{l,1}}\cdots x_m^{a_{l,m}},y_l)
\]
for some $a_{\cdot,\cdot}\in\Z$ and $y_j\in K_i$.
Using that $\f_i$ is a homomorphism in the first variable,
we can expand this, then we can collect the factors containing
$x_k$ using the commutativity of $K_{i+1}$ and finally we
use that $\f_i$ is a homomorphism in the second variable
and get that the above is equal to
\[
\f_i(x_1,y_1^{a_{1,1}}\cdots y_l^{a_{l,1}})\cdots
\f_i(x_m,y_1^{a_{1,m}}\cdots y_l^{a_{l,m}}).
\]
This proves the claim.

To prove the lemma, we note that the above
claim implies that if $(\prod_C A)\Ga_{i+1}\supseteq \Ga_{i}$, then
\[
(\textstyle\prod_{m(2C+2)}A)\Ga_{i+2}\supseteq\Ga_{i+1}.
\]
This proves the lemma by induction, and approximating
an element successively in $\wh U/\Ga_i$ for larger and larger
values of $i$.
\end{proof}
 
\begin{proof}[Proof of Proposition \ref{pr_unipotent}]
Let $A,G,L,U$ and $\pr$ be the same as in the proposition.
First we prove the proposition in the case, when $U$ is commutative.
Then each $U_{p_i}$ is isomorphic to $(\F_{p_i}^{d_i},+)$ for some
integers $d_i$.
Denote $d=\max d_i$. 
For each $p_i$, we give a sequence of normal subgroups
\[
\{1\}=N_{p_i}^{(0)}\lhde N_{p_i}^{(1)}\lhde\ldots
\lhde N_{p_i}^{(l)}\lhde U_{p_i}
\]
such that each of them is a normal subgroup in $G_{p_i}$ as well.
Write $N^{(m)}=N_{p_1}^{(m)}\times\ldots\times N_{p_n}^{(m)}$
They will satisfy the following properties:
\bea
\label{eq_N1}
(\textstyle\prod_{C_1}A)N^{(k-1)}\supseteq N^{(k)}\\
\label{eq_N2}
[N^{(k)}:N^{(k-1)}]\ge[U:N^{(k-1)}]^{\e/100d}\\
\label{eq_N3}
[U:N^{(l)}]<q^{\e},
\eea
where $C_1$ is a constant depending only on $d$.

Assume that $m>0$, and $N_{p_i}^{(k)}$ is defined for $k<m$
and they satisfy (\ref{eq_N1}) and (\ref{eq_N2}).
If $[U:N^{(m-1)}]<q^{\e}$, then we can set $l=m-1$, and we are done.
Assume the contrary.
To apply Lemma \ref{lm_Farah}
we take $\wh G_i=G_{p_i}/N_{p_i}^{(m-1)}$, $L=\wh L$ and we let
$\b_i: \wh G_i\to \wh L$ be the homomorphism induced by $\pr$.
Consider the group $\wh G=G/N^{(m-1)}$ and the set
$\overline A=AN^{(m-1)}\subset \wh G$.
By assumption, $A\subset G$ cannot be covered
with less than $|\Ker(\b)|^{\e}|G|^{-1/C}$ cosets of a subgroup
of $G$ of index $|\Ker(\b)|$.
We can assume that $C$ is so large that
$|G|^{1/C}<|\Ker(\b)|^{\e/2}$ and $q$ is so large that
$2^n<|\Ker(\b)|^{\e/4}$.
Then
$\overline A$ cannot be covered by less than $2^n|Ker(\b)|^{\e/4}$
cosets of a subgroup of $\wh G$ of order $|\wh L|$.
Using Lemma \ref{lm_Farah} we find an element
$g\in\overline A.\overline A.\overline A$ such that
$\b(g)=1$ and $d(g,1)>\e\log|\Ker(\b)|/100$,
where $d(\cdot,\cdot)$ is defined by (\ref{eq_dist}).
Let $N_{p_i}^{(m)}$ be the smallest normal subgroup of $G_{p_i}$
that contains $N_{p_i}^{(m-1)}$ and $\pi_{p_i}(g)$.
Then (\ref{eq_N1}) follows from Lemma \ref{lm_normal} applied for
$\wh G=G/N^{(m-1)}$ and $A=\overline A.\overline A.\overline A$.
However, we need to check the condition that the $L_{p_i}$-modules $M_i$
defined in the paragraph preceding Lemma \ref{lm_normal} do not contain one dimensional composition
factors.
Suppose the contrary.
We can assume that one of the $M_i$ contains the trivial representation as a sub-representation,
actually for this purpose we might need to enlarge $N^{(m-1)}$.
By \cite[Theorem A]{Gur}, it follows that $M_i$ is completely reducible, hence we can
write $M_i=M_i'\oplus M_i''$ as the sum of $L_{p_i}$-modules such that the action on $M_i'$ is trivial.
Then there is a proper normal subgroup $N\lhd U_{p_i}$ of $G_{p_i}$ corresponding to $M_i''$
such that $G_{p_i}$ acts trivially on $U_{p_i}/N$, hence $G_{p_i}/N$ is isomorphic
to the direct product $L_{p_i}\times(U_{p_i}/N)$.
This contradicts to the assumption that $\G$ and hence $G_p$ is perfect.

Let $q'$ be the product of primes $p_i$ for which
$N_{p_i}^{(m-1)}\neq N_{p_i}^{(m)}$.
Then 
\[
q'\ge e^{d(g,1)/d}>|\Ker(\b)|^{\e/100d},
\]
since $|\Ker(\b_i)|\le p_i^d$.
The groups $U_{p_i}$ are $p_i$-groups, hence
$[N^{(m)}:N^{(m-1)}]\ge q'$.
This implies (\ref{eq_N2}), since $[U:N^{(k-1)}]=|\Ker(\b)|$.
Therefore we proved equations (\ref{eq_N1})--(\ref{eq_N3}).

Equations (\ref{eq_N1}) and (\ref{eq_N3}) together imply that
there is an integer $q_1>q^{1-\e}$ such that
\[
\pi_{q_1}(\textstyle\prod_{lC_1}A)\supseteq U_{q_1}.
\]
Since $\pr(A)=L$ and $G_{q_1}=L_{q_1}U_{q_1}$, this proves the
proposition when $U$ is commutative.

In the general case, running the above argument for the group
$G/[U,U]$, we get
\[
\pi_{q_1}(\textstyle\prod_{lC_1}A).[U_{q_1},U_{q_1}]\supseteq U_{q_1}.
\]
Then Lemma \ref{lm_nilpotent}
applied for the group $\wh U=U_{q_1}$ and for the set
$\pi_{q_1}(\textstyle\prod_{lC_1}A)\cap U_{q_1}$ finishes the proof.
\end{proof}
It is worth mentioning that the result from~\cite{Gur} depends on the classification of the finite simple groups. But we do not really need this result as the involved representations are coming from a representation over $\bbq$ and therefore, for large enough $p$, the picture modulo $p$ is similar to the picture over $\bbq$.

\subsection{Assumptions (A0)--(A5) for $\bbl(\Z/q\Z)$}
\label{sc_assump}

We list the assumptions  mentioned
in Proposition B.
When we say that something depends on the constants appearing
in the assumptions (A1)--(A5) we mean $C$ and the function
$\d(\e)$ for which (A4) holds.

\begin{itemize}
\item[(A0)]
$L=L_1\times\cdots\times L_n$ is a direct product, and
the collection of the factors satisfy (A1)--(A5) for some
sufficiently large constant $C$.

\item[(A1)]
There are at most $C$ isomorphic copies of the same group in the
collection.

\item[(A2)]
Each $L_i$ is quasi-simple and we have
$|Z(L_i)|<C$.

\item[(A3)]
Any non-trivial representation of $L_i$ is of dimension at least
$|L_i|^{1/C}$.

\item[(A4)]
For any $\e>0$, there is a $\d>0$ such that the following holds.
If $\mu$ and $\nu$ are probability measures on $L_i$ satisfying
\[
\|\mu\|_2>|L_i|^{-1/2+\e}
\quad{\rm and}\quad
\mu(gH)<|L_i|^{-\e}
\] 
for any $g\in L_i$ and for any proper $H<L_i$, then
\be
\label{eq_assump2}
\|\mu*\nu\|_2\ll\|\mu\|_2^{1/2+\d}\|\nu\|_2^{1/2}.
\ee

\item[(A5)]
For some $m<C$, there are classes $\HH_0,\HH_1,\ldots,\HH_m$ of
subgroups of $L_i$ having the following properties.
\begin{itemize}

\item[$(i)$]
$\HH_0=\{Z(L_i)\}$.

\item[$(ii)$]
Each $\HH_j$ is closed under conjugation by elements of $L_i$.

\item[$(iii)$]
For each proper $H<L_i$, there is an $H^\sharp\in \HH_j$, for some $j$,
with $H\la_{C} H^\sharp$.

\item[$(iv)$]
For every pair of subgroups $H_1,H_2\in\HH_j$, $H_1\neq H_2$,
there is some $j'<j$ and $H^\sharp\in\HH_{j'}$
for which $H_1\cap H_2\la_{C} H^{\sharp}$.
\end{itemize}
\end{itemize}

One may think about (A5) that there is a notion for dimension of
the subgroups of $L_i$. 

In this section, we will check these assumptions. In the beginning of Section~\ref{sc_escp}, we have already checked (A1) and (A2). By a result of V.~Landazuri and G.~Seitz~\cite{LS}, we also know that (A3) holds.

Assume that (A4) does not hold for $L_i$; then there is an $\vare$ such that for any $\d$ one can find probability measures on $L_i$ with the following properties:
\begin{align}
\|\mu\|_2&>|L_i|^{-1/2+\vare},\label{e:NotUniform}\\
\mu(gH)&<|L_i|^{-\vare},&&\forall\h g\in G,\h\forall\h H\lneq G, \label{e:Generator}\\
\|\mu*\nu\|_2&\ge \|\mu\|_2^{1/2+\d} \|\nu\|_2^{1/2}. \label{e:Contrary}
\end{align}
One can easily see that $\vare$ is less than $1/2$. Choose $\d$ such that $R\d$ is less than $\vare/(1/2-\vare)$ and $2\vare_0/(1+\vare_0)$, where $\vare_0$ is the constant from Theorem C and $R$ is the constant from Lemma D.

By Lemma D  and (\ref{e:Contrary}), there is a symmetric subset $A$ of $L_i$ such that,
\begin{align}
\|\mu\|_2^{-2+\d'}\ll |A| &\ll \|\mu\|_2^{-2-\d'},\label{e:SizeControl}\\
|\Pi_3 A|&\ll \|\mu\|_2^{-\d'} |A|,\label{e:NoExpansion}\\
\min_{s\in A} (\tilde{\mu}*\mu)(s) &\gg \frac{1}{\|\mu\|_2^{-\d'}|A|},\label{e:LargeWeight}
\end{align}
where $\d'=R\d$ and the implied constants are absolute.

First we claim that $A$ is a generating set of $L_i$. If not, it generates a proper subgroup $H$. Hence, on one hand, we have that
\begin{align}
\notag(\tilde{\mu}*\mu)(A)&\le (\tilde{\mu}*\mu)(H)=\sum_{h\in H}(\tilde{\mu}*\mu)(h)\\
\notag&=\sum_{h\in H}\sum_{g\in L_i}\mu(g)\mu(gh)\\
&=\sum_{g\in L_i} \mu(g)\mu(gH)<|L_i|^{-\vare}&&\text{by (\ref{e:Generator})}.\label{e:WeightUpperBound}
\end{align}
On the there hand
\begin{align}
\notag(\tilde{\mu}*\mu)(A)&\gg \|\mu\|_2^{\d'}&&\text{by (\ref{e:LargeWeight}),}\\
&>|L_i|^{\d'(-1/2+\vare)}&&\text{by (\ref{e:NotUniform})}\label{e:WeightLowerBound}.
\end{align}
We get a contradiction, by (\ref{e:WeightUpperBound}), (\ref{e:WeightLowerBound}) and $\d'<\vare/(1/2-\vare)$.

Assume $\Pi_3 A\neq L_i$; then since $A$ is a generating set of $L_i$, by Theorem C, $|\Pi_3 A|\gg |A|^{1+\vare_0}$. Hence, by (\ref{e:NoExpansion}) and (\ref{e:SizeControl}), we have that
\begin{align}
\|\mu\|_2^{(-2+\d')\vare_0}\ll |A|^{\vare_0}\ll\|\mu\|_2^{-\d'}.\label{e:SizeBounds}
\end{align}
We get a contradiction by (\ref{e:SizeBounds}) and $\d'<2\vare_0/(1+\vare_0)$.

So we have $\Pi_3 A=L_i$. By (\ref{e:NoExpansion}), (\ref{e:SizeControl}) and (\ref{e:NotUniform}), we have
\begin{align}
\notag|L_i|=|\Pi_3 A|\ll \|\mu\|_2^{-\d'} |A| \ll \|\mu\|_2^{-2-2\d'}< |L_i|^{(1/2-\vare)(2+2\d')},
\end{align}
which is a contradiction by $\d'<\vare/(1/2-\vare)$. Overall we showed that (A4) holds for $L_i$.

As $L_i$ is a quasi-simple finite group over a finite field which is of a bounded degree extension of its prime field, property (A5) is a direct consequence of~\cite[Proposition 24]{Var}.

\subsection{Proof of Proposition \ref{pr_product}}

Let $N$ be a constant such that $|G_p|<p^N$ for all $p$.
We consider two cases.
The first case is when $|\pr(A)|<q^{-\e/10NC}|L|$, where $C$
is a constant such that any nontrivial representation of
$L_p$ is of dimension at least
$p^{1/C}$ (cf. assumption (A3) in section \ref{sc_assump}).
As we have seen in section \ref{sc_assump}, the group $L$
satisfies the assumptions (A0)--(A5), hence Proposition B
is applicable.
Then we get $|\prod_3\pr(A)|\gg|\pr(A)|^{1+\d}$.
By the pigeonhole principle $A$ contains at least
$|A|/|\pr(A)|$ elements of a coset of $\Ker(\pr)$.
Then it follows that $|\prod_4 A|\gg|\pr(A)|^\d|A|$.
Note that $|\pr(A)|>|L|^\e|G_q|^{-\d}$ by the assumption we made
in the proposition on the set $A$.
This proves the proposition in the first case (see (\ref{eq_Hel}) below).

Now we consider the second case, i.e. when $|\pr(A)|\ge q^{-\e/10NC}|L|$.
Then by \cite[Lemma 5.2]{BGS}, there is a set $A'\subset \pr(A)$
and integers $K_j$ such that for every $g\in A'$ we have
\[
|\{x\in L_{p_j}:\exists h\in A'\;{\rm s.t.}\;
\pi_{p_1\cdots p_{j-1}}(h)=\pi_{p_1\cdots p_{j-1}}(g)\;{\rm and}\;
\pi_{p_j}(h)=x\}|=K_j
\]
and the integers $K_j$ satisfy
\[
|A'|=\prod K_j\ge[\prod(2 \log p_j)^{-1}]|A|.
\]
Denote by $q_2$ the product of primes $p_j$
for which $K_j\ge p_j^{-1/3C}|L_j|$.
Then $q/q_2<q^{\e/3N}$, if all the primes $p_j$ are sufficiently
large.
By a theorem of Gowers \cite{Gow} (see also \cite[Corollary 1]{NP})
it follows that if $B_1,B_2,B_3\subseteq L_{p_j}$ are
sets with $|B_i|\ge p_j^{-1/3C}|L_j|$, $i=1,2,3$, then
$B_1.B_2.B_3=L_{p_j}$.
This implies that
\[
\pr(\pi_{q_2}(A.A.A))=L_{q_2}.
\]
For more details see the argument on \cite[pp. 26]{Var}.
Now using Proposition \ref{pr_unipotent} for the set
$\pi_{q_2}(A.A.A)$ we get an integer $q_1|q_2$ with
$q_1>q^{1-\e/2N}$
such that $\pi_{q_1}[\prod_C A]=G_{q_1}$ for some constant $C$
independent of $q$.
Thus $|\prod_C A|>q^{-\e/2}|G_q|$.
It is a general fact (see the proof of \cite[Lemma 2.2]{Hel}) that
\be
\label{eq_Hel}
|\textstyle\prod_{C}A|<\left(\frac{|A.A.A|}{|A|}\right)^{C-2}|A|
\ee
whenever $A$ is a symmetric set in a group.
This finishes the proof.

\section{Proof of Theorem \ref{th_gengrp}}
\label{sc_proof}
\subsection{Necessity}

Let us first show the necessary part. Let $\bbg$ be the Zariski-closure of $\Gamma$. Denote by  $\bbg^{\circ}$, the connected component of $\bbg$, and let $\Gamma^{\circ}=\bbg^{\circ}\cap \Gamma$. It is clear that $\Gamma^{\circ}$ is a normal finite-index subgroup of $\Gamma$, and so $\Gamma^{\circ}$ is also generated by a finite set $S^{\circ}$. We start by showing that $\gcal(\pi_q(\Gamma^{\circ}),\pi_q(S^{\circ}))$ form expanders as $q$ runs through square free integers  with large prime factors assuming $\gcal(\pi_q(\Gamma),\pi_q(S))$ form expanders. To this end, first we show that $\Gamma^{\circ}$ is a ``congruence" subgroup, i.e. it contains a congruence kernel $\Gamma(q)=\ker(\Gamma \xrightarrow{\pi_q} G_q)$ if the prime factors of $q$ are large enough. To prove this claim, we  notice that $\bbg^{\circ}$ and the quotient map $\iota:\bbg \rightarrow \bbg/\bbg^{\circ}$ are defined over $\bbq$. Hence $\iota(\Gamma(q))=(\iota(\Gamma))(q)$ for any $q$ with large prime factors. On the other hand, since $\bbg/\bbg^{\circ}$ is a finite $\bbq$-group, $(\iota(\Gamma))(q)=1$ for any $q$ with large prime factors, which completes the argument of the our claim.  Now it is pretty easy to show that $\gcal(\pi_q(\Gamma^{\circ}),\pi_q(S^{\circ}))$ form expanders as $q$ runs through square free integers  with large prime factors. For the sake of completeness we present one argument: it is well-known that our desired condition holds if and only if the Haar measure is the only finitely additive $\Gamma^{\circ}$-invariant measure on $\widehat{\Gamma}^{\circ}$, where $\widehat{\Gamma}^{\circ}$ is the profinite completion of $\Gamma^{\circ}$ with respect to $\{\Gamma^{\circ}\cap \Gamma(q)\}$. By the above discussion $\widehat{\Gamma}^{\circ}$ is a finite-index open subgroup of $\widehat{\Gamma}$ the profinite closure of $\Gamma$ with respect to $\{\Gamma(q)\}$; thus one can easily deduce our claim. As a consequence we get a uniform upper bound on $|\pi_q(\Gamma^{\circ})/[\pi_q(\Gamma^{\circ}),\pi_q(\Gamma^{\circ})]|$. On the other hand, $[\bbg^{\circ},\bbg^{\circ}]$ and the quotient map $\iota':\bbg^{\circ}\rightarrow \bbg^{\circ}/[\bbg^{\circ},\bbg^{\circ}]$ are defined over $\bbq$. Hence again we have that $\iota'$ and $\pi_q$ commute with each other for any $q$ with large prime factors.  Thus one can complete the proof of the necessary part using the facts that $\Gamma^{\circ}$ is Zariski-dense in $\bbg^{\circ}$ and $\bbg^{\circ}$ does not have any proper open subgroup.

\subsection{Sufficiency}

Next we show that the condition that the connected component
of the Zariski closure of $\Ga$ is perfect is sufficient for the Cayley graphs
to form a family of expanders.
The argument which shows this using Propositions \ref{pr_escp} and
\ref{pr_flatening} is based on the ideas of Sarnak and Xue \cite{SX}
and Bourgain and Gamburd \cite{BG1} and it is common to all of the papers
\cite{BG1}--\cite{BGS} and \cite{Var}.
In the previous section we have already remarked
that $\Ga^{\circ}=\Ga\cap\G^\circ$ is finitely generated.
Using Proposition \ref{pr_escp}
for $\Ga^\circ$, we get a symmetric set $S'\subset\Ga^\circ$ such that
if $q$ is square-free and coprime to the denominators of the entries in the elements of $S$,
$H\leq\pi_q(\Ga)$  and $l$ is an integer with $l>\log q$, then
\[
\pi_q[\chi_{S'}^{(l)}](H)\ll[\pi_q(\Ga^\circ):H]^{-\d}.
\]

We show that the Cayley graphs $\GG(\pi_q(\Ga^\circ),\pi_q(S'))$ are expanders and later
we will see that this implies the statement of the theorem.
Denote by $T=T_q$ the convolution operator
by $\chi_{\pi_q(S')}$ in the regular representation of $\pi_q(\Ga^\circ)$.
I.e. we write $T(\mu)=\chi_{\pi_{q}(S')}*\mu$ for $\mu\in l^2(\pi_q(\Ga^\circ))$.
We will show that there is a constant $c<1$  independent of $q$
such that the second largest eigenvalue of $T$ is less than $c$.
By a result of Dodziuk \cite{Dod}; Alon \cite{Alo}; and Alon
and Milman \cite{AM} (see also \cite[Theorem 2.4]{HLW})
this then implies that $\GG(\pi_q(\Ga^\circ),\pi_q(S'))$ is a family of expanders.

Consider an eigenvalue $\l$ of $T$, and let $\mu$ be a corresponding
eigenfunction.
Consider the irreducible representations of $\pi_q(\Ga^\circ)$; these
are subspaces of $l^2(\pi_q(\Ga^\circ))$ invariant under $T$.
Denote by $\rho$ the irreducible representation
that contains $\mu$.
Recall form Section \ref{sc_escp} that
\[
\pi_q(\Ga^\circ)=\prod_{p|q\;{\rm prime}}\pi_p(\Ga^\circ).
\]
We only consider the case when the kernel of $\rho$ does not contain
$\pi_p(\Ga^\circ)$ for any $p|q$, otherwise we can consider the quotient
of $\pi_q(\Ga^\circ)$ by
$\pi_p(\Ga^\circ)$,
and we can replace $q$ by a smaller integer.
Then $\rho$ is the tensor-product of nontrivial representations
of the groups $\pi_p(\Ga^\circ)$, hence the dimension of $\rho$
is at least $|\pi_q(\Ga^\circ)|^\e$ for some $\e>0$ (cf. assumption (A3) in Section
\ref{sc_assump} and note that by Corollary
\ref{c:NormalSubgroupPerfect} the semisimple part can not be
contained in $\Ker (\rho)$).
This in turn implies that the multiplicity of $\l$ in $T$ is at least
$|\pi_q(\Ga^\circ)|^\e$, since the regular representation $l^2(\Ga/\Ga_q)$
contains $\dim(\rho)$ irreducible components isomorphic to $\rho$.

Using this bound for the multiplicity, we can bound $\l^{2l}$
by computing the trace of $T^{2l}$ in the standard basis:
\[
\l^{2l}\le |\pi_q(\Ga^\circ)|^{-\e} \Tr(T^{2l})=|\pi_q(\Ga^\circ)|^{-\e}
|\pi_{q}(\Ga^\circ)|\|\pi_q[\chi_{S'}^{(l)}]\|_2^2,
\]
where $\|\cdot\|_2$ denotes the $l^2$ norm over the finite set
$\pi_q(\Ga^{\circ})$.
This proves the theorem, if we can show that
\be
\label{eq_tbs}
\|\pi_q[\chi_{S'}^{(l)}]\|_2\ll|\pi_q(\Ga^\circ)|^{-1/2+\e/4}
\ee
for some $l\ll \log q$.

First apply Proposition \ref{pr_escp} with $H=\{1\}$.
It gives $\pi_q[\chi_{S'}^{(l)}](1)\ll |\pi_q(\Ga^\circ)|^{-\e}$
for $l>\log q$ and for some $\e>0$.
If $l$ is even then $\pi_q[\chi_{S'}^{(l)}](1)>\pi_q[\chi_{S'}^{(l)}](g)$
for any $g\in\pi_q(\Ga)$ by the Cauchy-Schwartz inequality and the definition
of convolution (recall that $S$ is symmetric).
Then we get the estimate
\[
\|\pi_q[\chi_{S'}^{(l)}]\|_2\ll |\pi_q(\Ga^\circ)|^{-\e/2}.
\]
Observe that if we repeatedly apply Proposition \ref{pr_flatening} for the measures
$\mu=\nu=\pi_q[\chi_{S'}^{(2^kl)}]$, then we get (\ref{eq_tbs})
in finitely many steps.
To justify the use of Proposition \ref{pr_flatening}, we remark that since $S'$ is symmetric,
we have
\[
\left(\pi_q[\chi_{S'}^{(2^kl)}](gH)\right)^2\le\pi_q[\chi_{S'}^{(2^{k+1}l)}](H)
\]
that can be bounded using Proposition \ref{pr_escp}.
This shows that  $\GG(\pi_q(\Ga^\circ),\pi_q(S'))$ are expanders indeed.

To finish the proof we show the same for the family $\GG(\pi_q(\Ga),\pi_q(S))$.
Write $c=c(\GG(\pi_q(\Ga^\circ),\pi_q(S')))$, recall the definition from the introduction.
Assume that the elements of $S'$ are the product of at most $m$ elements of $S$.
Consider a set $A=V(\GG)=\pi_q(\Ga)$ of vertices with $|A|\le|V(\GG)|/2$, and denote by $N_k(A)$ the set
of vertices that can be joined to an element of $A$ by a path of length at most $k$
in $\GG(\pi_q(\Ga),\pi_q(S))$.
I.e. by definition
\[
N_k(A)=(\textstyle\prod_k S).A.
\]
Also, it is easy to see that $|N_k(A)|\le|S|^{k-1}|\partial A|+|A|$, so
it is enough to give a lower bound on $|N_k(A)|$.
We clearly have
\[
|N_{|\Ga/\Ga^\circ|}(A)|\ge|\Ga/\Ga^\circ|\max_{g\in\pi_q(\Ga)}|A\cap g\pi_q[\Ga^{\circ}]|.
\] 
This finishes the proof if say $|A\cap\pi_q(\Ga^\circ)|<|A|/(2|\Ga/\Ga^\circ|)$
or if $|A\cap\pi_q(\Ga^\circ)|>3|\pi_q(\Ga^\circ)|/4$.
If both of these inequalities fail, then by the expander property of $\GG(\pi_q(\Ga^\circ),\pi_q(S'))$
already proved, we conclude that
\[
N_m(A)>|A|+c/|S'|\min\{|A|/(2|\Ga/\Ga^\circ|),|\pi_q(\Ga^\circ)|/4\}
\]
which proves the theorem.

\begin{rmk}
The above proof implies a variant of Proposition \ref{pr_flatening}
that is useful in some applications.
Compare the statement below with \cite[Lemma 2 in Section 7]{BGS2}.
Let $q$ be a square-free integer and $\G$ be a Zariski-connected,
perfect algebraic group defined over $\Q$,
and write $G=\G(\Z/q\Z)$.
For every $\e>0$, there is a $\d>0$ such that the following hold:
Let $\mu$ be a probability measure which satisfies the following version of
the assumptions
in Proposition \ref{pr_flatening} for some $\e>0$.
I.e.
\[
\|\mu\|_2>|G|^{-1/2+\e}\quad{\rm and}\quad
\mu(gH)<[G:H]^{-\e}|G|^\d
\]
for any $g\in G$ and for any proper subgroup $H<G$.
Let $f\in l^2(G)$ be a complex valued function on the group $G$ such that
\[
\sum_{g\in a\G(\Z/q'\Z)} f(g)=0
\]
for all $a\in G$ and $q'|q$ with $q'\neq1$.
This condition is equivalent to saying that $f$ is orthogonal to those
irreducible subrepresentations in the regular representation of $G$
that factor through $G/\G(\Z/q'\Z)$
for some $q'\neq1$.
Then using the argument in the proof above, we can write
\be\label{eq_rmk}
\|\mu*f\|_2<q^{-\d}\|f\|_2
\ee
for some $\d>0$ depending on $\e$ and $\G$.
Indeed, repeated application of Proposition \ref{pr_flatening}
shows the analogue of (\ref{eq_tbs}) for $\mu^{(L)}$ in place of
$\pi_q[\chi_{S'}^{(l)}]$ for some integer $L$ which depend on $\e$ and
$\G$.
Combining this with the lower bounds for multiplicities of the eigenvalues
we get the inequality (\ref{eq_rmk}).
We also note that the statement in this remark also holds if we consider
a group $G$ which satisfies
the assumptions (A0)--(A5) listed in Section \ref{sc_assump} instead of
taking
$G=\G(\Z/q\Z)$.
\end{rmk}

\appendix
\section{Appendix: Effectivization of Nori's paper}
\label{sc_effective}
In this section, we address the non-effective parts of Nori's argument in~\cite{Nor} and present 
alternative effective arguments. Most of the arguments in the mentioned article are effective. We 
only need to present effective proofs of \cite[Proposition 2.7, Lemma 2.8 and Theorem 5.1]{Nor}.  It should be said that in this article by effective we mean that there 
is an algorithm to find the implied constants. Alternatively one can say the mentioned functions are 
recursively  defined. We should say that these results are far from the best possible. In fact using 
the classification of finite  simple groups, Guralnick~\cite[Theorem D]{Gur} showed that if 
$p>\max\{n+2,11\}$, then Nori\rq{}s statement hold for any subgroup of $\GL_n(\mathbb{F}_p)$ 
which is generated by its $p$-elements and {\it has no normal $p$-subgroup}. Unfortunately this last condition does not allow us to apply this sharp result.

Before stating the main results of this section, we recall very few terms from~\cite{Nor} and refer 
the reader to the mentioned article for the undefined terms. 

Here $R$ always denotes a finitely presented integral domain unless otherwise mentioned. 
\begin{dfn}[Definition 2.2 in \cite{Nor}]
An $R$-submodule $L$ of $M_n(R)$ is called a $k$-strict Lie subalgebra of $M_n(R)$ if 
\begin{enumerate}
\item $L$ is a Lie ring.
\item There is a submodule $L'$ of $M_n(R)$ such that
\[
M_n(R)=L\oplus L',
\]
and $L'$ is  locally free of rank $n^2-k$.
\end{enumerate} 
\end{dfn}
\begin{dfn}[Definition 2.5 in \cite{Nor}]
Let $\mathbb{U}_n$, $\mathbb{N}_n$ and $\mathbb{Y}_{n,k}$ be the schemes which represent the following functors from $\bbz[\frac{1}{(2n-1)!}]$-algebras $A$ to sets:
\begin{enumerate}
\item $\mathbb{N}_n(A):=\{x\in M_n(A)|\h x^n=0\}$,
\item $\mathbb{U}_n(A):=\{x\in M_n(A)|\h (x-1)^n=0\}$,
\item $\mathbb{Y}_{n,k}(A):=\{{\bf x}=(x_1,\ldots,x_k)\in \mathbb{N}_n(A)|\h L_{\bf x}$ is a $k$-strict Lie subalgebra$\}$, where $L_{\bf x}=Ax_1+\cdots+Ax_k$.
\end{enumerate}
\end{dfn}
\begin{dfn}\label{d:LH}
Let $L$ be a $k$-strict Lie subalgebra of $M_n(A)$ and $\mathbb{H}$ be a closed subgroup-scheme of $(\mathbb{GL}_n)_A:=\mathbb{GL}_n\times \spec{A}$. Let $\mathbb{L}$ be the $A$-scheme which represents the functor $S\mapsto S\otimes A$ defined for all commutative $A$-algebras. Let us define two closed subschemes of $(\mathbb{GL}_n)_A$:
\[
e(\mathbb{L}^{(n)}):=\exp(\mathbb{L}\cap (\mathbb{N}_n)_A),\hspace{1cm} \mathbb{H}^{(u)}:=\mathbb{H}\cap (\mathbb{U}_n)_A,
\]
for any $\bbz[1/(2n-1)!]$-algebra $A$.
\end{dfn} 
\begin{dfn}[Definition 2.3 and Remark 2.18 in \cite{Nor}]\label{d:acceptable}
Let $L$ and $\mathbb{H}$ be as in Definition~\ref{d:LH}. Then $(L,\mathbb{H})$ is called an acceptable pair if the following hold:
\begin{enumerate}
\item The projection  $\mathbb{H}\rightarrow \spec(A)$ is a smooth morphism with all the fibers connected. 
\item $\Lie(\mathbb{H}/A)=L$.
\item $(e(\mathbb{L}^{(n)}))_{{\rm red}}=(\mathbb{H}^{(u)})_{\rm red}$.
\end{enumerate}
In this case, $L$ or $\mathbb{H}$ are called acceptable. 
\end{dfn}
In this section, let $\underline{X}=\{X_1,\ldots,X_m\}$ and $R[\underline{X}]$ be the ring of 
polynomials in the variables $X_1,\ldots,X_m$ with coefficients in the ring $R$. If $F$ is a subset 
of a ring $R$, then $\langle F\rangle$ denotes the ideal generated by $F$ in $R$. 

Here are the main results of this section:
\begin{lem}[Effective version of Lemma 2.8 in \cite{Nor}]\label{l:2.8}
Let $R$ be a computable noetherian integral domain with quotient field $K$ and ${\bf z} \in \mathbb{Y}_{n,k}(R)$. If $L_{\bf z}\otimes_R K$  is acceptable, then we can algorithmically find a non-zero element $g\in R$ such that $L_{\bf z}\otimes_R R_g\subseteq M_n(R_g)$ is also acceptable.  
\end{lem}
(For the definition of a computable ring, see \cite[Chapter 4.6]{Grobner}.)
\begin{lem}[Effective version of Proposition 2.7 in \cite{Nor}]\label{l:2.7}
For a given $k,n$, we can give presentations of finitely many integral domains $R_i$ and algorithmically find elements ${\bf z}_i\in \mathbb{Y}_{n,k}(R_i)$ such that
\begin{enumerate}
\item $L_{{\bf z}_i}$ is acceptable if ${\rm char}(R_i)=0$.
\item ${\bf z}_i:\spec(R_i)\rightarrow \mathbb{Y}_{n,k}$ is a locally closed immersion and 
\[
\mathbb{Y}_{n,k}=\bigsqcup_i {\bf z}_i(\spec(R_i)).
\]
\end{enumerate}
\end{lem}
In order to prove Lemma~\ref{l:2.8}, we need to show the following effective version of certain results from EGA~\cite{EGA}. 
\begin{thm}[Effective version of Theorem 9.7.7 (i) and Theorem 12.2.4 (iii) in \cite{EGA}]\label{t:EGA}
Let $R$ be a computable integral domain. Let $F=\{f_1,\ldots, f_l\}\subseteq R[\underline{X}]$ and 
$
A=R[\underline X]/\langle F\rangle.
$
 If the generic fiber of the projection map
$\spec(A)\rightarrow \spec(R)$
is smooth and geometrically irreducible, then one can compute a non-zero element $g\in R$ such that
the projection map $\spec(A_g)\rightarrow \spec(R_g)$ is smooth and all of its fibers are geometrically irreducible.
\end{thm}
Finally we shall use Lemma~\ref{l:2.8} and Lemma~\ref{l:2.7} to get the following:
\begin{thm}[Effective version of a special case of Theorem 5.1 in \cite{Nor}]\label{t:5.1}
Let $\bbg$ be a perfect, Zariski-connected, simply-connected $\bbq$-group. Let 
$\Gamma\subseteq \bbg(\bbq)$ be a Zariski-dense subgroup generated by a finite symmetric set 
$S$. Then one can effectively find $p_0=p_0(S)$ such that for any $p>p_0$ one has 
$\pi_p(\Gamma)=\bbg(\mathbb{F}_p)$. 
\end{thm}

\subsection{Proof of Theorem~\ref{t:EGA}.}

In this section, first we show the \lq\lq{}generic flatness\rq\rq{} in 
Lemma~\ref{l:GenericFlatness} and the \lq\lq{}generic smoothness\rq\rq{} in Lemma~\ref{l:smooth2}. 
Then we reduce the general case of Theorem~\ref{t:EGA} to the hyperplane case and finish 
it as in~\cite{EGA}.

\begin{lem}\label{l:GenericFlatness}
Let $R$ be a computable noetherian integral domain. Let $K$ be the quotient field of $R$. Let $F=\{f_1,\ldots,f_l\}\subseteq R[\underline{X}]$ and $A=R[\underline{X}]/\langle F\rangle$. Then we can algorithmically find a non-zero element $g\in R$ such that
\begin{enumerate}
\item $A_g$ is a free $R_g$-module. 
\item $\langle F\rangle_g=\langle F\rangle_{K}\cap R_g[\underline{X}]$, where $\langle F\rangle_g$ (resp. $\langle F\rangle_{K}$) is the ideal generated by $F$ in $ R_g[\underline{X}]$ (resp. $K[\underline{X}]$).
\item All the fibers of the projection map $\spec(A_g)\rightarrow \spec(R_g)$ are equidimensional. 
\end{enumerate}
\end{lem}
\begin{proof}
For any ordering of $X_i$, we compute the Gr\"{o}bner basis of $\langle F\rangle_K$ and multiply all the head coefficients which appear in the process. Let $g\in R$ be the product of these head coefficients. Now the basic information on the Gr\"{o}bner basis gives us the first and the second parts. 

Now without loss of generality we can and will assume that $\{X_1,\ldots, X_d\}$ is a maximal 
independent set modulo $\langle F\rangle_K$. Since the head coefficients of the Gr\"{o}bner 
basis are units in $R_g$, for any $\pfr\in \spec(R_g)$, $\{X_1,\ldots, X_d\}$ is a maximal 
independent set modulo $\langle F\rangle_{k(\pfr)}\subseteq k(\pfr)[\underline{X}]$, where $k(\pfr)=R_{\pfr}/\pfr R_{\pfr}$. Hence, by \cite[Theorem 9.27]{Grobner}, we have that the Gelfand-Kirillov 
dimension of $A\otimes k(\pfr)$ is $d$, which proves the third part.
\end{proof}
\begin{dfn}
For $F=\{f_1,\ldots,f_l\}\subseteq R[\underline{X}]$, Let $J_e(F)$ denote the ideal generated by the $e\times e$ minors of $[\partial f_i/\partial X_{j}]$ in $R[\underline{X}]/\langle F\rangle$.
\end{dfn}
\begin{lem}\label{l:smooth1}
Let $\overline{k}$ be an algebraically closed field and $F=\{f_1,\ldots,f_l\}\subseteq \overline{k}[\underline{X}]$. Let $A=\overline{k}[\underline{X}]/\langle F\rangle$. Then the projection map $\spec(A)\rightarrow \spec(\overline{k})$ is smooth if and only if $J_{m- d}(F)=A$, where $m=\#\underline{X}$ and $d=\dim A$.
\end{lem}
\begin{proof}
This is essentially Jacobi's criteria. 
\end{proof}
\begin{lem}\label{l:smooth2}
Let $R$ be a computable noetherian integral domain, $F=\{f_1,\ldots,f_l\}$ be a subset of $R[\underline X]$, and 
$
A=R[\underline X]/\langle F\rangle.
$
If the generic fiber of the projection map $\spec(A)\rightarrow \spec(R)$ is smooth, then one can compute a non-zero element $g\in R$ such that the projection map $\spec(A_g)\rightarrow \spec(R_g)$ is smooth.
\end{lem}
\begin{proof}
Since clearly $\spec(A)$ is locally finitely presented, it is enough to compute $g\in R$ such that
\begin{enumerate}
\item The projection map $\spec(A_g)\rightarrow \spec(R_g)$ is flat.
\item For any $\pfr\in \spec(R_g)$, the projection map $\spec(A_g\otimes \overline{k(\pfr)})\rightarrow \spec(\overline{k(\pfr)})$ is smooth, where $\overline{k(\pfr)}$ is the algebraic closure of $k(\pfr)=R_{\pfr}/\pfr R_{\pfr}$.
\end{enumerate}
 By Lemma~\ref{l:GenericFlatness}, we can compute a non-zero element $g_1\in R$ such that the projection map $\spec(A_{g_1})\rightarrow\spec(R_{g_1})$ is flat and all of its fibers are of a fixed dimension $d$. On the other hand, since the generic fiber is smooth, by Lemma~\ref{l:smooth1}, $J_{m-d}(F)\cap R$ is non-zero. By computing a Gr\"{o}bner basis of $J_{m-d}(F)$, we can compute a non-zero element $g_2\in J_{m-d}(F)\cap R$. It is easy to check that $g=g_1g_2$ gives us the desired property.
 \end{proof}
\begin{lem}\label{l:NoetherNormalization+Primitive}
Let $R$ be an infinite computable noetherian integral domain, $F$ a finite subset of $R[\underline X]$ and $A=R[\underline X]/\langle F\rangle$. Then we can compute a matrix $[a_{ij}]\in {\rm GL}_m(R)$, a non-zero element $g\in R$, elements $x_{d+1}\in A_g$, $f\in R_g[X_1',\ldots,X_d']$ and $p\in R_g[X_1',\ldots,X_d',T]$, where $X_i'=\sum a_{ij} X_j$, such that the following hold
\begin{enumerate}
\item $R_g[X_1',\ldots,X_d']\cap \langle F\rangle_g=0$.
\item $A_g$ is an integral extension of $R_g[X_1',\ldots,X_d']$.
\item $A_{gf}=(R_g[X_1',\ldots,X_d'])_f[x_{d+1}]\simeq (R_g[X_1',\ldots,X_d',T]/\langle p\rangle)_f$.
\end{enumerate}
\end{lem}
\begin{proof}
Let $K$ be the quotient field of $R$. By \cite{Log}, we can compute a matrix $[a_{ij}]\in \GL_m(R)$ and elements $r_{d+2},\ldots,r_m\in R$ such that the following hold
\begin{enumerate}
\item $X_1',\ldots,X_d'$ are algebraically independent in $A\otimes K$. 
\item $X_j'$ are integral over $K[X_1',\ldots,X_d']$.
\item $S^{-1}A=K(X_1',\ldots,X_d')[x_{d+1}]$, where $S=K[X_1',\ldots,X_d']\setminus\{0\}$ and $x_{d+1}=X_{d+1}'+\sum_{i=d+2}^m r_i X_i'$.
\end{enumerate}
where $X_i'=\sum a_{ij} X_j$. 

Again computing Gr\"{o}bner basis of the ideal generated by $F$ in 
\[
K(X_1',\ldots,X_d')[X_{d+1}',\ldots,X_m']
\]
 with respect to various orderings, we can compute the minimal polynomials of $X_i'$ over $K(X_1',\ldots,X_d')$. Since $X_i'$ are integral over $K[X_1',\ldots,X_d']$ and the ring of polynomials over a field is integrally closed, all the minimal polynomials are monic polynomials with coefficients in $K[X_1',\ldots,X_d']$. Hence we can compute a non-zero element $g\in R$ such that $A_{g}$ is an integral extension of $R_{g}[X_1',\ldots,X_d']$. Moreover writing $X_i'$ as a polynomial in terms of $x_{d+1}$ with coefficients in $K(X_1',\ldots,X_d')$, we can find $f_1\in K[X_1',\ldots,X_d']$ such that $A_{f_1}\otimes K=K[X_1',\ldots,X_d']_{f_1}[x_{d+1}]$. We can also compute the minimal polynomial $p$ of $x_{d+1}$ over $K(X_1',\ldots,X_d')$. Now let $f$ be the product of $f_1$ by the product of all the denominators of the coefficients of the minimal polynomial. It is clear that these choices satisfy the desired properties. 
\end{proof}
\begin{proof}[Proof of Theorem~\ref{t:EGA}]
By Lemma~\ref{l:smooth2} and Lemma~\ref{l:GenericFlatness}, we can compute a non-zero element $g_1\in R$ such that the projection map $\spec(A_{g_1}) \rightarrow \spec(R_{g_1})$ is smooth and $A_{g_1}$ is a free $R_{g_1}$-module. Let $g_2\in R$, $f$ and $p$ be the parameters which are given by Lemma \ref{l:NoetherNormalization+Primitive}. Changing $R$ to $R_{g_1g_2}$ and using the above results, we can and will assume that 
\begin{enumerate}
\item $A$ is a free $R$-module. 
\item The projection map $\spec(A)\rightarrow \spec(R)$ is smooth,
\item $A$ is an integral extension of $R[X_1,\ldots,X_d]$ and the latter is the ring of polynomials,
\item $A_f\simeq (R[X_1,\ldots,X_d,T]/\langle p\rangle)_f$.
\end{enumerate}
Let $B=R[X_1,\ldots,X_d,T]/\langle p\rangle$. Since the generic fiber of $\spec(A)$ over $R$ is 
geometrically irreducible, so is the generic fiber of $\spec(B)$ over $R$. Hence by virtue of \cite
[Lemma 9.7.5]{EGA}, we can compute a non-zero element $g_3\in R$ such that all the fibers of 
$\spec(B_{g_3})\rightarrow \spec(R_{g_3})$ are geometrically irreducible. In particular, all the 
fibers of $\spec(A_{g_3f})\simeq \spec(B_{g_3f})\rightarrow \spec(R_{g_3})$ are geometrically 
irreducible. This means for any $\pfr\in \spec(R_{g_3})$
\[
A_{g_3f}\otimes \overline{k(\pfr)}
\]
is an integral domain. If it is a non-zero ring, then $A_{g_3}\otimes \overline{k(\pfr)}$ is also an integral domain. On the other hand, $A_{g_3f}\otimes \overline{k(\pfr)}=0$ if and only if $f$ is either zero or a zero-divisor in $A_{g_3}\otimes k(\pfr)$.

 By a similar argument as in Lemma~\ref{l:GenericFlatness}, we can compute a non-zero element $g_4\in R$ such that $(A/\langle f\rangle)_{g_4}$ is a free $R_{g_4}$-module. Let $g=g_3g_4$. We claim that all the fibers of $\spec(A_{g})\rightarrow \spec(R_g)$ are geometrically irreducible. By the above discussion, it is enough to show that for any $\pfr\in \spec(R_g)$, $f$ is not either zero nor a zero-divisor in $A_g\otimes \overline{k(\pfr)}$. 

Let $\lambda_f(x)=fx$ be the map of multiplication by $f$ in $A_g$. Since $A_g$ is an integral domain and $f$ is not zero, we have the following short exact sequence of $R_g$-modules:
\[
0\rightarrow A_g \xrightarrow{\lambda_f} A_g \rightarrow (A/\langle f\rangle)_g \rightarrow 0.
\]
Hence for any $\pfr\in \spec(R_g)$ we have the following exact sequence
\[
{\rm Tor}((A/\langle f\rangle)_g,k(\pfr))\rightarrow A_g\otimes k(\pfr) \rightarrow A_g\otimes k(\pfr).
\]
Since $(A/\langle f\rangle)_g$ is a free $R_g$-module, ${\rm Tor}((A/\langle f\rangle)_g,k(\pfr))=0$. Therefore $f$ is neither a zero nor a zero-divisor in $A_g\otimes k(\pfr)$. Thus by the above discussion, we are done.
\end{proof}

\subsection{Proof of Lemma~\ref{l:2.8}.}

By Definition~\ref{d:acceptable}, a pair of a Lie ring and an algebraic group scheme is 
acceptable if and only if it satisfies three properties. In this section, we show how one can use 
Theorem~\ref{t:EGA} to guarantee the first property. The second property is achieved using 
smoothness and the definition of the Lie algebra of a smooth group scheme. The third property is 
dealt with in Lemma~\ref{l:ThirdCon}.

\begin{lem}\label{l:PatternGen1}
Let $\bbg$ be an algebraic group and $X$ an irreducible subvariety. If $1\in X=X^{-1}$, then 
\[
\prod_{\dim \bbg} X= X\cdot \cdots \cdot X,
\]
is the group generated by $X$.
\end{lem}
\begin{proof}
This is clear.
\end{proof}
\begin{lem}\label{l:PatternGen2}
Let $\bbg$ be an algebraic group and $X$ an irreducible subvariety. Then
\[
\prod_{\dim \bbg} X\cdot X^{-1}=(X\cdot X^{-1})\cdot \cdots\cdot (X\cdot X^{-1})
\]
is the group generated by X. 
\end{lem}
\begin{proof}
It is a consequence of Lemma \ref{l:PatternGen1}.
\end{proof}
\begin{lem}\label{l:PatternGen3}
Let $\bbg$ be an algebraic group and $X_i$ irreducible subvarieties which contain $1$. Let $\widetilde{X}=(X_1\cdot X_1^{-1})\cdot \cdots \cdot (X_k \cdot X_k^{-1})$. 
Then $\prod_{\dim \bbg} (\widetilde{X}\cdot \widetilde{X}^{-1})$ is the group generated by $X_i$. 
\end{lem}
\begin{proof}
By Lemma \ref{l:PatternGen2}, it is enough to observe that $X_i\subseteq \widetilde{X}$.
\end{proof}
\begin{lem}\label{l:presentation}
Let $R$ be a computable integral domain whose characteristic is at least $2n$ and ${\bf z}=(z_1,\ldots,z_k)\in 
\mathbb{Y}_{n,k}(R)$. Let $K$ be the quotient field of $R$, $\bbh$ be the $K$-subgroup scheme of $(\mathbb{GL}_n)_K$ which is generated by $\exp(tz_i)$ and $\mathcal{H}$ be the closure of $\bbh$ in $(\mathbb{GL}_n)_R$. Then we can algorithmically find an element $g\in R$ and a finite subset $F=\{f_1,\ldots,f_l\}\subseteq R_g[\mathbb{GL}_n]$ such that 
\[
\mathcal{H}\times_{\spec(R)} \spec(R_g)\simeq R_g[\mathbb{GL}_n]/\langle F\rangle
\]
 as closed subschemes of $(\mathbb{GL}_n)_{R_g}$.
\end{lem}
\begin{proof}
It is clear that, for any $i$, the image of $a_{z_i}:\bba^1_K\rightarrow (\mathbb{GL}_n)_K$
\[
a_{z_i}(t):=\exp(tz_i),
\]
is a 1-dimensional irreducible $K$-algebraic subgroup of $ (\mathbb{GL}_n)_K$. Hence by Lemma \ref{l:PatternGen3} we can find an algebraic morphism $\Phi:\bba^{(2k-1)n^2}\rightarrow  (\mathbb{GL}_n)_K$ whose image is exactly $\bbh$. Hence by means of the elimination method we can compute a presentation for $\bbh$, i.e. $F=\{f_1,\ldots, f_l\}\in R[\mathbb{GL}_n]$ such that
\[
\bbh \simeq \spec(K[\mathbb{GL}_n]/\langle F\rangle_K),
\] 
as $K$-varieties. Now by the second part of Lemma \ref{l:GenericFlatness}, we can compute a non-zero element $g\in R$ such that
\[
\mathcal{H}_g:=\mathcal{H}\times_{\spec(R)}\spec(R_g)\simeq \spec(R_g[\mathbb{GL}_n]/\langle F\rangle_g).
\]
\end{proof}
\begin{lem}\label{l:LieRing}
Let $R$ be a computable integral domain, $K$ be the quotient field of $R$, ${\bf z}\in \mathbb{Y}_{n,k}(R)$, $L=L_{\bf z}$ and $\bbh$ be a closed subgroup of $(\mathbb{GL}_n)_K$. If $(\bbh, L\otimes K)$ is an acceptable pair, then we can algorithmically find a non-zero element $g\in R$ such that
$\Lie(\mathcal{H})(R_g)=L_g,$ where $\mathcal{H}$ is the closure of $\bbh$ in $(\mathbb{GL}_n)_R$ and $L_g=L\otimes R_g$.
\end{lem}
\begin{proof}
By \cite[Lemma 2.12]{Nor}, we know that $\bbh$ is generated by $\exp(tz_i)$. Hence by Lemma \ref{l:PatternGen3} and Theorem \ref{t:EGA}, we can compute a non-zero element $g_1\in R$ and a finite subset $F=\{f_1,\ldots,f_l\}\subseteq R[\mathbb{GL}_n]$ such that
\begin{enumerate}
\item The projection map $\mathcal{H}_{g_1}:=\mathcal{H}\times_{\spec(R)} \spec(R_{g_1})\rightarrow \spec(R_{g_1})$ is smooth.
\item As $R_{g_1}$-schemes, 
\[\mathcal{H}_{g_1}\simeq \spec( (R[\mathbb{GL}_n]/\langle F\rangle)_{g_1})\simeq \spec(R_{g_1}[\underline X]/\langle \tilde F\rangle),
\]
 where $\underline X=\{X_1,\ldots, X_{n^2+1}\}$, $\tilde F= F\cup\{ X_{n^2+1}D(X_1,\ldots,X_{n^2})-1\}$ and $D$ is the determinant of the first $n^2$ variables.
\end{enumerate}
Since $\mathcal{H}_{g_1}$ is a smooth $R_{g_1}$-scheme, $\Lie(\mathcal{H}_{g_1}/R_{g_1})=\Ker({\rm Jac}(\tilde F))$, where ${\rm Jac}(\tilde F)=[\partial \tilde f_i/\partial X_j]$ is the Jacobian of 
\[
(X_1,\ldots,X_{n^2+1})\mapsto (\tilde f(X_1,\ldots,X_{n^2+1}))_{\tilde f\in\tilde F}. 
\] 
 By Gauss-Jordan process, we can compute a non-zero element $g_2\in R$ such that $\Ker({\rm Jac}(\tilde F))_{g_2}$ is a free $R_{g_2}$-module. We can also compute an $R_{g_2}$-basis. Since we know that $L\otimes K=\Ker({\rm Jac}(\tilde F))_{g_2}\otimes_{R_{g_2}} K$ and we have $R_{g_2}$-basis for both of them, we can compute a non-zero element $g$ such that $L_g=\Ker({\rm Jac}(\tilde F))_{g}$, which finishes the proof.
\end{proof}
\begin{lem}\label{l:ThirdCon}
Let $R$, $K$, ${\bf z}$, $L$, $\bbh$ and $\mathcal{H}$ be as in Lemma \ref{l:LieRing}. If $(\bbh, L\otimes K)$ is an acceptable pair, then we can algorithmically find a non-zero element $g\in R$ such that 
\[
(e(\bbl^{(n)})\times_{\spec(R)} \spec(R_g))_{\rm red}=(\mathcal{H}^{(u)}\times_{\spec(R)} \spec(R_g))_{\rm red}.
\]
\end{lem}
\begin{proof}
Since $L$ is given through an $R$-basis, we can compute a non-zero element $g_1\in R$ and an $R_{g_1}$-basis for the dual of $L$. Hence we can compute a presentation for $\mathbb{L}$. Thus using elimination method we can compute a presentation of $e(\bbl^{(n)})_{g_1}:=e(\bbl^{(n)})\times_{\spec(R)} \spec(R_{g_1})$. We can also compute a presentation of $\mathcal{H}^{(u)}$. 

Since $(\bbh, L\otimes K)$ is an acceptable pair, we have 
\[
(e(\bbl^{(n)})\times_{\spec(R)} \spec(K))_{\rm red}=(\mathcal{H}^{(u)}\times_{\spec(R)} \spec(K))_{\rm red}.
\]  
So having a presentation of both sides over $R_{g_1}$, one can easily compute $g_2\in R$ such that
\[
(e(\bbl^{(n)})\times_{\spec(R)} \spec(R_g))_{\rm red}=(\mathcal{H}^{(u)}\times_{\spec(R)} \spec(R_g))_{\rm red}
\]
holds for $g=g_1g_2$.
\end{proof}
\begin{proof}[Proof of Lemma~\ref{l:2.8}]
One can repeat Nori\rq{}s argument~\cite[Lemma 2.8]{Nor} and get the effective version using Theorem \ref{t:EGA}, Lemma \ref{l:LieRing} and Lemma \ref{l:ThirdCon}.
\end{proof}

\subsection{Proof of Lemma~\ref{l:2.7}.}

In this section, first we give a precise presentation of $\mathbb{Y}_{n,k}$. Then using Lemma~\ref{l:2.8} by an inductive argument we get the desired result.
\begin{dfn}
Let $F=\{f_1,\ldots,f_l\}$ and $F\rq{}=\{f\rq{}_1,\ldots,f\rq{}_{l\rq{}}\}$ be two subsets of $R[\underline X]$, where $\underline X=\{X_1,\ldots,X_m\}$. Then let $V(F)$ denote the closed subscheme of $\bba^m_R$ defined by the relations $F$, and 
\[
W(\bba^m_R; F,F\rq{}):=V(F)\setminus V(F\rq{}).
\]
If $\mathfrak{a}$ and $\mathfrak{b}$ are two ideals of $R[\underline X]$, then $V(\mathfrak{a})$ denotes the closed subscheme of $\bba^m_R$ defined by $\mathfrak{a}$ and 
\[
W(\bba^m_R;\mathfrak{a},\mathfrak{b}):=V(\mathfrak{a})\setminus V(\mathfrak{b}).
\]
\end{dfn}
\begin{dfn}
For any ${\bf z}=(z_1,\ldots,z_k)\in M_n(R)^k$, fix the standard $R$-basis of $M_n(R)$ and view $z_i$ as column vectors in this basis. Let $F_{\bf z}$ be the set of all the maximum dimension minors of the matrix $\left[\begin{array}{lcr}z_1&\cdots&z_k\end{array}\right]$ and $\mathfrak{a}_{\bf z}$ be the ideal generated by $F_{\bf z}$. 

We also consider the case $R=\bbz[\underline X]$, where $\underline X=\{X_{ij}^{i\rq{}}| 1\le i,j\le n, 1\le i\rq{} \le k\}$ and set ${\bf x}=(x_1,\ldots,x_k)$, where the $ij$-th entry of $x_{i\rq{}}$ is $X_{ij}^{i\rq{}}$. In this case, $F_{n,k}:=F_{\bf x}$ and $\mathfrak{a}_{n,k}:=\mathfrak{a}_{\bf x}$.
\end{dfn}
\begin{rmk}
We sometimes identify the functor $M_n$ with $\bba^{n^2}_{\bbz}$. This way, any ${\bf z}\in M_n(R)$ gives rise to a ring homomorphism $\phi_{\bf z}$ from $\bbz[\bba^{kn^2}]$ to $R$ and it is clear that $\phi_{\bf z}(\mathfrak a_{n,k})=\mathfrak{a}_{\bf z}$.
\end{rmk}
\begin{lem}\label{l:local}
Let $(A,\mathfrak{m})$ be a pair of a local ring and its maximal ideal. Let $\phi:A^n\rightarrow A^n$ be an $A$-linear map. Then the following statements are equivalent:
\begin{enumerate}
\item $\phi$ is surjective.
\item $\overline{\phi}:(A/\mathfrak{m})^n\rightarrow (A/\mathfrak{m})^n$ is invertible.
\item $\phi$ is invertible.
\end{enumerate}
\end{lem}
\begin{proof}
This is clear.
\end{proof}
\begin{lem}\label{l:Ypresentation1}
Let $R$ be any commutative ring. Then 
\[
{\bf z}=(z_1,\ldots,z_k)\in W(\bba^{kn^2}_{\bbz};0,\mathfrak{a}_{n,k})(R)
\]
 if and only if $M_n(R)/L$ is locally of dimension $n^2-k$, where $L=Rz_1+\cdots+Rz_k$.
\end{lem}
\begin{proof}
By the definition of $W(\bba^{kn^2}_{\bbz};0,\mathfrak{a}_{n,k})(R)$, it is straightforward to check that ${\bf z}\in W(\bba^{kn^2}_{\bbz};0,\mathfrak{a}_{n,k})(R)$ if and only if $\mathfrak{a}_{\bf z}=R$.

On the other hand, $M_n(R)/L$ is locally of dimension $n^2-k$ if and only if for any $\pfr\in \spec(R)$ there are $z^{(\pfr)}_{k+1},\ldots,  z^{(\pfr)}_{n^2}\in M_n(R)$ such that 
\begin{equation}\label{e:sum}
M_n(R_{\pfr})=(\sum_{i=1}^k R_{\pfr} z_i)+\oplus_{i=k+1}^{n^2} R_{\pfr} z^{(\pfr)}_i.
\end{equation}
Let $\phi_{\pfr}:R_{\pfr}^{n^2}\rightarrow M_n(R_{\pfr})$ be the following $R_{\pfr}$-linear map
\[
\phi_{\pfr}(r_1,\ldots,r_{n^2}):=\sum_i r_iz^{(\pfr)}_i,
\]
where $R_{\pfr}^{n^2}$ is the direct sum of $n^2$ copies of $R_{\pfr}$ and $z^{(\pfr)}_i=z_i$ for any $i\le k$.  By Lemma \ref{l:local}, it is clear that (\ref{e:sum}) holds if and only if $\overline{\phi}_{\pfr}$ is invertible. It is easy to see that the latter is equivalent to $\mathfrak a_{\overline{\bf z}}=k(\pfr)$, where $k(\pfr)=R_{\pfr}/\pfr R_{\pfr}$ and $\overline{\bf z}=(\overline z_1,\ldots,\overline z_k)\in M_n(k(\pfr))^k$. Let $S_{\pfr}=R\setminus \pfr$. By the definition, it is clear that $\mathfrak{a}_{\overline{\bf z}}=k(\pfr)$ if and only if $S_{\pfr}^{-1}\mathfrak a_{\bf z}=R_{\pfr}$. The latter holds for any $\pfr\in \spec(R)$ if and only if $\mathfrak a_{\bf z}=R$, which completes the proof.
\end{proof}

\begin{dfn}

Let ${\bf z}=(z_1,\ldots,z_k)\in (R^{n^2})^k$. We sometimes view such a vector in two other ways: as an $n^2\times k$ matrix whose $i$-th column is $z_i$; or a $k$-tuple of $n\times n$ matrices whose $i$-th entry is $z_i$ written in matrix form with respect to the standard basis. 

Let $J\subseteq \{1,\ldots,n^2\}$ be of order $k$. Then ${\bf z}_J$ denotes the $k\times k$ submatrix of ${\bf z}$ whose rows are determined by $J$. For a vector $v\in R^{n^2}$, $v_J$ denotes the subvector of size $k$ determined by $J$.

For a given $a\in {\rm Mor}(\bba^{kn^2}_{\bbz},\bba^{n^2}_{\bbz})$ and any subsets $J,J\rq{}\subseteq \{1,\ldots,n^2\}$ of order $k$, we define $f^{(a)}_{J,J\rq{}}\in {\rm Mor}(\bba^{kn^2}_{\bbz},\bba^k_{\bbz})$ as follows: 
\[
f^{(a)}_{J,J\rq{}}({\bf z})={\bf z}_{J\rq{}}{\rm adj}({\bf z}_J)a({\bf z})_J-\det({\bf z}_J)a({\bf z})_{J\rq{}}.
\]
Also let $F^{(a)}_{n,k}$ be the set consisting of all the entries of $f^{(a)}_{J,J\rq{}}$ for all the possible $J$ and $J\rq{}$.
\end{dfn}
\begin{lem}\label{l:Ypresentation2}
Let $R$ be any commutative ring and $a\in  {\rm Mor}(\bba^{kn^2}_{\bbz},\bba^{n^2}_{\bbz})$.  Then 
\[
{\bf z}=(z_1,\ldots,z_k)\in W(\bba^{kn^2}_{\bbz};F^{(a)}_{n,k},F_{n,k})(R)
\]
if and only if 
\begin{enumerate}
\item $M_n(R)/L_{\bf z}$ is locally of dimension $n^2-k$, where $L_{\bf z}=Rz_1+\cdots+Rz_k$,
\item $a({\bf z})\in L_{\bf z}$. 
\end{enumerate}
\end{lem}
\begin{proof}
Let $\mathcal{Y}_{n,k}=W(\bba_{\bbz}^{kn^2};\varnothing,F_{n,k})$ and let $\mathcal{Y}^{(a)}_{n,k}$ be the functor from commutative rings to sets such that 
\[
\mathcal{Y}^{(a)}_{n,k}(R)=\{{\bf z}\in \mathcal{Y}_{n,k}(R)|a({\bf z})\in L_{\bf z}\}.
\]
By Lemma~\ref{l:Ypresentation1}, it is enough to show that 
\[
\mathcal{Y}^{(a)}_{n,k}(R)=W(\bba_{\bbz}^{kn^2};F^{(a)}_{n,k},F_{n,k})(R).
\]
Let us view ${\bf z}$ as an $n^2\times k$ matrix. Then if ${\bf z}\in \mathcal{Y}_{n,k}$, then it 
belongs to $\mathcal{Y}^{(a)}_{n,k}(R)$ if and only if there is $\vec{r}=(r_1,\ldots,r_k)$ such 
that ${\bf z}\vec{r}= a({\bf z})$. The latter holds if and only if for any $J\subseteq \{1,\ldots,n^2\}$ of order $k$ we have ${\bf z}_J \vec{r}=a({\bf z})_J$.

We claim that if ${\bf z}\in \mathcal{Y}^{(a)}_{n,k}(R)$ then there is a unique $\vec{r}$ which 
satisfies the equations ${\bf z}_J\vec{r}=a({\bf z})_J$ for all the subsets $J$ of order $k$ in $\{1,\ldots,n^2\}$. To show this claim it is enough to notice that $\det({\bf z}_J) \vec{r}=\adj({\bf z}_J)a({\bf z})_J$ and the ideal generated by $\det({\bf z}_J)$ as $J$ runs through all the subsets 
of order $k$ contains 1 as ${\bf z}\in \mathcal{Y}_{n,k}(R)$. 

We also observe that if ${\bf z}\in \mathcal{Y}^{(a)}_{n,k}(R)$, then for any $J$ and $J\rq{}$ we have
\[
{\bf z}_{J\rq{}}\adj({\bf z}_J)a({\bf z})_J=\det({\bf z}_J) {\bf z}_{J\rq{}}\vec{r}=\det({\bf z}_J) a({\bf z})_{J\rq{}}.
\]
Hence ${\bf z}\in W(\bba_{\bbz}^{kn^2};F^{(a)}_{n,k},F_{n,k})(R)$, i.e. $\mathcal{Y}^{(a)}_{n,k}(R)\subseteq W(\bba_{\bbz}^{kn^2};F^{(a)}_{n,k},F_{n,k})(R)$.

Let ${\bf z}\in  W(\bba_{\bbz}^{kn^2};F^{(a)}_{n,k},F_{n,k})(R)$. We claim that if $\det({\bf z}_J)$ is a unit in $R$ for some $J$, then ${\bf z}\in \mathcal{Y}^{(a)}_{n,k}(R)$. To see this it is enough to check that 
\[
\vec{r}=\det({\bf z}_J)^{-1}\adj({\bf z}_J)a({\bf z})_J
\]
 satisfies all the equations ${\bf z}_{J\rq{}} \vec{r}=a({\bf z})_{J\rq{}}$. In particular, for any local ring $R$, we have 
\[
\mathcal{Y}^{(a)}_{n,k}(R)=W(\bba_{\bbz}^{kn^2};F^{(a)}_{n,k},F_{n,k})(R).
\]
For an arbitrary commutative ring $R$, let again ${\bf z}\in W(\bba_{\bbz}^{kn^2};F^{(a)}_{n,k},F_{n,k})(R)$. By the above discussion, for any $\pfr\in \spec(R)$, we have that ${\bf z}\in \mathcal{Y}^{(a)}_{n,k}(R_{\pfr})$, i.e. there is a unique $\vec{r}_{\pfr}\in R_{\pfr}^k$ such that ${\bf z}_J\vec{r}_{\pfr}=a({\bf z})_J$ for any $J$. On the other hand, by the uniqueness argument, since the ideal generated by $\det({\bf z}_J)$ is equal to $R$, there is $\vec{r}\in R^k$ such that for any $\pfr$ and any $J$ we have ${\bf z}_J \vec{r}=a({\bf z})_J$ in $R_{\pfr}^k$. Now one can easily deduce that ${\bf z}_J \vec{r}=a({\bf z})_J$ in $R^k$, which means ${\bf z}\in \mathcal{Y}^{(a)}_{n,k}(R)$ and we are done.
\end{proof}
\begin{dfn}
Let $a_{ij}\in {\rm Mor}(\bba_{\bbz}^{kn^2},\bba_{\bbz}^{n^2})$ be the following morphism
\[
a_{ij}({\bf z}):=[z_i,z_j]=z_iz_j-z_jz_i,
\]
for any $1\le i,j\le k$. Let $\widetilde{F}_{n,k}:=\bigcup_{i,j=1}^k F^{(a_{ij})}_{n,k}$.
\end{dfn}
\begin{cor}\label{c:Ypresentation}
For any commutative ring $R$, we have 
\[
\mathbb{Y}_{n,k}(R)=W(\bba_{\bbz}^{kn^2};\widetilde{F}_{n,k},F_{n,k})(R).
\]
\end{cor}
\begin{proof}
This is a direct consequence of Lemma~\ref{l:Ypresentation2}
\end{proof}
\begin{lem}\label{l:Covering}
Let $F$ and $F\rq{}=\{f_1\rq{},\ldots,f_{l\rq{}}\rq{}\}$ be two subsets of $\bbz[\underline{X}]$, where $\underline{X}=\{X_1,\ldots,X_m\}$. Assume that $\langle F\rangle$ is a radical ideal. Then we can computationally determine if $W(\bba_{\bbz}^m;F,F\rq{})$ is nonempty, and if it is, then we can give a presentation of an integral domain $R$ and $z\in W(\bba_{\bbz}^m;F,F\rq{})(R)$ such that 
\begin{enumerate}
\item $z:\spec(R)\rightarrow W(\bba_{\bbz}^m;F,F\rq{})$ is an open immersion.
\item For any given $d\in R$, we can computationally describe the complement of $z(\spec(R[\frac{1}{d}]))$ in $ W(\bba_{\bbz}^m;F,F\rq{})$.
\end{enumerate}
\end{lem}
\begin{proof}
It is clear that $W(\bba_{\bbz}^m;F,F\rq{})$ is empty if and only if $\langle F\rq{}\rangle\subseteq \sqrt{\langle F\rangle}$, which can be computationally determined. To show the rest, 
first we claim that we can assume that $F\rq{}=\varnothing$.  To show this claim, we start with the following open affine covering:
\[
W(\bba_{\bbz}^m;F,F\rq{})=\cup_i W(\bba_{\bbz}^m;F,\{f\rq{}_i\}).
\]
And, we notice that $W(\bba_{\bbz}^m;F,\{f\rq{}_i\})\simeq W(\bba_{\bbz}^{m+1};F\cup \{f\rq{}_iX_{m+1}-1\},\varnothing).$ Now if we find $R$ and $z$ for $F\cup \{f\rq{}_iX_{m+1}-1\}$ and $F\rq{}=\varnothing$, then one can see that the first assertion still 
holds and the complement of $z(\spec(R[\frac{1}{d}])$ in $W(\bba_{\bbz}^m;F,F\rq{})$ is equal 
to the union of its complement in $W(\bba_{\bbz}^m;F,\{f_i\rq{}\})$ and 
$W(\bba_{\bbz}^m;F\cup \{f_i\rq{}\},F\rq{})$.

So without loss of generality, we can and will assume that $F\rq{}=\varnothing$. By \cite[Chapter 8.5]{Grobner}, we can compute a primary decomposition $\cap_i \pfr_i$ of $\langle F\rangle$. Since $\langle F\rangle=\sqrt{ \langle F\rangle}$, $\pfr_i$ is a prime ideal for any $i$. If $\langle F\rangle$ is a prime ideal, let $c=1$ and $R=\bbz[\underline{X}]/\langle F\rangle$; otherwise,
let $c\in \cap_{i\le 2} \pfr_i\setminus \pfr_1$ (we can computationally find such $c$) and $R=(\bbz[\underline{X}]/\pfr_1)[\frac{1}{c}]$. Clearly this choice of $R$ satisfies the first assertion in the statement of Lemma. Now let $d\in R$ be a given element. Then one can easily check that the complement of the natural open immersion of $\spec(R[\frac{1}{d}])$ in $W(\bba_{\bbz}^m;F,\varnothing)$ is isomorphic to
\[
W(\bba_{\bbz}^{m};F\cup \{c\},\varnothing)\cup W(\bba_{\bbz}^{m};F\cup \{d\},\varnothing).
\]
\end{proof}

\begin{proof}[Proof of Lemma~\ref{l:2.7}]
Following Nori\rq{}s proof of \cite[Proposition 2.7]{Nor} and using Lemma~\ref{l:2.8}, Corollary~\ref{c:Ypresentation}, and Lemma~\ref{l:Covering} whenever needed, one can easily prove this lemma.
\end{proof}

\subsection{Proof of Theorem~\ref{t:5.1}.}

\begin{lem}\label{l:ExplicitZariskiClosure}
Let $S\subseteq \GL_n(\bbq)$ be a finite set of matrices. Let $\Gamma$ be the group generated 
by $S$. Assume that the Zariski-closure of $\Gamma$ in $(\mathbb{GL}_n)_{\bbq}$ is Zariski-
connected. Then we can compute a square-free integer $q_0$ and a finite subset $F=\{f_1,\ldots,f_l\}\subseteq \bbz[1/q_0][\mathbb{GL}_n]$ such that $\Gamma\subseteq \GL_n(\bbz[1/q_0])$ and its Zariski-
closure in $(\mathbb{GL}_n)_{\bbz[1/q_0]}$ is isomorphic to $\bbz[1/q_0][\mathbb{GL}_n]/\langle F\rangle$.
\end{lem}
\begin{proof}
Since $S$ is a finite set of matrices, we can find an odd prime $p$ such that $\Gamma\subseteq \GL_n(\bbz_p)$. Hence by changing $\Gamma$ to $\Gamma\cap \GL_n^{(1)}(\bbz_p)$ (the first congruence subgroup is denoted by $\GL_n^{(1)}(\bbz_p)$), we can 
and will assume that $\Gamma$ is torsion free. It is worth mentioning that we are allowed to make such a change because of the following:
\begin{enumerate}
\item We can compute representatives for the cosets of $\Gamma\cap\GL_n^{(1)}(\bbz_p)$ in $\Gamma$. Thus we can compute a generating set for $\Gamma\cap\GL_n^{(1)}(\bbz_p)$.
\item Since we have assumed that the Zariski-closure $\bbh$ of $\Gamma$ in $(\mathbb{GL}_n)_{\bbq}$ is Zariski-connected, the Zariski-closure of $\Gamma\cap \GL_n^{(1)}(\bbz_p)$ in $(\mathbb{GL}_n)_{\bbq}$ is also $\bbh$.
\end{enumerate}
We find a presentation for $\bbq[\bbh]$ and then similar to the proof of Lemma~\ref{l:presentation} we can finish the argument. 
 
Since $\Gamma$ is torsion-free, the Zariski-closure of the cyclic group generated by any element 
of $\Gamma$ is of dimension at least one. Hence by the virtue of Lemma~\ref{l:PatternGen3} it is 
enough to find a presentation of the Zariski-closure $\bbg$ of the cyclic group generated by 
$\gamma\in S$. We can compute the Jordan-Chevalley decomposition $\gamma_u\cdot \gamma_s$ of $\gamma$. Let $\bbg_u$ (resp. $\bbg_s$) be the Zariski-closure of the group 
generated by $\gamma_u$ (resp. $\gamma_s$). Then $\bbg\simeq \bbg_u\times\bbg_s$ as 
$\bbq$-groups~\cite[Theorem 4.7]{Bor}. Using the logarithmic and exponential maps, one can 
easily find a presentation of $\bbg_u$. So it is enough to find a presentation of $\bbg_s$. We can 
compute all the eigenvalues $\lambda_1,\ldots,\lambda_n$ of $\gamma_s$. By \cite[Proposition 
8.2]{Bor}, in order to find a presentation of $k[\bbg]$, where $k$ is the number field generated by 
$\lambda_i$, it is enough to find a basis for the following subgroup of $\bbz^n$
\[
\{(m_1,\ldots,m_n)\in \bbz^n|\h \prod_i \lambda_i^{m_i}=1\},
\]
i.e. all the character equations, which is essentially done in \cite{Ric}. So far we found a finite subset $F'$ of $k[\mathbb{GL}_n]$ such that $k[\bbg_s]\simeq k[\mathbb{GL}_n]/\langle F'\rangle$. Since $\bbg_s$ is defined over $\bbq$, we have that $\bbq[\bbg_s]\simeq  \bbq[\mathbb{GL}_n]/(\langle F'\rangle\cap \bbq[\mathbb{GL}_n])$. On the other hand, using Gr\"{o}bner basis we can find a generating set $F_s$ for $\langle F'\rangle\cap \bbq[\mathbb{GL}_n]$, which finishes our proof.
\end{proof}
\begin{lem}\label{l:GenericHS}
 Let $\lcal$ be a smooth $\bbz[1/q_0]$-subgroup scheme of $(\mathbb{GL}_n)_{\bbz[1/q_0]}$. 
Let $F\subseteq \bbz[1/q_0][\mathbb{GL}_n]$ such that $\bbz[1/q_0][\mathbb{GL}_n]/\langle F\rangle$. If the generic fiber $\bbl$ of $\lcal$ is a 
simply-connected semisimple $\bbq$-group, then
\begin{enumerate}
\item we can algorithmically find a positive integer $p_0$ such that for any prime 
$p>p_0$, the special fiber $\lcal_p:=\lcal\times_{\spec(\bbz[1/q_0])}\spec(\bbf_p)$ of $\lcal$ 
over $p$ is a semisimple $\bbf_p$-group.
\item we can algorithmically find a positive integer $p_0$ such that, for any $p>p_0$, $\lcal(\bbz_p)$ is a hyper-special parahoric in $\bbl(\bbq_p)$.
\end{enumerate}
\end{lem}
\begin{proof}
 The second part is a consequence of the first part as it is explained in~\cite[Section 3.9.1]{Tit}. 
Here we only prove the first part.

We can compute the Lie algebra $\lfr$ of $\bbl$. Since $\bbl$ is not a nilpotent Lie algebra,
not all the elements of a basis of $\lfr$ can be ad-nilpotent. Hence we can find an ad-semisimple
 element $x$ of $\lfr$. Since $\lfr$ is a semisimple Lie algebra and $x$ is a semisimple element, 
the centralizer $c_{\lfr}(x)$ of $x$ in $\lfr$ is a reductive algebra and $c_{\lfr}(x)/\mathfrak{z}(c_{\lfr}(x))$ is a semisimple Lie algebra (if not trivial), where $\mathfrak{z}(c_{\lfr}(x))$ is the 
center of $c_{\lfr}(x)$. If $c_{\lfr}(x)$ is not commutative, then repeating the above argument we 
can find $x'\in c_{\lfr}(x)\setminus \mathfrak{z}(c_{\lfr}(x))$. We can compute the eigen-
values $\lambda_i$ (resp. $\lambda'_i$) of ${\rm ad}(x)$ (resp. ${\rm ad}(x')$) and find 
\[
\lambda\neq \frac{\lambda_i-\lambda_j}{\lambda'_{i'}-\lambda'_{j'}},
\]
 for any $i,j,i',j'$, then $c_{\lfr}(x+\lambda x')=c_{\lfr}(x)\cap c_{\lfr}(x')$. By repeating this 
process, we can compute a number field $k$ over which $\lfr$ splits and we can also compute a 
Cartan subalgebra. Hence we can compute a Chevalley basis $x_i$ for $\lfr\otimes_{\bbq} k$. 
Looking at the commutator relations, we can compute an element $a$ of $k$ such that $\sum_i \mathcal{O}_k[1/a] x_i$ form a Lie subring of $\lfr\otimes_{\bbq} k$, where $\mathcal{O}_k$ is the ring of integers in $k$. Thus for any $p$ which does not divide $N_{k/\bbq}(a)$ the special fiber $\lcal_p$ is a semisimple $\bbf_p$-group, as we wished.
\end{proof}
\begin{lem}\label{l:GenericPerfectness}
Let $\bbh$ be a perfect  Zariski-connected $\bbq$-subgroup of $\mathbb{GL}_n$. Let $F$ be a 
finite subset of $\bbq[\mathbb{GL}_n]$ such that $\bbq[\bbh]\simeq\bbq[\mathbb{GL}_n]/\langle F\rangle$. Then one can compute a square-free 
integer $q_1$ and a finite subset $F'$ of $\bbz[1/q_1][\mathbb{GL}_n]$ such that
\begin{enumerate}
\item The Zarsiki-closure $\cal$ of $\bbh$ in $(\mathbb{GL}_n)_{\bbz[1/q_1]}$ is defined by $F'$.
\item The projection map $\cal\rightarrow \spec(\bbz[1/q_1])$ is smooth.
\item We can compute a generating set for $\cal(\bbz[1/q_1])$. 
\item $\pi_p(\cal(\bbz[1/q_1]))$ is a perfect group if $p\nmid q_1$.
\end{enumerate}
\end{lem}
\begin{proof}
By \cite[Algorithm 3.5.3]{GS1}, we can compute the unipotent radical and a Levi subgroup of $\bbh$. Therefore we 
can effectively write $\bbh$ as the semidirect product of a semisimple $\bbq$-group $\bbl$ and a 
unipotent $\bbq$-group $\bbu$. We can compute a square-free integer $q_2$ and $\bbz[1/q_2]$-group schemes 
$\mathcal{L}$ and $\mathcal{U}$ such that:
\begin{enumerate}
\item The projection maps to $\spec(\bbz[1/q_2])$ are smooth.
\item All the fibers are geometrically irreducible. 
\item $\mathcal{L}$ acts on $\mathcal{U}$.
\item The generic fiber of $\mathcal{L}$ (resp. $\mathcal{U}, \cal:=\mathcal{L}\ltimes \mathcal{U}$) is isomorphic to $\bbl$ (resp. $\bbu, \bbh$).
\end{enumerate}
It is worth mentioning that the first and the second items are consequences of 
Theorem~\ref{t:EGA} and the rest are easy.  Using logarithmic and exponential maps, we can  
effectively enlarge $q_2$, if necessary, and assume that for any $p\nmid q_2$ we have 
$[\ucal_p(\bbf_p),\ucal_p(\bbf_p)]=[\ucal,\ucal]_p(\bbf_p)$, where 
$\ucal_p=\bbu\times_{\spec(\bbz[1/q_2])}\spec(\bbf_p)$ ($[\ucal,\ucal]_p$ is defined in a similar 
way). We can also get a generating set for $\ucal(\bbz[1/q_2])$. By Lemma~\ref{l:GenericHS}, we can enlarge $q_2$ and assume that $\lcal(\bbz_p)$ is a 
hyper-special parahoric subgroup of $\bbl(\bbq_p)$ for any $p\nmid q_2$. In particular, by further 
enlarging $q_2$, we have that $\lcal(\bbz[1/q_2])$ is an arithmetic lattice in (the non-compact 
semisimple group) $\bbl(\bbr)\cdot \prod_{p|q_2} \bbl(\bbq_p)$. Thus we have
\begin{enumerate}
\item By the classical strong approximation theorem, we have that 
\[
\pi_p(\lcal(\bbz[1/q_2]))=\lcal_p(\bbf_p)
\]
 is a product of quasi-simple groups. 
\item By \cite{GS2}, we can compute a generating set $\Omega$ for $\lcal(\bbz[1/q_2])$. Thus we get a generating set for $\cal(\bbz[1/q_2])$.
\end{enumerate}

On the other hand, since $\bbh$ is perfect, the action of $\bbl$ on  $\ufr/[\ufr,\ufr]$ has no non-trivial fixed vector, where $\ufr=\Lie(\bbu)$. This is equivalent to say that the elements of $\Omega$ do 
not have a common non-zero fixed vector. Fix a basis $\mathfrak{B}$ of $\ufr/[\ufr,\ufr]$ and let $X_{\gamma}:=[\gamma]_{\mathfrak{B}}-I$, where the 
$[\gamma]_{\mathfrak{B}}$ is the matrix associated with the action of $\gamma$ and $I$ is the 
identity matrix. Let $X$ be a column blocked-matrix whose blocked-entries are $X_{\gamma}$ for 
$\gamma\in \Omega$. By our assumption, $X$ is of full rank, i.e. the product of its minors of maximum 
dimension is a non-zero element of $\bbz[1/q_2]$. Hence by enlarging $q_2$, if necessary, we can 
assume that the elements of $\pi_p(\Omega)$ do not fix any non-trivial element of 
$\ucal_p(\bbf_p)/[\ucal_p(\bbf_p),\ucal_p(\bbf_p)]$. Hence by the above discussion, for any prime $p\nmid q_2$, we have
\[
\pi_p(\cal(\bbz[1/q_2]))=\lcal_p(\bbf_p)\ltimes \ucal_p(\bbf_p)
\]
is a perfect group, which finishes our proof.
\end{proof}
\begin{proof}[Proof of Theorem~\ref{t:5.1}.]
Let $q_1$ be a square-free integer given by Lemma~\ref{l:GenericPerfectness}. Let $\gcal$ be the
Zariski-closure of $\bbg$ in $(\mathbb{GL}_n)_{\bbz[1/q_1]}$. Lemma~\ref{l:GenericPerfectness} 
provides us with an effective version of Theorem A for the group $\gcal(\bbz[1/q_1])$. On the other 
hand, we have already proved the effective versions of \cite[Theorem B and C]{Nor}. Hence 
following the proof of Proposition~\ref{p:LiftingUp}, one can effectively compute a positive number 
$\delta$ such that: for any proper subgroup $H=H^+$ of $\gcal_p(\bbf_p)^+$, one has that 
\[
\{\gamma\in \gcal(\bbz[1/q_1])|\h \|\gamma\|_S\le [\pi_p(\gcal(\bbz[1/q_1])):H]^{\delta}\}
\]
is in a proper algebraic subgroup. In particular, if $\Omega$ generates a Zariski-dense subgroup of 
$\bbg$, then $\pi_p(\Gamma)^+=\gcal_p(\bbf_p)^+$ for any $p>\max_{\gamma\in \Omega}\{\|
\gamma\|_S^{1/\delta}\}$, where $S=\{p|\h p \text{ is a prime divisor of}\h q_1\}.$ 
\end{proof}


 \bigskip
A. Salehi Golsefidy

{\sc Department of Mathematics, University of California, San Diego,
CA 92122, USA}

{\em e-mail address:} asalehigolsefidy@ucsd.edu

\bigskip
P. P. Varj\'u

{\sc Department of Mathematics, Princeton University, Princeton,
NJ 08544, USA and

Analysis and Stochastics Research Group of the
Hungarian Academy of Sciences, University of Szeged, Szeged, Hungary}

{\em e-mail address:} ppvarju@gmail.com

\end{document}